\documentclass{article}
\usepackage{tikz}
\usetikzlibrary{shapes,backgrounds}
\usetikzlibrary{positioning} %using nodes in tikz
\usepackage{graphicx} % support the \includegraphics command and optionsd
\usepackage{color}
\usepackage{comment}
\usepackage{booktabs} % for much better looking tables
\usepackage{array} % for better arrays (eg matrices) in maths
\usepackage{paralist} % very flexible & customisable lists (eg. enumerate/itemize, etc.)
\usepackage{verbatim} % adds environment for commenting out blocks of text & for better verbatim
\usepackage{subcaption} % make it possible to include more than one captioned figure/table in a single float
\usepackage{mathdots} % better dots, very nice
\usepackage[psdextra]{hyperref}
\usepackage{amsmath}
\usepackage{amsfonts}
\usepackage{amssymb}
\usepackage{amsthm}
\usepackage{tikz-cd}
\usepackage{epigraph}
\usepackage{mathtools}
\newtheorem{theorem}{Theorem}[section]
\newtheorem{lemma}[theorem]{Lemma}
\newtheorem{proposition}[theorem]{Proposition}
\newtheorem{corollary}[theorem]{Corollary}
\newtheorem{definition}[theorem]{Definition}

\newtheorem{remark}[theorem]{Remark}

\newtheorem{assumptions}[theorem]{Assumptions}
\usepackage{setspace}
\usepackage[normalem]{ulem}
\usepackage[bb=libus]{mathalpha}
\usepackage{todonotes}

\usepackage[left=3cm, right=3cm, top=2cm]{geometry}

\newcommand{\x}[1]{{\color{red} #1}}

\usepackage[most]{tcolorbox}
\usetikzlibrary{patterns}
\pgfdeclarepatternformonly{mystrikeout}{\pgfqpoint{-1pt}{-1pt}}{\pgfqpoint{11pt}{11pt}}{\pgfqpoint{10pt}{10pt}}%
{
  \pgfsetlinewidth{0.4pt}
  \pgfpathmoveto{\pgfqpoint{0pt}{0pt}}
  \pgfpathlineto{\pgfqpoint{10.1pt}{10.1pt}}
  \pgfusepath{stroke}
}
\newtcolorbox{soutbox}{breakable,
 enhanced jigsaw,
 opacityback=0,
 parbox=false,
 boxrule=0mm,
 top=0mm,bottom=0pt,left=0pt,right=0pt,
 boxsep=0pt,
 frame hidden,
 finish={\fill[pattern=mystrikeout, pattern color=blue] (frame.north west) rectangle (frame.south east);}
}

\title{Uniform Sobolev inequalities on geometric graphs}

\author{Samuel Mercer, Yves van Gennip}

\begin{document}

\maketitle

\begin{abstract}
There is significant interest in the study of calculus on graphs, especially regarding the use of gradient-based methods for applications in data driven problems such as classification, clustering and regularisation for inverse problems. Geometric graphs, whose vertices are take from from a Euclidean domain and whose edge structure is determined by the distance between the nodes in the domain, have been central in theoretical studies. Typical approaches for analysis, such as studying consistency and the existence of continuum limits, rely on $\Gamma$-convergence. This technique has some limitations, as it requires the typical length scale which determines the connectivity structure of the graph to be much larger than the scales frequently used for applications. Moreover, it may fail to provide quantitative results. This paper provides necessary and sufficient conditions on the asymptotic behaviour of this length scale for the existence of a uniform collection of Sobolev inequalities on a sequence of geometric graphs. Furthermore, these inequalities hold when the length scales are much smaller than what is typically assumed for $\Gamma$-convergence results and within the range of what is used for data-driven problems. The Sobolev inequalities provide a quantitative estimate on the $L^q$-regularisation effect of discrete gradients.
\end{abstract}

\section{Introduction}

We establish Sobolev inequalities of the form \eqref{eq:graphsobolev} and \eqref{eq:graphsobolev2} that control the $L^q$ and $L^\infty$ norms, respectively, of graph-based functions in terms of lower-order $L^p$ norms and discrete variations. The two main results in this vein in Theorem~\ref{thm:graphsobolev} (for $p<d$) and Theorem~\ref{thm:graphsobolev2} (for $p>d$) are obtained for sequences of geometric graphs take from from $\mathbb{R}^d$. The geometric graph setting is given in full detail in Section~\ref{sec:geometry}; here we give a brief overview. 

Consider an open subset $\Omega\subset \mathbb{R}^d$ (for any fixed $d\in\mathbb{N}$) with Lipschitz boundary and equipped with a probability measure $\mu$, and a sequence of finite sets $V_n \subset \Omega$ indexed by $n\in \mathbb{N}$ and equipped with discrete probability measures $\mu_n$, satisfying $|V_n|\rightarrow \infty$ as $n\rightarrow\infty$. One should think of $(V_n;\mu_n)$ suitably approximating $(\Omega,\mu)$, in particular $(\mu_n)_{n\in\mathbb{N}}$ converges to $\mu$ as $n\rightarrow\infty$ in the $\infty$-Wasserstein metric space. We denote the distance in this space by $d_\infty(\mu_n,\mu)$.
Additionally, let $(\varepsilon_n)_{n\in \mathbb{N}}$ be a sequence of positive reals converging to zero and $\eta:[0,\infty)\rightarrow [0,\infty)$ be a non-negative, non-increasing function, continuous at zero and $\eta(0)>0$. We refer to $(\varepsilon_n)_{n\in \mathbb{N}}$ as the concentrating parameter sequence. For a function $u:V_n\rightarrow \mathbb{R}$ we define
$$ \mathcal{GE}_n^p(u):= \frac{1}{\varepsilon_n^{d+p}}\sum_{x\in V_n}\sum_{y\in V_n} \eta\left(\frac{|x-y|}{\varepsilon_n}\right)|u(x)-u(y)|^p\mu_n(x)\mu_n(y).$$
The primary question in this paper is the following. Let $q>p\geq 1$. Under what conditions on $q,p$ and the concentrating parameter sequence $(\varepsilon_n)_{n\in \mathbb{N}}$ does there exist a constant $C>0$ such that, for any $n\in\mathbb{N}$ and any function $u:V_n\rightarrow\mathbb{R}$,
\begin{equation}\label{eq:question}  \left( \sum_{x\in V_n} |u(x)|^q\mu_n(x)\right)^{1/q} \leq C\left(\left( \sum_{x\in V_n} |u(x)|^p \mu_n(x)  \right)^{1/p} + \left( \mathcal{GE}_n^p(u) \right)^{1/p} \right)?
\end{equation}

The above Sobolev inequality gives a quantitative estimate on how regularisation occurs. For $q>p$ there are spiky functions that have a much larger norm in $L^q(V_n;\mu_n)$ than in $L^p(V_n;\mu_n)$. If a gradient based method is applied to these spiky functions in a way that decreases the functional $\mathcal{GE}_n^p$ then their $L^q$ norm must decrease as well, creating a $L^q$-regularisation effect. Spiky functions may occur in solutions to data-driven optimization problems due to over-fitting. 

If we allow the constant $C$ to depend on $n$ then the above inequality is trivial, since the $L^q$ and $L^p$ norms are Lipschitz equivalent on the finite set $V_n$. For example, if each $\mu_n$ is the uniform probability measure on $V_n$, then, for any $n\in\mathbb{N}$ and function $u:V_n \rightarrow \mathbb{R}$,
$$ \left( \frac{1}{|V_n|}\sum_{x\in V_n} |u(x)|^q\right)^{1/q} \leq |V_n|^{\frac{1}{p}-\frac{1}{q}}\left( \frac{1}{|V_n|}\sum_{x\in V_n} |u(x)|^p   \right)^{1/p}.$$
This follows from the inequality $\sum_{x\in V_n} |\tilde u(x)|^p \geq \sum_{x\in V_n} |\tilde u(x)|^q$ for $\tilde u := |u|/\left(\sum_{y\in V_n} |u(y)|^q\right)^{1/q}$. 
The bound above is in fact sharp: equality is achieved for a function $u$ which is zero on all vertices but one, where it has value $1$.  However, $ |V_n|^{\frac{1}{p}-\frac{1}{q}}\rightarrow\infty$ as $n\rightarrow\infty$, making the bound not very useful for large $n$. That is why our objective is to provide a uniform collection of Sobolev inequalities of the form above in which $C$ is independent of $n$. 

An important quantity related to the distance $d_\infty(\mu_n,\mu)$ is the connectivity threshold of the set $V_n$, which we denote by $s_n$ and is constructed as follows (see also Definition~\ref{def:connthreshold}). For $s>0$ and $n\in\mathbb{N}$, we define $\mathbb{G}_{s}(V_n):=(V_n,E_{n,s})$ to be the simple graph with vertex set $V_n$ and edge set given by the pairs of vertices that are within a distance $s$. The connectivity threshold of $V_n$ is the smallest real number $s>0$, such that $\mathbb{G}_s(V_n)$ is connected. In Corollary~\ref{cor:connect} we prove $s_n\leq 2\sqrt{d} \; d_\infty(\mu_n,\mu)$, under suitable assumptions including convexity of $\Omega$. If instead $\Omega$ is bi-Lipschitz homeomorphic to a convex set, then Corollary~\ref{cor:connect2} shows that there exists a constant $C(\Omega)$, depending on $\Omega$, such that $s_n\leq C(\Omega)d_\infty(\mu_n,\mu)$. For many concrete examples one can also produce a Lipschitz lower bound for $s_n$ in terms of $d_\infty(\mu_n,\mu)$, suggesting a very close relationship. Although, for now, establishing a general lower bound remains an open problem. 

We refer to the graphs $\mathbb{G}_{s}(V_n)$ as geometric graphs and they will be a central object of study within this paper. This is because, in the case that $\eta(r)$ is given by an indicator function $\mathbf{1}_{r\leq 1}$ and thus defines a geometric graph structure on $V_n$, the double sum in the definition of $\mathcal{GE}_{n}^p$ can be expressed as a double sum over the edge set of the graph $\mathbb{G}_{\varepsilon_n}(V_n)$. In general, we assume that $\eta$ is positive and continuous at zero and non-negative elsewhere, in which case we can choose constants $a,b>0$ such that $\eta(r)\geq a\cdot \mathbf{1}_{r\geq 1}$ (see Remark~\ref{rem:simpletokernel}). Thus we can bound the double sum in the definition of $\mathcal{GE}_n^p$ below by a double sum over the edges of the graph $\mathbb{G}_{b\varepsilon_n}(V_n)$. This tells us that we can reduce problem~\eqref{eq:question} for this more general class of kernels to the case where the kernel is an indicator function: if \eqref{eq:question} is satisfied in this latter case, it also holds for those more general kernels $\eta$.

In the literature many theoretical studies of graph-based discrete variation functionals like $\mathcal{GE}_n^p$ assume that the concentrating parameter $\varepsilon_n$ is much larger than $d_\infty(\mu_n,\mu)$, and thus, also much larger than $s_n$ (for example \cite{garcia2016continuum,Dejan2019}). The assumption is necessary (see the counterexamples in Appendix~\ref{app:counterexamples}) if we  want to establish an isotropic continuum limit of $\mathcal{GE}_n^p$ via $\Gamma$-convergence. However, numerical experiments that use geometric graphs for classification purposes, reveal that the best results are typically achieved when $\varepsilon_n$ is chosen as close to the connectivity threshold as possible (see \cite[Figure 4 and Figure 9]{Dejan2019}). The advantage of the question we have posed in \eqref{eq:question} is that a priori no continuum limit needs to be observed. The inequality is phrased entirely within the discrete structure.

In Theorem~\ref{thm:graphsobolev} we prove, under suitable assumptions, that a uniform collection of Sobolev inequalities for $1\leq p<d$ and $p\leq q\leq \frac{pd}{d-p}$ hold if and only if the sequence $\left( \varepsilon_n |V_n|^{\frac{1}{p}-\frac{1}{q}} \right)_{n\in\mathbb{N}}$ is bounded. One of our assumptions, specifically part~(i) of Assumptions~\ref{assum}, requires the existence of a $K>0$ such that $K \varepsilon_n \geq d_\infty(\mu_n,\mu)$ for all $n$. This is strictly weaker than part~(i$^*$) of Assumptions~\ref{assum}, which is typically assumed for $\Gamma$-convergence results in the literature. An assumption such as (i) is to be expected, since the inequality in \eqref{eq:question} will fail if the graph is disconnected. For example, suppose $\eta$ is the indicator function for the interval $[0,1]$. If $\varepsilon_n$ is chosen to be smaller than $\inf_{\substack{x, y\in V_n\\x\neq y}}|x-y|$, then for any function $u:V_n\rightarrow \mathbb{R}$, we have $\mathcal{GE}_n^p(u)=0$. In that case a uniform collection of Sobolev inequalities cannot hold.

Besides this we have no additional requirements on the sequence $(\varepsilon_n)_{n\in\mathbb{N}}$. If $q=\frac{pd}{d-p}$, then $\frac1p-\frac1q=\frac1d$, thus a corollary of Theorem~\ref{thm:graphsobolev} is that, to achieve the largest exponent $q$ for which our uniform collection of Sobolev inequalities holds (for $p<d$), $\left( \varepsilon_n |V_n|^{1/d} \right)_{n\in\mathbb{N}}$ must be bounded. Together with some of our combinatorial results in Section~\ref{sec:geometry}, this bound leads to the following observations.
\begin{itemize}
    \item Assuming that $\left( \varepsilon_n |V_n|^{1/d} \right)_{n\in\mathbb{N}}$ is bounded, there exists a constant $C_1>0$ such that $\varepsilon_n\leq C|V_n|^{-1/d}$. 
    \item According to Lemma~\ref{lem:transportbound}, part~(i) from Assumptions~\ref{assum} implies there exists constants $K,C_2>0$ such that, for any $n\in\mathbb{N}$, $$\varepsilon_n\geq K d_{\infty}(\mu_n,\mu)\geq C_2|V_n|^{-1/d}.$$
    \item According to Corollary~\ref{cor:connect}, for any $n\in\mathbb{N}$, $s_n\leq 2\sqrt{d} \; d_{\infty}(\mu_n,\mu_\infty).$   
\end{itemize}
Thus $\varepsilon_n$ and $d_\infty(\mu_n,\mu)$ must both be of the order $|V_n|^{-1/d}$ for the Sobolev inequalities to hold. Additionally $s_n$ must at most be of the order $|V_n|^{-1/d}$. If a lower bound on $s_n$ could be proven as we previously conjectured, then we could conclude $s_n$ must also be of the order $|V_n|^{-1/d}$.

The strategy we shall use to prove Theorem~\ref{thm:graphsobolev} is composed of three parts.

\begin{enumerate}
    \item For a given $n\in \mathbb{N}$ and $u:V_n\rightarrow\mathbb{R}$ we lift $u$ to a function $\widetilde{u}$ with domain $\Omega$. To choose such a lifting we will use some notions from optimal transport. We then construct a function $v:\Omega\rightarrow \mathbb{R}$ that is a mollification of $\widetilde{u}$. The next step is to bound $\Vert v \Vert_{W^{1,p}(\Omega)}$ in terms of $\Vert u \Vert_{L^p(V_n;\mu_n)}$ and $\mathcal{GE}_n^p(u)$. Then, via the classical Sobolev embedding theorem, we can produce a bound on $\Vert v \Vert_{L^q(\Omega)}$.

    The steps described here are a standard application of arguments in the literature regarding the analysis of nonlocal functionals. We believe the first use of the mollification method described above appeared in \cite[Theorem 3.1]{Alberti1998}.

    \item Next, our goal is to bound $\Vert v - \widetilde{u}\Vert_{L^q(\Omega)}$ in terms of $\Vert u \Vert_{L^p(V_n;\mu_n)}$ and $\mathcal{GE}_n^p(u)$. This part of the proof requires some combinatorial properties of the geometric graph and here is where we provide new techniques for the literature. In particular, we introduce the property of $k$-friendly graphs and Lemma~\ref{lem:graphfriendly} which help us prove the required inequalities. 

    \item The final part is to introduce a tool that helps us remedy some issues arising in parts 1 and 2 that are caused by the boundary of $\Omega$. In order to construct a mollification we take averages in a ball of a suitable radius around each point. However, near the boundary this ball may intersect the complement of $\Omega$ where the function $\widetilde{u}$ is undefined. Moreover, some of the local combinatorial bounds we prove for geometric graphs are only valid at vertices that are a fixed distance away from the boundary. In order to overcome this quagmire we introduce an extension operator in Lemma~\ref{lem:ext1} to extend the relevant functions to a larger domain and a discretized extension result in Lemma~\ref{lem:ext2}.  
\end{enumerate}

In Theorem~\ref{thm:graphsobolev2} we also consider the case when $p>d$. The work we do here is, in large part, not new. The argument we prove follows the same strategy provided in \cite[Lemma 4.1 and Lemma 4.5]{Dejan2019}. Though \cite[Lemma 4.5]{Dejan2019} is phrased as a compactness result, the argument used is almost sufficient for the Sobolev inequality we are looking for. However, there are still two key limitations, which the theory we develop in this paper will help us overcome. First of all, \cite[Lemma 4.5]{Dejan2019} requires the concentrating parameter to be much larger than the connectivity threshold; this we do not assume. The second, is that the result only holds over a proper compact subset of the domain, not the full domain. The tools we need to overcome this latter problem are precisely the extension results previously discussed. 

There are other Sobolev-style inequalities in the graph-theory literature. Most readily available are Sobolev inequalities on lattices, as the arguments used to prove classical Sobolev inequalities easily generalise to the lattice setting. One such example is \cite[Theorem 4.2]{Fusco2025}. In a more general graph setting, there exists the Sobolev inequalities proved by Fan Chung in \cite[Theorem 11.1 and Theorem 11.4]{Chung1997}. The drawback of these latter results is that they are both dependent on the isoperimetric dimension of the graph. For general graphs, including the geometric graphs studied in this paper, the isoperimetric dimension of a graph is very hard to compute.   

We shall also produce a uniform collection of Poincar\'e inequalities in Theorem~\ref{thm:poincare}. Poincar\'e inequalities on graphs have a broad number of applications. Moreover, in the $p=2$ case, there is a close relationship between Poincar\'e inequalities and the eigenvalues of the graph Laplacian (see \cite[Chapter 1]{Chung1997}). The eigenvalues of the graph Laplacian in the geometric graph setting have been rigorously investigated in \cite{Garcia2020}. In fact, in the $p=2$ case a Poincar\'e inequality can be derived from \cite[Corollary 1]{Garcia2020} via the Rayleigh quotient. Furthermore, a very interesting multi-scale Poincar\'e inequality has recently appeared in~\cite[Proposition 3.6]{trillos2025}. This result provides one with an even more precise upper bound than what can be achieved with a typical $p=2$ Poincar\'e inequality.

We will now break down the different sections and each of the key results contained within them.

\begin{itemize}
    \item In Section~\ref{sec:opt} we introduce the concepts and notation we shall use from measure theory and optimal transport. We shall also define the $TL^p$ metric which was first introduced in \cite{garcia2016continuum}. Later on we shall use this metric space to formulate some compactness results.
    \item In Section~\ref{sec:graph} we introduce the concepts and notation we shall use from graph theory. Here we introduce two important graph properties for proving Sobolev inequalities. The first is the $k$-friendly property in Definition~\ref{def:kfriendly1}. The result for $k$-friendly graphs in Lemma~\ref{lem:friendlybound}  will help us prove Sobolev inequalities for $p<d$. The second is the envelop property in Definition~\ref{def:envelopes}. The result provided in Lemma~\ref{lem:envelopebound} is a reformulation of a result from \cite[Lemma 4.1]{Dejan2019}, which will help us prove Sobolev inequalities for $p>d$. 
    \item In Section~\ref{sec:geometry} we introduce the geometric graph setting which shall be the central object of study in this paper.  Here we also prove the combinatorial results that shall be useful later for deriving our Sobolev inequalities and their corollaries. In Lemma~\ref{lem:transportbound} we prove a lower bound on the $\infty$-Wasserstein distance between a discrete and a continuum probability measure in terms of the size of the support of the discrete probability measure. In Proposition~\ref{prop:connect} we prove a lower bound on the connectivity threshold of a finite set in terms of the $\infty$-Wasserstein distance of any discrete probability measure with support in that set and any continuum probability measure. In Lemma~\ref{lem:graphfriendly} we prove that geometric graphs are $k$-friendly, for a specific, graph-dependent, value of $k$. In Lemma~\ref{lem:degreebound} we prove a lower bound on the vertex degrees in a geometric graph.
    \item Section~\ref{sec:sobolev} separates into four subsections.

    \begin{itemize}
        \item In Section~\ref{sec:extensionresults} we prove two extension results. One of our extension results, Lemma~\ref{lem:ext1}, extends a function on an open domain with Lipschitz boundary to a larger domain in a way such that a collection of its nonlocal variations on the larger domain is uniformly controlled by the nonlocal variation on its original domain. 
        
        The second result, Lemma~\ref{lem:ext2}, can be thought of as a discretization of the previous result. That is, we wish to extend a geometric graph to a larger geometric graph on a larger domain with the same approximation properties. We then wish to extend discrete functions to the larger graph in a way that $\mathcal{GE}_n^p$ is controlled in a uniform way independent of $n\in\mathbb{N}$. We can compare Lemma~\ref{lem:ext1} to the extension constructed in the proof of \cite[Lemma 4.4]{garcia2016continuum}. The construction in that proof is given by an explicit formula, whereas we use a partition of unity. Moreover, that proof extends the functions' domain to all of $\mathbb{R}^d$ whereas we only construct an extension onto a bounded domain $\Sigma$. The construction in the proof of \cite[Lemma 4.4]{garcia2016continuum} works only in the case where $\Omega$ has a $C^2$ boundary, ours applies to Lipschitz boundaries.
        \item In Section~\ref{sec:unifPoincare} we prove our first compactness result in Proposition~\ref{prop:compactness} and a uniform collection of Poincar\'e inequalities in Theorem~\ref{thm:poincare}. The compactness result Proposition~\ref{prop:compactness} is largely the same as that in \cite[Corollary 4.7]{garcia2016continuum} or \cite[Proposition 4.4]{Dejan2019} and follows via a similar argument. The main difference is that we assume part~(i) of Assumptions~\ref{assum}, rather than part~(i$^*$).
        \item In Section~\ref{sec:Sobolevplessd}, specifically in Theorem~\ref{thm:graphsobolev}, we prove our Sobolev inequalities when $p<d$. We also produce Corollaries~\ref{cor:compactness1} and~\ref{cor:compactness2} and Proposition~\ref{prop:compactness3} that improve upon Proposition~\ref{prop:compactness} and currently known compactness results in the literature (\cite[Corollary 4.7]{garcia2016continuum} or \cite[Proposition 4.4]{Dejan2019}). Specifically, we improve the integrability exponent of the space in which compactness can be observed. 
        \item In Section~\ref{sec:Sobolevpgreaterd}, specifically in Theorem~\ref{thm:graphsobolev2}, we establish our Sobolev inequalities when $p>d$. Additionally we produce compactness results in Corollary~\ref{cor:compactness4} and Proposition~\ref{prop:compactness5}. Proposition~\ref{prop:compactness5} is an improvement upon \cite[Lemma 4.5]{Dejan2019} and follows  by a similar argument. As we previously mentioned, also Theorem~\ref{thm:graphsobolev2} improves upon (a Sobolev inequality that can be derived from) \cite[Lemma 4.5]{Dejan2019} by allowing the concentrating parameter to be of the same size as the connectivity threshold and by establishing the result over the full domain. \end{itemize}

    \item In Section~\ref{sec:discussion} we discuss a number of consequences of our results and also possible directions for future work. We particularly focus on Theorem~\ref{thm:graphsobolev} as it is our main contribution in this paper. Possible topics for future research that we discuss are H\"older regularity on graphs and higher-order variations on graphs.

\end{itemize}

\section{Optimal transport and the \texorpdfstring{$TL^p$}{tlp} metric}
\label{sec:opt}

We shall denote a measure space by a triple $(\mathcal{M},\mathfrak{F},\mu)$  where $\mathcal{M}$ is a set, $\mathfrak{F}$ is a $\sigma$-algebra on $\mathcal{M}$ and $\mu$ is a measure on $\mathcal{M}$. In the case $\mu(\mathcal{M})=1$ we refer to the triple as a probability space. As usual we define $L^p(\mathcal{M};\mu)$ as the space of $p$-th Lebesgue integrable functions on $\mathcal{M}$. 

For an open subset $\Omega \subset \mathbb{R}^d$ we will always take the $\sigma$-algebra to be the collection of all Borel sets, moreover we denote the Lebesgue measure on $\Omega$ by $\mathfrak{L}^d|_{\Omega}$. In the case that $\Omega = \mathbb{R}^d$ we simply write $\mathfrak{L}^d$. We refer to a map $T:\Omega\rightarrow \Omega$ which is measurable with respect to the Borel $\sigma$-algebra on both domain and codomain as a Borel map. We shall denote the identity map on $\Omega$ by $I$ and if necessary we write $I|_{\Omega}$. 

\begin{definition}
For a measure space $(\mathcal{M},\mathfrak{F},\mu)$ and a function $f: \mathcal{M}\rightarrow \mathbb{R}$ the essential supremum of $f$ is defined as
$$ \underset{(\mathcal{M,\mu})}{\operatorname{esssup}} f:= \sup \{ c\in \mathbb{R} \; | \; \mu(\{ x\in \mathcal{M} \; | \; f(x)\leq c \})>0 \}.$$
In the case that $\mathcal{M}$ is a Borel subset of $\mathbb{R}^d$ and $\mu$ is the Lebesgue measure on $\mathcal{M}$ we will simply write $\underset{\mathcal{M}}{\operatorname{esssup}} f:= \underset{(\mathcal{M},\mathfrak{L}^d|_{\mathcal{M}}) }{\operatorname{esssup}} f$.
\end{definition}

For a probability space $(\mathcal{M},\mathfrak{F},\mu)$ and an integrable random variable $X:\mathcal{M}\rightarrow\mathbb{R}$ we write $\mathbb{E}_{\mu}[X]$ for the expectation of $X$. Next we shall define the push-forward of a measure, a tool which we shall frequently use.

\begin{definition}
Let $(\mathcal{M},\mathfrak{F},\mu)$
be a measure space and $\mathcal{N}$ a set with a $\sigma$-algebra $\mathfrak{G}$. Suppose $T:\mathcal{M}\rightarrow \mathcal{N}$ is measurable. We define $T\#\mu$ to be a measure on the pair $(\mathcal{N},\mathfrak{G})$ as follows: for $A \in \mathfrak{G}$,
$$ (T\#\mu)(A)= \mu \left( T^{-1}(A) \right).$$
\end{definition}

Henceforth in this section we shall fix a subset $\Omega\subset \mathbb{R}^d$. Let $p\geq 1$. We define $\mathcal{P}_p(\Omega)$ to be the set of all Borel probability measures on $\Omega$ with finite $p$-th moment. This means $\mu\in \mathcal{P}_p(\Omega) $ if and only if $\mu$ is a probability measure on $\Omega$ and $\int_{\Omega}|x|^p \, d\mu(x)<+\infty$. 

We define $\mathcal{P}_\infty(\Omega)$ to be the set of all probability measures on $\Omega$ with essentially bounded support. In  other words $\mu\in \mathcal{P}_\infty(\Omega) $ if and only if $\mu$ is a probability measure on $\Omega$ and  $\underset{(\Omega,\mu)}{\operatorname{esssup}} \{x\xmapsto{}|x|\}<+\infty$. 

\begin{definition}
Consider two different Borel probability measures $\mu,\nu$ on the set $\Omega$. For $i\in\{1,2\}$ define $\pi_i : \Omega \times \Omega \rightarrow \Omega$ as the projection maps onto the $i$-th co-ordinate.

The set of couplings $\Gamma(\mu,\nu)$ is defined to be the set of probability measures on $\Omega\times \Omega$ such that $\gamma\in \Gamma(\mu,\nu)$ if and only if $\pi_1 \# \gamma =\mu$ and $\pi_2 \# \gamma=\nu$. 
\end{definition}

Using the notion of couplings we can now define the $p$-th Wasserstein distance.

\begin{definition}
    For $p\geq 1$, define a distance $d_p$ on $\mathcal{P}_p(\Omega)$ as follows. For $\mu,\nu \in \mathcal{P}_p(\Omega)$,
    $$ d_p(\mu,\nu):= \inf_{\gamma \in \Gamma(\mu,\nu)} \left( \int_{\Omega\times \Omega} |x-y|^p \, d\gamma(x,y) \right)^{1/p}.$$ 
    
    The metric $d_\infty$ is defined on $\mathcal{P}_\infty(\Omega)$ as follows. For $\mu,\nu \in \mathcal{P}_\infty(\Omega)$,
    $$ d_\infty(\mu,\nu) := \inf_{\gamma \in \Gamma(\mu,\nu)} \underset{(\Omega\times\Omega,\gamma)}{\operatorname{esssup}} \left\{ (x,y) \xmapsto{} |x-y| \right\} .$$ 

   We refer the reader to the discussion following \cite[Definition 6.1]{Villani09} to see a proof that $(\mathcal{P}_p(\Omega),d_p)$ indeed defines a metric space for $1\leq p \leq+\infty$.
\end{definition}

Suppose we have two Borel probability measures $\mu,\nu \in \mathcal{P}_p(\Omega)$ and also a Borel map $T:\Omega\rightarrow \Omega$ such that $\nu=T\#\mu$; such a map is called a (Borel) transport map. Define the map $I\times T: \Omega \rightarrow \Omega\times\Omega$ by $(I\times T)(x):=(x,T(x))$ and a coupling $\gamma := (I\times T) \#\mu \in \Gamma(\mu,\nu)$. We then see that, for $p<+\infty$,
$$ d_p(\mu,\nu) \leq \left( \int_{\Omega\times \Omega} |x-y|^p \, d\gamma(x,y) \right)^{1/p} = \left( \int_{\Omega} |x-T(x)|^p \, d\mu(x) \right)^{1/p}.$$
In the case that $p=\infty$ we have
\begin{equation}
\label{eq:borelbound}
    d_\infty(\mu,\nu) \leq \underset{(\Omega,\mu)}{\operatorname{esssup}} \left\{ x \xmapsto{} |x-T(x)| \right\} .
\end{equation}

Given probability measures $\mu,\nu \in \mathcal{P}_p(\Omega)$, it is possible that there is no Borel map $T:\Omega \rightarrow \Omega$ such that $T\#\mu=\nu$. For example, if $\mu$ is a Dirac mass then, for any map $T$, $T\#\mu$ is also a Dirac mass. In the following theorem we provide sufficient conditions for the existence of a suitable Borel map and also for the bound in~\eqref{eq:borelbound} to become a strict equality.

\begin{theorem}
\label{thm:transportexistence}
Assume $\Omega\subset \mathbb{R}^d$ is open and bounded. Let $\mu,\nu$ be probability measures on $\Omega$ and assume $\mu$ is absolutely continuous with respect to the Lebesgue measure. Then there exists a Borel map $T:\Omega\rightarrow \Omega$ such that $T\#\mu=\nu$ and
$$ d_\infty(\mu,\nu) = \sup_{x\in\Omega} \{ |x-T(x)| \}.$$
\end{theorem}

\begin{proof}
Since $\mu$ and $\nu$ can be extended trivially to probability measures on the compact set $\overline{\Omega}$, we can use \cite[Theorem 3.2]{Champion2008} and the arguments preceding it in \cite[Section 3]{Champion2008} to establish the existence of a transport plan $\gamma\in \Gamma(\mu,\nu)$ which achieves the infimum in the definition of $d_\infty(\mu,\nu)$ (all defined for $\overline{\Omega}$ instead of $\Omega$) and is infinitely cyclically monotone \cite[Definition 3.1]{Champion2008}. Then \cite[Theorem 5.5]{Champion2008} establishes the existence of a Borel transport map $\widetilde{T}: \overline\Omega \to \overline\Omega$ such that $\gamma=(\operatorname{id}\times \widetilde{T})\# \mu$. Since for our extended measures we have  $\mu(\overline\Omega\setminus\Omega)=\nu(\overline\Omega\setminus\Omega)=0$, we can identify $\widetilde{T}|_{\Omega}$ with a map whose range is $\Omega$ such that $\gamma=(\operatorname{id}\times \widetilde{T}|_{\Omega})\# \mu$. We can substitute this $\gamma$ into the definition of $d_\infty(\mu,\nu)$ to get
$$ d_\infty(\mu,\nu) = \underset{(\Omega,\mu)}{\operatorname{esssup}} \left\{ x \xmapsto{} |x-\widetilde{T}(x)| \right\}.$$
To conclude, let $N:=\left\{ x\in \Omega \;| \; |x-\widetilde{T}(x)|>\underset{(\Omega,\mu)}{\operatorname{esssup}} \left\{ x \xmapsto{} |x-\widetilde{T}(x)| \right\}\right\}$, and observe that $N$ is a $\mu$-null set. Then define $T:\Omega\rightarrow\Omega$ such that $T(x)=x$ for $x\in N$ and $T(x)=\widetilde{T}(x)$ for $x\in \Omega\setminus N$. 
\end{proof}

Next, we will introduce the $TL^p(\Omega)$ metric space which will feature in our compactness results. For the remainder of this section, assume $\Omega$ is open.

\begin{definition}
The space $TL^p(\Omega)$ is given by the set of pairs,
$$ \{ (f,\mu) \; | \; \mu\in \mathcal{P}_p(\Omega) ,f \in L^p(\Omega;\mu) \},$$
equipped with a metric $d_{TL^p}$ defined by
$$ d_{TL^p(\Omega)} \big( (f,\mu),(g,\nu) \big) :=  \inf_{\gamma\in \Gamma(\mu,\nu)} \left(   \int_{\Omega\times \Omega } |x-y|^p+|f(x)-g(y)|^p \, d\gamma(x,y) \right)^{1/p}.$$

Indeed $d_{TL^p}$ does satisfy the axioms of a metric (although it is not complete) and we refer the reader to \cite{garcia2016continuum} for an overview of the properties of this space.  
\end{definition}

Given a sequence $\big((f_n,\mu_n)\big)_{n\in\mathbb{N}}$ in $TL^p(\Omega)$ and $(f,\mu)\in TL^p(\Omega)$, suppose there exists a sequence of Borel maps $T_n:\Omega \rightarrow \Omega$ such that $T_n\#\mu=\mu_n$ and
\begin{equation}\label{eq:TLpconvcond}
\int_{\Omega} \left(|f_n(T_n(x))-f(x)|^p+|T_n(x)-x|^p\right) \, d\mu(x) \rightarrow 0 \; \; \text{as} \; \; n\rightarrow \infty.
\end{equation}
Then $(f_n,\mu_n)\rightarrow (f,\mu)$ as $n\rightarrow \infty$ in $TL^p(\Omega)$. Indeed, if the functions $T_n\times I:\Omega\times\Omega\rightarrow \Omega$ are defined by $(T_n\times I)(x):=(T(x),x)$, then the couplings $\gamma_n:=(T_n\times I)\#\mu$, for $n\in \mathbb{N}$, provide an upper bound on the infimum in the $TL^p$ metric. By assumption this upper bound converges to zero as $n\to \infty$. Throughout this paper we will only conclude convergence in $TL^p(\Omega)$ via this observation.

For later use we note the following lemma that follows as a straightforward application of H\"older's inequality.

\begin{lemma}
\label{lem:tlpholder}   
Let $p,q$ be constants such that $1\leq p< q$. Let $(\mu_n)_{n\in \mathbb{N}}$ be a sequence of probability measures such that $\mu_n\rightarrow \mu$ in $(\mathcal{P}_q(\Omega),d_q)$  as $n\rightarrow\infty$, for some probability measure $\mu\in \mathcal{P}_q(\Omega)$. Let $(u_n)_{n\in\mathbb{N}}$ be a sequence of functions such that $u_n\in L^p(\Omega,\mu_n)$ and $(u_n,\mu_n)\rightarrow (u,\mu)$ as $n\rightarrow \infty$ in $TL^p(\Omega)$, for some function $u\in L^p(\Omega;\mu)$. Lastly assume that $\{\Vert u_n \Vert_{L^q(\Omega,\mu_n)}\}_{n\in\mathbb{N}}$ is bounded. Then for any $r\in [1,q)$, $(u_n,\mu_n)\rightarrow (u,\mu)$ as $n\rightarrow \infty$ in $TL^r(\Omega)$.    
\end{lemma}

\begin{proof}
See \cite[Lemma 5.18]{mercer2025}.
\end{proof}

The metric space $TL^p(\Omega)$ has the formal structure of a vector bundle; in a sense it is a bundle of all the vector spaces $L^p(\Omega;\mu)$. This description is only formal as each of the vector spaces may have different dimensions for different probability measures $\mu$, hence they are not isomorphic. In recent work by the authors $TL^p(\Omega)$ is shown to be a particular example of a Banach stacking \cite[Proposition 5.17]{mercer2025}, though in this paper we will not need such a description.

\section{Graph theory and notation}
\label{sec:graph}

When we use the subset notation $A \subset B$, we allow for the sets $A$ and $B$ to be equal. Moreover, for a finite set $A$, we shall define $|A|$ to be its cardinality. 

Throughout this paper a simple graph will always be given by a pair $\mathbb{G}=(V,E)$ where $V$ is a finite set, which we call the vertex set, and $E$ is a set of $2$-tuples in $V$, which we call the edge set. For two distinct vertices $x,y$ in $V$ if $\{x,y\}\in E$ we say $x,y$ are adjacent; for shorthand we write $xy:=\{x,y\}$ (in particular, $xy=yx$). For $x\in V$ we denote $N_{\mathbb{G}}(x)$ to be the set of vertices $y\in V$ such that $y$ is adjacent to $x$ in $\mathbb{G}$. Since the graph is simple and thus does not contain self-loops, $x\not\in N_\mathbb{G}(x)$. The degree of node $x\in V$ in $\mathbb{G}$ is its number of neighbours in $\mathbb{G}$, i.e. $|N_\mathbb{G}(x)|$.

When we use the subset notation $A \subset B$, we allow for the sets $A$ and $B$ to be equal.

\begin{definition}\label{def:kfriendly1}
We say that a simple graph $\mathbb{G}=(V,E)$ is $k$-friendly, for $k\in[0,+\infty)$, if for any two adjacent vertices $x,y\in V$ we have $|  N_\mathbb{G}(x)\cap N_\mathbb{G}(y)|\geq k$.
\end{definition}

For example the complete graph on $n$ vertices is $(n-2)$-friendly.\footnote{We coin the term \textit{$k$-friendly} here. We have not been able to confirm that this concept goes by this or any other name in the literature. The notion of \textit{strongly regular graph} is related, but different.} For later purposes we will need a more general definition.

\begin{definition}\label{def:kfriendly2}
    Let $\mathbb{G}=(V,E)$ and $\mathbb{G}'=(V',E')$ be two simple graphs such that $V'\subset V$. We say that $\mathbb{G}$ is $k$-friendly with respect to $\mathbb{G}'$ if, for any two vertices $x,y \in V'$ that are adjacent in $\mathbb{G}'$, we have $|  N_\mathbb{G}(x)\cap N_\mathbb{G}(y)|\geq k$.     
\end{definition}

Naturally a graph $\mathbb{G}$ is $k$-friendly if and only if it is $k$-friendly with respect to itself. For some intuition let $K_{n,n}$ denote the complete bipartite graph on two classes of size $n$. Let $K_{n,n}'$ be its complement graph. Then $K_{n,n}$ is $n$-friendly with respect to $K_{n,n}'$. The following lemma will be useful for proving Sobolev inequalities on graphs.

\begin{lemma}
\label{lem:friendlybound}
    Let $\mathbb{G}=(V,E)$, $ \mathbb{G}'=(V',E') $ be simple graphs such that $V'\subset V$, $E'\subset E$ and $\mathbb{G}$ is $k$-friendly with respect to $\mathbb{G}'$. Then, for any function $u:V\rightarrow \mathbb{R}$, and constants $q\geq p\geq 1$, we have
    $$ \left( \sum_{xy\in E'} |u(x)-u(y)|^q \right)^{1/q} \leq \frac{2^{1-\frac{p}{q}}}{(k+2)^{\frac{1}{p}-\frac{1}{q}}} \left( \sum_{xy\in E} |u(x)-u(y)|^p \right)^{1/p}.$$
\end{lemma}

\begin{proof}
    To simplify computations write $|u|_p:= \left( \sum_{xy\in E} |u(x)-u(y)|^p \right)^{1/p}$. Let $xy\in E'$ and $z\in \{x,y\}\cup (N_\mathbb{G}(x)\cap N_\mathbb{G}(y))$. Note that, $|u(x)-u(y)|\leq |u(x)-u(z)|+|u(y)-u(z)|$, hence one of the following must hold: $|u(y)-u(z)|\geq \frac{1}{2}|u(y)-u(x)|$ or $|u(x)-u(z)|\geq \frac{1}{2}|u(y)-u(x)|$. This leads us to 
    $$ |u(x)-u(y)|^p \leq \frac{2^p}{k+2}\sum_{z\in \{x,y\}\cup ( N_\mathbb{G}(x)\cap N_\mathbb{G}(y))} |u(x)-u(z)|^p + |u(y)-u(z)|^p \leq \frac{2^p|u|_p^p}{k+2}.$$
    In particular, $\frac{(k+2)^{1/p}|u(x)-u(y)|}{2|u|_p} \leq 1$. Since $q\geq p\geq 1$, we then deduce
    $$ \frac{(k+2)^{q/p}|u(x)-u(y)|^q}{2^q|u|_p^q} \leq \frac{(k+2)|u(x)-u(y)|^p}{2^p|u|_p^p}.$$
    A final computation then concludes our result:
    \begin{align*}
     \left( \sum_{xy\in E'} |u(x)-u(y)|^q \right)^{1/q} &\leq \left( \frac{2^q(k+2)|u|_p^q}{2^p(k+2)^{q/p}|u|_p^p}  \sum_{xy\in E'} |u(x)-u(y)|^p \right) ^{1/q} \\
     &\leq \left( \frac{2^q(k+2)|u|_p^q}{2^p(k+2)^{q/p}|u|_p^p}  \sum_{xy\in E} |u(x)-u(y)|^p \right) ^{1/q} \\
     &= \left( \frac{2^q(k+2)|u|_p^q}{2^p(k+2)^{q/p}}\right)^{1/q} = \frac{2^{1-\frac{p}{q}}}{(k+2)^{\frac{1}{p}-\frac{1}{q}}} |u|_p.
    \end{align*}
\end{proof}

\begin{remark}
In Lemma~\ref{lem:friendlybound} we assumed that $q\geq p$, which implies for the standard $p$ and $q$ vector norms the inequality $\|x\|_q \leq \|x\|_p$. Thus, since $E'\subset E$, in the absence of the $k$-friendly assumption in Lemma~\ref{lem:friendlybound}, we have, for any $u:V\rightarrow\mathbb{R}$,
$$ \left( \sum_{xy\in E'} |u(x)-u(y)|^q \right)^{1/q} \leq  \left( \sum_{xy\in E} |u(x)-u(y)|^p \right)^{1/p}.$$ The $k$-friendly property in the lemma allows us to improve the a priori bound by a factor of $\frac{2^{1-\frac{p}{q}}}{(k+2)^{\frac{1}{p}-\frac{1}{q}}}$. This additional factor will be a crucial ingredient in our proof of Theorem~\ref{thm:graphsobolev}.
\end{remark}

\begin{definition}\label{def:envelopes}
Let $\mathbb{G}=(V,E)$, $\mathbb{G}'=(V,E') $  be two simple graphs on the same vertex set $V$. We say that $\mathbb{G}$ envelopes $\mathbb{G}'$ if $E'\subset E$ and, for any three vertices $x,y,z\in V$, if $xy\in E'$ and $yz\in E'$, then $xz\in E$. 
\end{definition}
We note that $\mathbb{G}$ envelopes $\mathbb{G}'$ if and only if $\mathbb{G}$ contains the square of $\mathbb{G}'$ as a subgraph, where, per definition, two nodes in the square of $\mathbb{G}'$ are linked if and only if the nodes are at most a graph distance $2$ apart in $\mathbb{G}'$. This graph property is useful for providing uniform estimates for a function's `oscillations' on graphs.

\begin{definition}
Let $\mathbb{G}=(V,E)$ be a simple graph and $u:V\rightarrow \mathbb{R}$. We then define a function $\operatorname{osc}_{\mathbb{G}}(u):V\rightarrow \mathbb{R}$ as follows, for $x\in V$,
$$ \operatorname{osc}_{\mathbb{G}}(u)(x) := \max_{z\in\{x\}\cup N_{\mathbb{G}}(x)} u(z) - \min_{z\in \{x\}\cup N_{\mathbb{G}}(x)} u(z) .$$ 
\end{definition}

We originally found the below lemma in an argument provided in \cite[Lemma 4.1]{Dejan2019}. We have only conveniently reformulated this in terms of simple graphs.
\begin{lemma}
\label{lem:envelopebound}
Let $\mathbb{G}=(V,E)$, $ \mathbb{G}'=(V,E') $ be simple graphs such that $\mathbb{G}$ envelopes $\mathbb{G}'$. Then, for any function $u:V\rightarrow \mathbb{R}$ and $p>0$, 
$$\left[\operatorname{osc}_{\mathbb{G}'}(u)(x)\right]^p \leq \frac{2^p}{|N_{\mathbb{G}'}(x)|} \sum_{yz\in E} |u(y)-u(z)|^p$$
(where the right-hand side is interpreted to be $\infty$ if $N_{\mathbb{G}'}(x)=\emptyset$).
\end{lemma}
\begin{proof}
Let $x\in V$ and choose $\overline{x},\underline{x}\in \{x\}\cup N_{\mathbb{G}'}(x)$ such that $\operatorname{osc}_{\mathbb{G}'}(u)(x) := u(\overline{x})-u(\underline{x})$. If $\overline{x}=\underline{x}$ then the bound is trivial so assume otherwise. For each $y\in \left(\{x\}\cup N_{\mathbb{G}'}(x)\right)\setminus \{\overline{x},\underline{x} \}$ we have either $ u(\overline{x})-u(y)\geq \frac{1}{2} \operatorname{osc}_{\mathbb{G}'}(u)(x) $ or $ u(y)-u(\underline{x})\geq \frac{1}{2} \operatorname{osc}_{\mathbb{G}'}(u)(x)$. Moreover, $\overline{x}y$ and $\underline{x}y$ are distinct edges in $E$. Additionally observe $\overline{x}\underline{x}$ is an edge in $E$ distinct from those listed above. As a result we deduce
\begin{align*}
\left[\operatorname{osc}_{\mathbb{G}'}(u)(x) \right]^p &\leq \frac{2^p}{|N_{\mathbb{G}'}(x)|}\left( |u(\overline{x})-u(\underline{x})|^p +\sum_{y\in \left(\{x\}\cup N_{\mathbb{G}'}(x)\right)\setminus \{\overline{x},\underline{x} \}} | u(\overline{x})-u(y)|^p + |u(y)-u(\underline{x})|^p \right)  \\
&\leq \frac{2^p}{|N_{\mathbb{G}'}(x)|} \sum_{yz\in E} |u(y)-u(z)|^p .
\end{align*}
\end{proof}

\section{Geometric graph setting and combinatorics}
\label{sec:geometry}

For any $x\in \mathbb{R}^d$, we define $B_r(x)$ to be the open ball of radius $r$ centred at $x$.

In this section we introduce the geometric graph setting which will feature in our uniform Sobolev inequalities. In particular, our graphs will be taken from from an open and bounded subset $\Omega \subset \mathbb{R}^d$ with a Lipschitz boundary.

\begin{remark}\label{rem:Lipschitz}
An open set $\Omega\subset\mathbb{R}^d$ has Lipschitz boundary if, for every point $x\in \partial\Omega$, locally around $x$ the boundary can be described as the graph of a Lipschitz function on an orthogonal choice of co-ordinates. To be precise, there exists an orthonormal basis $\{e_1,e_2, \dots, e_d\}$ of $\mathbb{R}^d$, an open neighbourhood $U$ of $x$, and a Lipschitz function $\Psi:\mathbb{R}^{d-1}\rightarrow \mathbb{R}$ such that the following hold:
\begin{itemize}
    \item $\Omega\cap U = \{  y_1e_1+\dots + y_de_d \; | \; (y_1,\dots,y_d) \in U \; \text{and} \; y_d< \Psi(y_1,\dots,y_d) \}, $
    \item $\partial\Omega\cap U = \{  y_1e_1+\dots + y_de_d \; | \; (y_1,\dots,y_d) \in U \; \text{and} \; y_d = \Psi(y_1,\dots,y_d) \}, $
    \item $\big(\mathbb{R}^d\setminus\overline{\Omega}\big)\cap U = \{  y_1e_1+\dots +y_de_d \; | \; (y_1,\dots,y_d) \in U \; \text{and} \; y_d > \Psi(y_1,\dots,y_d) \}. $
\end{itemize}

On the other hand, an open set $\Omega\subset\mathbb{R}^d$ has a weak Lipschitz boundary if $\overline{\Omega}$ is a Lipschitz manifold of dimension $d$ with boundary (the boundary is itself a Lipschitz manifold of dimension $d-1$). What this means is, for every point $x\in \partial\Omega$, there exists an open neighbourhood $U$ of $x$ (open in the topology of $\mathbb{R}^d$) and a bi-Lipschitz homeomorphism $\Phi:U\rightarrow B$ between $U$ and the unit open ball $B$ in $\mathbb{R}^d$, such that the following hold. Let $B^-$ be the set of $x\in B$ such that $x_n<0$, $B^+$ be the set of $x\in B$ such that $x_n>0$ and $H$ be the set of $x\in B$ such that $x_n=0$. Then 
$\Phi(U\cap \Omega)=B^-$, $\Phi(U\cap (\mathbb{R}^d\setminus \overline{\Omega}))=B^+$ and $\Phi(U\cap\partial\Omega)=H$.

Indeed, if $\Omega$ has a Lipschitz boundary then it has a weak Lipschitz boundary, although a proof of this fact is not trivial. One has to construct a bi-Lipschitz homeomorphism using the local graph representation. The argument follows in a similar fashion to the $C^1$ inverse function theorem. A Lipschitz function is differentiable almost everywhere, according to Rademacher's theorem \cite[Theorem 3.2]{Evans1992}, and, even though it need not have a continuous derivative, this difficulty can be overcome with techniques developed in \cite{Clarke1976}.

There are sets with a weak Lipschitz boundary but not a Lipschitz boundary. See, for example, \cite[Figure 1]{Licht2019}.
\end{remark}

Since we assume that $\Omega$ has a Lipschitz boundary, it also has a weak Lipschitz boundary. We shall use this fact multiple times in our proofs. Moreover, many of our results would still hold were we to assume $\Omega$ only has a weak Lipschitz boundary rather than a Lipschitz boundary. The reason we still assume $\Omega$ has a Lipschitz boundary is so we can apply results from classical Sobolev theory in the continuum. It is not currently known by the authors if the weak Lipschitz boundary condition is sufficient for proving results from classical Sobolev theory. In particular, the theorems we require within this paper are the Rellich--Kondrachov compactness result~\cite[Theorem 11.10]{Leoni2009} and the classical Sobolev inequalities~\cite[Theorem 1.4.4.1]{Grisvard1985} for $p<d$ and $p>d$. Alongside this, there are results in the literature we use, such as \cite[Lemma 4.6]{Dejan2019}, that assume $\Omega$ has Lipschitz boundary for the same reasons we have described above.  

On $\Omega$ we fix a probability measure $\mu$ with continuous density $\rho: \Omega \to \mathbb{R}$ with respect to the Lebesgue measure, thus $\mu$ is absolutely continuous with respect to the Lebesgue measure on $\Omega$. Throughout this paper we shall make frequent use of indicator functions; the notation we use should be interpreted as follows. For a given statement $P$, $\mathbf{1}_P=1$ if $P$ is true and $\mathbf{1}_P=0$ if $P$ is false. Additionally $\Gamma$ denotes Euler's gamma-function; for the unit ball $\mathcal{B}$ in $\mathbb{R}^d$,
$ \mathfrak{L}^d(\mathcal{B})= \frac{\pi^{d/2}}{\Gamma\left( \frac{d}{2}+1\right)}.$ We also define the diameter of a Borel subset $A\subset \mathbb{R}^d$ by
$$ \operatorname{diam}(A):= \underset{A\times A}{\operatorname{esssup}} \{ (x,y)\xmapsto{} |x-y| \}.$$
A standard fact regarding diameters is that, for any Borel subset $A\subset\mathbb{R}^d$,
$$ \mathfrak{L}^d(A) \leq \frac{\pi^{d/2}}{2^d \Gamma(\frac{d}{2}+1)} \operatorname{diam}(A)^d .$$

Before we construct our graphs we note the following lower bound on the infinite Wasserstein distance between the probability measure $\mu$ and any discrete probability measure. This lower bound gives some intuition into how well a continuum measure can be approximated by a discrete one. 

\begin{lemma}
\label{lem:transportbound}
Let $V\subset\Omega$ be a finite subset. Let $S:\Omega \rightarrow V$ be a Borel map. Then we have
$$  \underset{\Omega}{\operatorname{esssup}}\{x\xmapsto{} |S(x)-x | \} \geq \frac{\Gamma(\frac{d}{2}+1)^{1/d}\mathfrak{L}^d(\Omega)^{1/d}}{\pi^{1/2}|V|^{1/d}}.$$
\end{lemma}
\begin{proof}
We note that $\{S^{-1}(x) \;| \; x\in V\}$ partitions $\Omega$, hence $\sum_{x\in V} \mathfrak{L}^d(S^{-1}(x)) = \mathfrak{L}^d(\Omega)$. So there exists $y\in V$ such that the following holds:
$$\frac{\mathfrak{L}^d(\Omega)}{|V|} \leq \mathfrak{L}^d(S^{-1}(y)) \leq \frac{\pi^{d/2}}{2^d \Gamma(\frac{d}{2}+1)} \operatorname{diam}(S^{-1}(y))^d.$$ 
Let $x,y\in S^{-1}(y)$. We note that
$$ |x-y|=|x-S(x)+S(y) -y| \leq |x-S(x)|+|y-S(y)|,$$
hence $\max\{|x-S(x)|,|y-S(y)|\}\geq |x-y|/2$. We deduce
$$  \underset{\Omega}{\operatorname{esssup}} \{x\xmapsto{}|S(x)-x| \} \geq  \underset{S^{-1}(y)}{\operatorname{esssup}} \{x\xmapsto{}|S(x)-x| \} \geq \frac{1}{2}\operatorname{diam}(S^{-1}(y)) .$$
Putting this together with our previous bound we get
$$ \frac{\mathfrak{L}^d(\Omega)}{|V|} \leq\frac{\pi^{d/2}}{ \Gamma(\frac{d}{2}+1)} \left(  \underset{\Omega}{\operatorname{esssup}} \{ x\xmapsto{} |S(x)-x|  \}\right)^d.$$

\end{proof}

The bound above shows that, up to a constant, we cannot approximate $\mu$ in the infinite Wasserstein distance, by a discrete measure of $n$ points any better than $n^{-1/d}$.

Now we shall fix the geometric graphs considered throughout this paper. By $\mathbb{N}$ we denote the set of natural numbers excluding $0$. We introduce the following:

\begin{itemize}
    \item a sequence of finite subsets $V_n\subset \Omega$ for $n\in\mathbb{N}$ which will form our vertex sets;
    \item a sequence of probability measures $\mu_n$ defined on $V_n$; for $n\in \mathbb{N}$ we define 
    $$   \Lambda_n^- := \min_{x\in V_n} \mu_n(x) \; \; \text{and} \; \; \Lambda_n^+:= \max_{x\in V_n} \mu_n(x);$$
    \item a sequence $(\varepsilon_n)_{n\in\mathbb{N}}$ such that $\varepsilon_n>0$ and $\varepsilon_n\rightarrow 0$ as $n\rightarrow \infty$, we shall refer to $(\varepsilon_n)_{n\in\mathbb{N}}$ as the concentrating parameters; 
    \item a function $\eta:[0,\infty)\rightarrow [0,\infty)$, called the kernel; for $\varepsilon>0$ we define $\eta_\varepsilon(\cdot):=\frac{1}{\varepsilon^d}\eta(\cdot/\varepsilon)$.
\end{itemize}

 For $s>0$ and $n\in\mathbb{N}$ we define the simple graph $\mathbb{G}_{n,s}:=(V_n,E_{n,s})$ as follows. For $x,y\in V_n$, $xy\in E_{n,s}$ if and only if $x\neq y$ and $|x-y|\leq s$. We emphasise that, different to the setup in other papers nearby in the literature, the definition of the edge sets $E_{n,s}$ does not depend on the kernel function $\eta$.

\begin{assumptions}
\label{assum}
We refer to the assumptions below as needed throughout the paper. They ensure that our discrete graph setup is a good approximation of the continuum domain $\Omega$.
\begin{enumerate}[(i)]
    \item There exists a sequence of Borel maps $(T_n)_{n\in\mathbb{N}}$ such that $T_n:\Omega \rightarrow V_n$ and $T_n \# \mu=\mu_n$. Moreover there exists $K>0$ such that
    \begin{equation*}
        \sup_{x\in\Omega}|T(x)-x|\leq K\varepsilon_n .
    \end{equation*}
    \item[(i$^*$)] There exists a sequence of Borel maps $(T_n)_{n\in\mathbb{N}}$ such that $T_n:\Omega \rightarrow V_n$ and $T_n \# \mu=\mu_n$. Moreover the following holds:
    \begin{equation*}
        \frac{1}{\varepsilon_n} \sup_{x\in\Omega}|T(x)-x|\rightarrow 0 \; \; \text{as} \; \; n\rightarrow \infty .
    \end{equation*}
    \item The kernel $\eta$ is continuous at $0$, non-increasing and $\eta(0)>0$.
    \item There exists a constant $D>0$ such that, for any $x\in \Omega$, $\frac{1}{D}\leq \rho(x) \leq D$.
    \item There exists a constant $\mathfrak{D}>0$ such that, for any $n\in \mathbb{N}$ and $x\in V_n$,  $$\frac{1}{\mathfrak{D}|V_n|}\leq \mu_n(x) \leq \frac{\mathfrak{D}}{|V_n|}.$$
\end{enumerate}
\end{assumptions}

We note the following consequences of our assumptions.
\begin{itemize}
\item Assumption (i) implies $d_\infty(\mu_n,\mu)\leq K\varepsilon_n$ whereas assumption (i$^*$) implies $\frac{1}{\varepsilon_n}d_\infty(\mu_n,\mu)\rightarrow 0$ as $n\rightarrow\infty$.
    \item Assumptions (i) and (i$^*$) are both lower bounds on how quickly $(\varepsilon_n)_{n\in \mathbb{N}}$ can converge to zero. If $\varepsilon_n$ is chosen too small then the underlying graph $\mathbb{G}_{n,\varepsilon_n}$ is completely disconnected. On the other hand (i$^*$) (for $n$ sufficiently large) and (i) for a $K\leq \frac{1}{2\sqrt{d}}$ imply $\mathbb{G}_{n,\varepsilon_n}$ is connected, according to Proposition~\ref{prop:connect}.
    \item
    We can apply Lemma~\ref{lem:transportbound} and assumption (i) to deduce there exists a constant $C>0$ such that $\varepsilon_n \geq C/|V_n|^{1/d}$.
    \item 
    It is clear that (i$^*$) implies (i).  
    It may be useful to have a combinatorial interpretation of the assumptions (i) (for $K>0$) and (i$^*$). Denote $\overline{B}_{r}(x)$ as the closed ball of radius $r$ centred at $x$. Then, given any $x\in \Omega $, (i) implies that $V_n \cap \overline{B}_{K\varepsilon_n}(x)$ is non-empty, for any $n\in \mathbb{N}$. On the contrary, (i$^*$) from Assumptions~\ref{assum} implies that $\left|V_n \cap \overline{B}_{K\varepsilon_n}(x)\right| \rightarrow +\infty$ as $n\rightarrow \infty$.
\end{itemize}

\begin{remark}
\label{rem:simpletokernel}
Part (ii) from Assumptions~\ref{assum} implies there exist constants $a>0$ and $b\geq 1$ such that $\eta(r)\geq a\cdot\mathbf{1}_{br\leq 1}$. By making a change of variables, we observe, for any $r,s\in [0,+\infty)$, that
$$ \mathbf{1}_{r\leq s} \leq \frac{1}{a} \eta \left( \frac{r}{bs} \right).$$
We will refer back to this relationship numerous times. 
\end{remark}

\begin{remark}
\label{rem:pointcloud}

In the literature about discrete-to-continuum limits that use the $TL^p$ distance (for example \cite{garcia2016continuum,von2016,Garcia2017,Osting2017,Thorpe2017,Alonso2018,Davis2018,Garcia2018,Dejan2019,Theil2019,Garcia2020,slepcev2020,Thorpe2020,Dunlop2020,Kaplan2020,ThorpeVanGennip23}), the graphs constructed are typically random point clouds. That setting satisfies Assumptions~\ref{assum}, as we will show now. 

To construct a random geometric graph, a sequence of points $(Z_n)_{n\in\mathbb{N}}\subset \Omega$ are iid sampled from the probability measure $\mu$. We refer to $\Pi$ as the joint probability space they live in and define $V_n:=\{Z_1,\ldots,Z_n\}$. A discrete probability measure $\mu_n$, called the empirical measure (correpsonding to $V_n$), is introduced on the vertex set $V_n$, given by
$$ \mu_n := \frac{1}{n} \sum_{i=1}^n \delta_{Z_i} .$$
Here $\delta_x$ refers to the Dirac mass at $x$. For the setup we are now considering, there exists an upper bound on the Wasserstein distance $d_\infty(\mu_n, \mu)$ which we now detail.

The authors in \cite{Nicolas2015} prove that, for any $\alpha>2$, there exists a constant $C$ depending on $D$ and $\alpha$ and a sequence of events $(\mathcal{A}_n)_{n\in \mathbb{N}}$, such that the probability of $\mathcal{A}_n^c$ in $\Pi$ decays at a rate $\mathcal{O}(n^{-\alpha/2})$ and, within the event, $\mathcal{A}_n$ the following statement holds. There exists a Borel map $T_n:\Omega\rightarrow V_n$ such that $T_n\#\mu=\mu_n$ and 
\begin{equation}
\label{eq:transportrates}
    d_\infty(\mu_n,\mu)\leq \sup_{x\in\Omega} | T_n(x) -x | \leq C \begin{cases}
        \left( \frac{\log \log n}{n} \right)^{1/2}, & \text{if} \; d=1, \\
        \frac{(\log n)^{3/4}}{n^{1/2}}, & \text{if} \; d=2,\\
        \frac{(\log n)^{1/d}}{n^{1/d}}, & \text{if} \; d\geq 3.
    \end{cases}
\end{equation}

Fix $\alpha>2$ as above. We observe there exists a constant $B>0$ such that
$$ \sum_{n=1}^{\infty} \mathbb{P}_{\Pi}(\mathcal{A}_n^c)\leq B\sum_{n=1}^{\infty} \frac{1}{n^{\alpha/2}}<+\infty.$$
As a consequence of the Borel--Cantelli lemma all but finitely many of the events $(\mathcal{A}_n)_{n\in \mathbb{N}}$ occur with probability 1. This implies with probability $1$ in $\Pi$, that $\mu_n$ converges to $\mu$ with respect to the $d_{\infty}$ metric at a rate given by equation~\ref{eq:transportrates}.

Within the works \cite{garcia2016continuum},\cite{Dejan2019},\cite{Garcia2018} the authors assume
\begin{equation}
\label{ratesDp}
\begin{aligned}
&\lim_{n \rightarrow \infty} \left(\frac{\log \log n}{n} \right)^{1/2}\frac{1}{\varepsilon_n} = 0, & \text{if } d=1,\\
& \lim_{n \rightarrow \infty} \frac{(\log n)^{3/4}}{n^{1/2}} \frac{1}{\varepsilon_n} = 0, & \text{if } d=2, \\
& \lim_{n \rightarrow \infty} \frac{(\log n)^{1/d}}{n^{1/d}} \frac{1}{\varepsilon_n} = 0, & \text{if } d\geq 3.
\end{aligned}
\end{equation}
Given the bound in \eqref{eq:transportrates} we note that \eqref{ratesDp} implies (i$^*$) from Assumptions~\ref{assum} holds with probability $1$ in $\Pi$. Throughout this paper, (i) from Assumptions~\ref{assum} will play the role that \eqref{ratesDp} has in the previous works listed at the start of this remark; notably it is strictly weaker. We will discuss the implications of assuming (i) instead of (i$^*$) in further detail in the discussion section.    

In the random point cloud setting, the probability measures $\mu_n$ are empirical measures and thus part~(iv) from Assumptions~\ref{assum} automatically holds. 
\end{remark}

The following proposition investigates the connectivity of our geometric graph setting. An illustration of the proof is given in Figure~\ref{fig:connect}.

\begin{proposition}
\label{prop:connect}
Assume that $\Omega$ is convex. Let $V\subset \Omega$ have finite cardinality and let $T:\Omega\rightarrow V$. Define $\mathbb{G}_s$ to be the simple graph on $V$ such that two vertices $x,y$ are adjacent when $x\neq y$ and $|x-y|\leq s$. Then $\mathbb{G}_s$ is connected for $s\geq 2\sqrt{d} \sup_{x\in \Omega} |T(x)-x|$.  
\end{proposition}

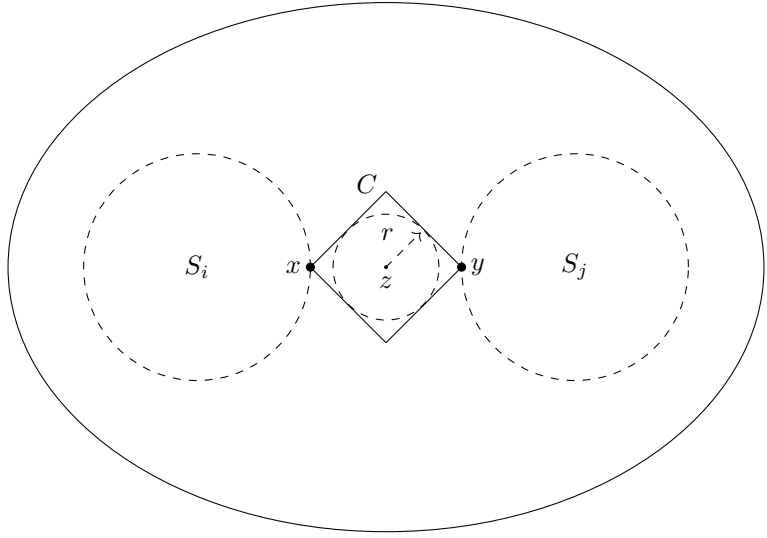
\begin{figure}[ht]
    \begin{center}
    \begin{tikzpicture}[scale=0.5]
\draw (0,0) ellipse (10 and 7) ;
\draw (-2,0) -- (0,2) ;
\draw (0,2) -- (2,0) ;
\draw (-2,0) -- (0,-2) ;
\draw (0,-2) -- (2,0) ;
\draw[dashed] (-5,0) circle (3) ;
\draw[dashed] (5,0) circle (3) ;
\filldraw[black] (-2,0) circle (3pt) node[anchor=east]{$x$} ;
\filldraw[black] (2,0) circle (3pt) node[anchor=west]{$y$} ;
\draw (-5,0) node {$S_i$} ;
\draw (5,0) node {$S_j$} ;
\filldraw[black] (0,0) circle (1pt) node[anchor=north]{$z$} ;
\draw[->,dashed] (0,0) -- node[anchor=south east]{$r$} ++(45:1.3cm) ;
\draw[dashed] (0,0) circle (1.4) ;
\draw (-0.5,2.2) node {$C$} ;
\end{tikzpicture}
\end{center}
    \caption{Illustration of the proof of Proposition~\ref{prop:connect}}
    \label{fig:connect}
\end{figure}

Before we proceed with the proof we define the following. For sets $A,B\subset\mathbb{R}^d$,
$$ \operatorname{dist}(A,B):=\inf_{x\in A}\inf_{y\in B} |x-y|.$$

\begin{proof}
Let $s\geq 2\sqrt{d} \sup_{x\in \Omega} |T(x)-x|$ (we note that the right-hand side is finite by boundedness of $\Omega$) and assume for a contradiction that $\mathbb{G}_s$ is disconnected. First we partition the vertex set $V$ into the connected components of the graph $\mathbb{G}_s$: $S_1,S_2, \ldots ,S_l$. Since $\mathbb{G}_s$ is disconnected, $l\geq2$. Choose indices $i,j$ that minimize $\operatorname{dist}(S_i,S_j)$, then choose $x\in S_i$ and $y\in S_j$ such that $|x-y|=\operatorname{dist}(S_i,S_j)$. Let $z$ be the midpoint of $x,y$ and $C$ a (hyper)cube centred at $z$ which has opposite vertices positioned at $x$ and $y$. The convexity assumption on $\Omega$ ensures that $z\in \Omega$.\footnote{We note that midpoint convexity of $\Omega$ (i.e., for all $x,y\in \Omega$, $\frac12(x+y)\in \Omega$), which is implied by convexity, would suffice for this conclusion. For any set that, locally, lies on one side of its boundary, as our $\Omega$ with Lipschitz boundary does, midpoint convexity is in fact equivalent to convexity. Indeed, for a closed set midpoint convexity implies convexity: given elements $x$ and $y$ of the set, repeated applications of midpoint convexity to a dyadic partition of the line segment between $x$ and $y$ give a dense subset of points on the line segment that are all in the set. If the set is closed, it thus contains the whole line segment. Hence, if a set is midpoint convex but not convex, there must exist points $x$ and $y$ in the set and a point $z$ on the line segment between $x$ and $y$ that lies on the boundary of the set, yet not in the set. If locally at this point, the set lies on one side of the boundary, we arrive at a contradiction, since the earlier constructed dense subset of the line segment lies dense on both sides of the boundary.

Without additional assumptions, midpoint convexity does not imply convexity, as for example $\mathbb{Q}$ as subset of $\mathbb{R}$ is midpoint convex yet not convex.} Let $r$ be half the length of one of the edges of the cube $C$, i.e. $r:=|x-y|/(2\sqrt{d})$. Since $x$ and $y$ are in different connected components they are not adjacent in $\mathbb{G}_s$, therefore $r> \frac{s}{2\sqrt{d}} \geq \sup_{x\in \Omega} |T(x)-x|$. We then choose a constant $r'>0$ such that $r>r'>\sup_{x\in \Omega} |T(x)-x|$.

We observe that $S_i\cap C=\{x\}$, otherwise there would be a point in $S_i$ closer to $y$ than $x$ is to $y$, which would contradict that $|x-y|=\operatorname{dist}(S_i,S_j)$. Similarly $S_j\cap C=\{y\}$. Moreover, for $k\not\in\{i,j\}$, $S_k\cap C=\emptyset$, otherwise $\operatorname{dist}(S_k,S_i)<\operatorname{dist}(S_i,S_j)$ which contradicts that $i,j$ were chosen minimal.

Consequently, $C\cap V=\{x,y\}$ and therefore $\overline{B_{r'}(z)}\cap V=\emptyset$. This yields a contradiction since, as previously stated, $r'>  \sup_{x\in \Omega} |T(x)-x|$ and therefore $T(z)\in \overline{B_{r'}(z)} \cap V $. The proof then follows.
\end{proof}

The above proposition suggests a close relationship between the connectivity of a geometric graph on a finite subset $V\subset \Omega$ and the $\infty$-Wasserstein distance between a discrete probability measure on $V$ and a continuum probability measure on $\Omega$. To make this relationship explicit we define the connectivity threshold for a finite set.

\begin{definition}\label{def:connthreshold}
Let $V\subset \mathbb{R}^d$ have finite cardinality. For $s\geq 0$, take $\mathbb{G}_s(V)$ to be the simple graph with vertex set $V$, in which two vertices are connected if they are at most a distance $s$ (inclusive) apart. Define the connectivity threshold of $V$, denoted by $s^*(V)$, to be the infimum over all $s\geq0$ such that $\mathbb{G}_s(V)$ is connected. 
\end{definition}
We note that $s^*(V)$ is well defined, since for large enough $s$, $\mathbb{G}_s(V)$ is connected by finiteness of $V$.

\begin{proposition}\label{prop:threshold}
Let $V\subset \mathbb{R}^d$ have finite cardinality. Then  $\mathbb{G}_{s^*(V)}(V)$ is itself connected.   
\end{proposition}

\begin{proof}
For simplicity we shall write $s^*:=s^*(V)$. We shall argue by contradiction so assume $G_{s^*}(V)$ is disconnected. Define
$$ s':= \min \{ |x-y| \; | \; x,y\in V \; \text{and} \; |x-y|> s^*   \}.$$
Given that $V$ is finite we have $s'>s^*$. Therefore, we can choose $\delta>0$ such that $s^*+\delta<s'$. Observe that $\mathbb{G}_{s^*+\delta}(V)=\mathbb{G}_{s^*}(V)$. However, by the definition of $s^*$, $\mathbb{G}_{s^*+\delta}(V)$ is connected. This yields a contradiction.   
\end{proof}

Define $\mathcal{P}_c(\Omega)$ to be the set of Borel probability measures $\mu$ on $\Omega$ that are equivalent to $\mathfrak{L}^d|_{\Omega}$, i.e. each $\mu$ is absolutely continuous with respect to $\mathfrak{L}^d|_{\Omega}$ and vice versa. Define $\mathcal{P}(V)$ to be the set of probability measures supported on $V$. As a corollary of Proposition~\ref{prop:connect} we have the following upper bound.

\begin{corollary}
\label{cor:connect}
Assume that $\Omega$ is convex and let $V\subset\Omega$ have finite cardinality. Then
$$ s^*(V) \leq 2\sqrt{d} \inf_{\mu\in\mathcal{P}_c(\Omega)} \inf_{\nu\in \mathcal{P}(V)} d_\infty(\mu,\nu).$$
\end{corollary}

\begin{proof}
    Let $\nu\in \mathcal{P}(V)$, $\mu\in \mathcal{P}_c(\Omega)$ and $\varepsilon>0$. By Theorem~\ref{thm:transportexistence} choose  $T:\Omega\rightarrow \Omega$ such that $T\#\mu=\nu$ and
    $$ d_\infty(\mu,\nu)= \sup_{x\in\Omega} |T(x)-x|.$$
    Our objective is to modify $T$ to a map $T_\varepsilon:\Omega\rightarrow V$ suitable for Proposition~\ref{prop:connect}; in particular $T_\varepsilon$ should have its range in $V$. Define $N:= \{x\in \Omega \; | \; T(x)\notin V\}$, and observe that $N$ is a $\mu$-null set, since $T\#\mu=\nu$ has support in $V$. Given that $\Omega$ is pre-compact, we can construct a finite open cover of $\Omega$, $\{B_{\varepsilon/2}(x'_1),B_{\varepsilon/2}(x'_2),\ldots,B_{\varepsilon/2}(x'_k)$\} with $x'_1,x'_2,\ldots,x'_k\in \overline\Omega$, for some $k\in \mathbb{N}$. For any $x'_i \in \partial \Omega$, $B_{\varepsilon/2}(x'_i) \cap \Omega \neq \emptyset$. Replace $x'_i$ by a point in this intersection (again named $x'_i$). Then $\{B_\varepsilon(x'_1),B_\varepsilon(x'_2),\ldots,B_\varepsilon(x'_k)$\} is a finite open cover of $\Omega$ with $x'_1,x'_2,\ldots,x'_k\in \Omega$. Given that $N$ is null, choose $x_1,\ldots,x_k\in \Omega\setminus N$ such that, for each $i\in \{1,\ldots,k\}$, $x_i\in B_\varepsilon(x_i')$. Observe that $\{B_{2\varepsilon}(x_1),B_{2\varepsilon}(x_2),\ldots,B_{2\varepsilon}(x_k)$\} is an open cover of $\Omega$.   

Next we define $T_\varepsilon:\Omega\rightarrow V$ as follows. For each $x\in \Omega$, choose $x_\varepsilon$ to be any element of $\{x_1,\ldots,x_k\}$ such that $x\in B_{2\varepsilon}(x_\varepsilon)$ and set $T_\varepsilon(x):=T(x_\varepsilon)$. Then, for $x\in \Omega$, 
$$|T_\varepsilon(x)-x|=|T(x_\varepsilon)-x_\varepsilon+x_\varepsilon-x|\leq d_\infty(\mu,\nu)+2\varepsilon.$$
Therefore, $\sup_{x\in\Omega} |T_\varepsilon(x)-x|\leq d_\infty(\mu,\nu)+2\varepsilon$ and $T_\varepsilon$ has range in $V$. Applying Proposition~\ref{prop:connect}, which is allowed since $\Omega$ is assumed to be convex, we conclude that
$$s^*(V)\leq 2\sqrt{d} \;( d_\infty(\mu,\nu)+ 2\varepsilon).$$
Taking an infimum over all $\mu\in \mathcal{P}_c(\Omega)$, $\nu\in \mathcal{P}(V)$ and $\varepsilon>0$ concludes our result.

\end{proof}

\begin{corollary}\label{cor:connect2}
    Assume that there exists a bi-Lipschitz homeomorphism from $\Omega$ onto $\Sigma$.\footnote{Such a function is also called a lipeomorphism and if such a function exists $\Omega$ and $\Sigma$ are called lipeomorphic.} Then there exists a constant $C(\Omega)>0$, depending on the set $\Omega$, such that, for any $V\subset \Omega$ with finite cardinality,
    $$ s^*(V)\leq C(\Omega) \inf_{\mu\in\mathcal{P}_c(\Omega)} \inf_{\nu\in \mathcal{P}(V)} d_\infty(\mu,\nu).$$

\end{corollary}

\begin{proof}
    Let $\Sigma$ be a convex set and $\phi:\Omega\rightarrow \Sigma$ a bi-Lipschitz homeomorphism. Choose a constant $L>0$ such that, for any $x,y\in \Omega$,
    $$ \frac{1}{L}|x-y|\leq|\phi(x)-\phi(y)|\leq L|x-y|.$$
    Let $V\subset \Omega$ be a finite subset and define $\widetilde{V}:=\phi(V)$. We observe that, for any edge $\phi(x)\phi(y)$ in $\mathbb{G}_{s^*(\widetilde{V})}(\widetilde{V})$, $xy$ is an edge in $\mathbb{G}_{Ls^*(\widetilde{V})}(V)$. Since $\mathbb{G}_{s^*(\widetilde{V})}(\widetilde{V})$ is connected according to Proposition~\ref{prop:threshold}, we observe that $\mathbb{G}_{Ls^*(\widetilde{V})}(V)$ is also connected. Therefore $s^*(V)\leq Ls^*(\widetilde{V})$.

    Let $\mu\in \mathcal{P}_c(\Omega)$, $\nu\in \mathcal{P}(V)$ and $\pi\in\Gamma(\mu,\nu)$. Define $\phi\times \phi: \Omega\times\Omega\rightarrow \Sigma\times \Sigma$ such that, for $(x,y)\in\Omega\times\Omega$, $\phi\times\phi(x,y)=(\phi(x),\phi(y))$. Define $\widetilde{\mu}:=\phi\#\mu$, $\widetilde{\nu}:=\phi\#\nu$ and $\widetilde{\pi}:= (\phi\times\phi)\#\pi$. We note that $\widetilde{\pi}\in\Gamma(\widetilde{\mu},\widetilde{\nu})$. Using Corollary~\ref{cor:connect} we get the following bound,
    \begin{align*}
      s^*(V) \leq Ls^*(\widetilde{V})&\leq 2\sqrt d L\underset{(\Sigma\times\Sigma,\widetilde{\pi})}{\operatorname{esssup}} \{ (w,z)\mapsto |w-z| \} \\
      &= 2\sqrt d L\underset{(\Omega\times\Omega,\pi)}{\operatorname{esssup}} \{ (x,y)\mapsto |\phi(x)-\phi(y)| \} \\
      &\leq 2 \sqrt d L^2 \underset{(\Omega\times\Omega,\pi)}{\operatorname{esssup}} \{ (x,y)\mapsto |x-y| \}.  
    \end{align*}
    Taking an infimum over all $\pi\in\Gamma(\mu,\nu)$ we get $s^*(V)\leq 2\sqrt d L^2d_\infty(\mu,\nu)$. Taking $C(\Omega):=2\sqrt d L^2$, the result follows.
\end{proof}

\begin{remark}\label{rem:conndinfty}
Throughout this remark we shall assume $\Omega$ is bi-Lipschitz homeomorphic to a convex set.

It would be especially interesting if a lower bound can be produced for $s^*(V)$ in terms of the metric $d_\infty$ of the same form as the upper bound in Corollary~\ref{cor:connect}, namely
$$ s^*(V)\geq C \inf_{\mu\in \mathcal{P}_c(\Omega)} \inf_{\nu\in \mathcal{P}(V)} d_\infty(\mu,\nu),
$$
for a constant $C>0$ independent of $V$, possibly dependent on $\Omega$. 

There is evidence in the form of combinatorial results for random point clouds (the setting discussed in Remark~\ref{rem:pointcloud}) to suggest, in the $d=2$ case, that a lower bound of this form is impossible. In particular, using the same notation as in Remark~\ref{rem:pointcloud}, it is known that with high probability (see~\cite[Theorem 1.1]{Nicolas2015})
$$ d_\infty(\mu_n,\mu) = \mathcal{O}\left( \frac{(\log n)^{3/4}}{n^{1/2}} \right).$$
On the contrary (for the case when $\Omega$ is a disc, see~\cite[Theorem 3.2]{Gupta1999}), with high probability, 
$$ s^*(V_n) = \mathcal{O}\left( \frac{(\log n)^{1/2}}{n^{1/2}} \right).$$

Therefore, $d_\infty(\mu_n,\mu)$ is larger that $s^*(V_n)$ by a factor of $(\log n)^{1/4}$. Nonetheless, for $d\geq 3$ in the random-point-cloud setting, $s^*(V_n)$ and $d_\infty(\mu_n,\mu)$ both converge to zero at the same rate (with high probability). For now, the existence of a lower bound remains an open problem. For further details on the connectivity thresholds of geometric graphs, we refer to~\cite[Chapter 13]{Penrose2003}, further details on the infinite-Wasserstein distance for random point clouds can be found in~\cite{Nicolas2015}.
\end{remark}

As a consequence of Corollary~\ref{cor:connect2}, if $(i)$ from Assumptions~\ref{assum} holds for $K>0$ and a bi-Lipschitz homeomorphism exists from $\Omega$ onto a convex set, then, for $r_n:= C(\Omega) K\varepsilon_n$, $\mathbb{G}_{n, r_n}$ is connected.

Next we shall investigate the friendly property, which we introduced in Definition~\ref{def:kfriendly2} , for the geometric graphs $\mathbb{G}_{n,s}$. Figure~\ref{fig:geometry} gives some intuition as to how we identify the $k$-friendly property inside of $\mathbb{G}_{n,s}$. Given two points $x,y$ such that $|x-y|\leq s$, there is an open ball of radius $s/2$ inside of $B_s(x)\cap B_s(y)$. The following lemma will allow us to quantify the $k$-friendly property for geometric graphs.

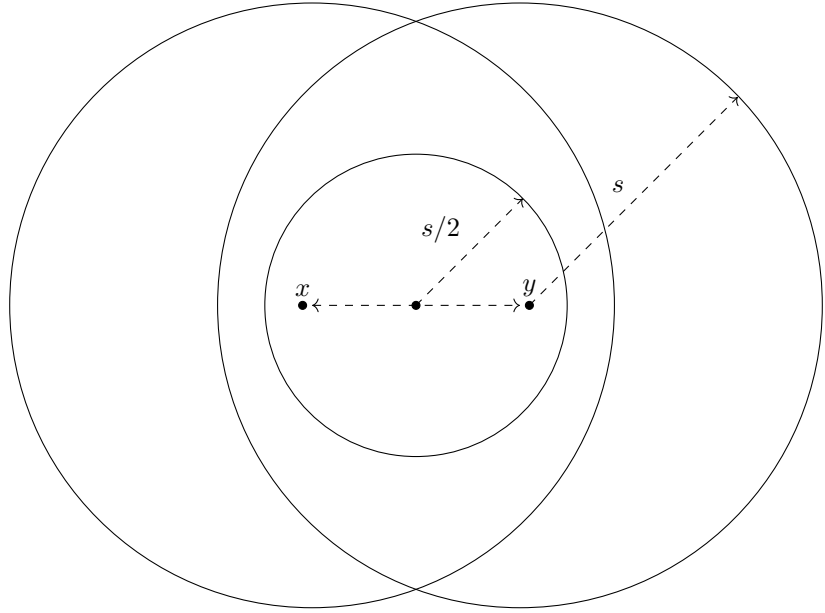
\begin{figure}[ht]
    \begin{center}
    \begin{tikzpicture}[scale=0.5]
\filldraw[black] (7,0) circle (3pt) node[anchor=south]{$x$};
\filldraw[black] (13,0) circle (3pt) node[anchor=south]{$y$} ;
\draw (7.25,0) circle (8) ;
\draw (12.75,0) circle (8) ;
\draw[<->,dashed] (7.25,0) -- (12.75,0) ;
\filldraw[black] (10,0) circle (3pt) ;
\draw[->,dashed] (10,0) -- node[anchor=south east]{$s/2$} ++(45:4cm) ;
\draw[->,dashed] (13,0) -- node[anchor=south east]{$s$} ++(45:7.8cm) ;
\draw (10,0) circle (4) ;
\end{tikzpicture}
\end{center}
    \caption{Illustration of the friendly property for a simple graph $\mathbb{G}_{n,s}$.}
    \label{fig:geometry}
\end{figure}

\begin{lemma}
\label{lem:graphfriendly}
Assume that (i) for a $K>0$, (ii) and (iii) from Assumptions~\ref{assum} hold. Let $\Omega'$ be an open subset of $\Omega$ and $\delta>0$ be a constant. Assume that for any $x\in \Omega'$, $B_\delta(x)\subset \Omega$. For $n\in\mathbb{N}$ and $s>0$  define $\mathbb{G}'_{n,s}:=(V'_{n},E'_{n,s})$ to be the simple graph $\mathbb{G}_{n,s}$ intersecting the set $\Omega'$, that is $V_n':= V_n\cap \Omega'$ and also $ xy \in E'_{n,s}$ if and only if $xy \in E_{n,s}$ and $x,y\in V_n'$. Then, for any $s\in [4K\varepsilon_n,\delta)$, we have that $\mathbb{G}_{n,s}$ is $k$-friendly with respect to $\mathbb{G}'_{n,s}$ where 
$$ k= \frac{\pi^{d/2}s^d}{\Lambda_n^+D4^d\Gamma(\frac{d}{2}+1)}-2.$$
\end{lemma}

\begin{proof}
Fix $n\in\mathbb{N}$. If $V_n'$ is empty then the statement is trivial, so assume otherwise. Let $x,y\in V_n'$ be such that $xy\in E_{n,s}'$. Since $s<\delta$ we have that $B_s(x)\cup B_s(y) \subset \Omega$. Because $|x-y|\leq s$ we observe, as in Figure~\ref{fig:geometry}, there exists a ball $B$ of radius $s/2$ such that $B\subset B_s(x) \cap B_s(y)$.  Let $B^*$ be the ball $B$ but with half the radius. Since $$\frac{s}{4} \geq K\varepsilon_n\geq \sup_{x\in\Omega} |T_n(x) -x| ,$$   
 we observe $T_n(B^*)\subset B$.  We then compute
$$
     \Lambda_n^+|V_n\cap B| \geq \mu_n(B) \geq  \mu_n(T_n(B^*)) = \mu(T_n^{-1}(T_n(B^*))) \geq \mu(B^*) \geq \frac{1}{D} \mathcal{L}^d(B^*) = \frac{\pi^{d/2}s^d}{D4^d\Gamma(\frac{d}{2}+1)}.
 $$
 In particular
 $$| N_{\mathbb{G}_{n,s}}(x) \cap N_{\mathbb{G}_{n,s}}(y)|=  |V_n \cap B_s(x)\cap B_s(y)|-2 \geq |V_n\cap B|-2 \geq \frac{\pi^{d/2}s^d}{\Lambda_n^+D4^d\Gamma(\frac{d}{2}+1)}-2.$$   
 In the above computation, the term $-2$ results from the fact that $x$ is not counted among its own neighbours, nor is $y$, so we have to remove them from our count. The result now follows. 
\end{proof}

\begin{remark}
A relatively simple but useful observation is that, via the triangle inequality, for any $n\in\mathbb{N}$ and $s>0$, $\mathbb{G}_{n,2s}$ envelopes $\mathbb{G}_{n,s}$.
\end{remark}

It will also be useful to have a lower bound on the node degrees for $\mathbb{G}_{n,s}$.

\begin{lemma}
\label{lem:degreebound}
Assume that (i) for a $K>0$, (ii) and (iii) from Assumptions~\ref{assum} hold. 
Let $n\in\mathbb{N}$, $s\geq 2K\varepsilon_n$ and $x\in V_n$ be such that $B_s(x)\subset\Omega$. Then
$$ 1+|N_{\mathbb{G}_{n,s}}(x)| \geq \frac{\pi^{d/2}s^d}{\Lambda_n^+D2^d\Gamma(\frac{d}{2}+1)} .$$
\end{lemma}

\begin{proof}
According to our assumptions $\sup_{x\in\Omega} |T_n(x)-x|\leq K\varepsilon_n$, thus $T_n(B_{s/2}(x))\subset B_s(x)$.  Therefore
$$ \mu_n( B_s(x))= \mu(T_n^{-1}(B_s(x)))\geq \mu(B_{s/2}(x))\geq \frac{\pi^{d/2}s^d}{D2^d\Gamma(\frac{d}{2}+1)}.$$
Hence
$$ 1+|N_{\mathbb{G}_{n,s}}(x)| \geq  \frac{\mu_n(B_s(x))}{\Lambda_n^+}\geq \frac{\pi^{d/2}s^d}{\Lambda_n^+D2^d\Gamma(\frac{d}{2}+1)}.$$
\end{proof}

For our kernel $\eta$ and $t\geq0$ we define 
$$
\aleph^t_{\eta}:=\int_0^\infty \eta(r)r^{d+t-1} \, dr.
$$

In the statement of our results we may additionally assume there exists a particular $t\geq 1$ (with the particulars depending on the situation and hence specified per result) such that $\aleph_\eta^t<+\infty$ alongside some or all parts of Assumptions~\ref{assum}. 

We conclude this section with a lemma about the kernel $\eta$.

\begin{lemma}
\label{lem:weightbound}
Assume (i) for a $K>0$, (ii) and (iii) from Assumptions~\ref{assum} hold. Then there exists $\alpha_K>0$, a constant depending only on $K$, such that, for any $x\in \Omega$ and $n\in\mathbb{N}$, we have
$$ \frac{1}{\varepsilon_n^d}\sum_{y\in V_n} \eta\left( \frac{|x-y|}{\varepsilon_n} \right)\mu_n(y) \leq \alpha_K\aleph_{\eta}^0.$$

\end{lemma}

\begin{proof}

To simplify computations we will extend the function $\eta$'s domain to the whole $\mathbb{R}$ by defining $\eta(x)=\eta(0)$ for $x<0$.

As a consequence of (ii) in Assumptions~\ref{assum}, for any $r\in \mathbb{R}$, there exists a $\beta>0$ (dependent on $K$), such that $\eta(r-K)\leq 2\eta(r/\beta)$. Indeed, choose $\delta>0$ such that, for any $r\leq \delta$, $\eta(r)\geq \eta(0)/2$. Set $\beta=1+\frac{K}{\delta}$. If $r\geq K+\delta$, then we have
$$ r-K-r/\beta=\frac{\beta-1}{\beta}r-K=\frac{Kr}{\delta+K}-K\geq 0.$$
Thus $r-K\geq r/\beta$ and hence, by assumption (ii), $\eta(r-K)\leq \eta(r/\beta)\leq 2\eta(r/\beta)$. If $r\leq K+\delta$, then
$$ r/\beta \leq \frac{\delta+K}{1+K/\delta}=\delta \; \; \text{and} \; \; \eta(r/\beta)\geq \eta(0)/2. $$
So we have $\eta(r-K)\leq \eta(0) \leq 2\eta(r/\beta)$ as required.
The above bound combined with assumptions (i) and (iii) allows us to compute that
\begin{align*}
\frac{1}{\varepsilon_n^d}\sum_{y\in V_n} \eta\left( \frac{|x-y|}{\varepsilon_n} \right)\mu_n(y) &= \int_{\Omega} \frac{1}{\varepsilon_n^d}\eta\left( \frac{|x-y|}{\varepsilon_n}\right)d\mu_n(y) 
= \int_{\Omega} \frac{1}{\varepsilon_n^d}\eta\left( \frac{|x-T_n(y)|}{\varepsilon_n}\right)d\mu(y) \\
&= \int_{\Omega} \frac{1}{\varepsilon_n^d}\eta\left( \frac{|x-y+y-T_n(y)|}{\varepsilon_n}\right)d\mu(y) 
\leq \int_{\Omega} \frac{1}{\varepsilon_n^d}\eta\left( \frac{|x-y|}{\varepsilon_n}-K\right)d\mu(y) \\
&\leq  2D\int_{\mathbb{R}^d} \frac{1}{\varepsilon_n^d}\eta\left( \frac{|x-y|}{\beta\varepsilon_n}\right)dy= 2\beta^dD\int_{\mathbb{R}^d} \eta\left( |z|\right)dz . 
\end{align*}
In particular, the maps $T_n$ above are as in assumption (i). The result now follows.

\end{proof}

\section{Graph functional inequalities}
\label{sec:sobolev}

We start by introducing the graph norms and variations that will play a role in our graph functional inequalities.
We define $\mathbb{R}^{V_n}$ to be the set of functions $u:V_n\rightarrow \mathbb{R}$. For $p\geq1$ we define $L^p(V_n;\mu_n)$ to be $\mathbb{R}^{V_n}$ equipped with the norm
$$ \Vert u\Vert_{L^p(V_n;\mu_n)}:= \left( \sum_{x\in V_n} |u(x)|^p\mu_n(x)\right)^{1/p}.$$ 

For $p\geq 1$, $\varepsilon>0$ and $n\in \mathbb{N}$ we define the function $\mathcal{GE}^p[V_n,\mu_n;\varepsilon]:\mathbb{R}^{V_n}\rightarrow \mathbb{R}$ by
$$ \mathcal{GE}^p[V_n,\mu_n;\varepsilon](u):= \frac{1}{\varepsilon^p}\sum_{x\in V_n}\sum_{y\in V_n} \eta_{\varepsilon}(|x-y|)|u(x)-u(y)|^p\mu_n(x)\mu_n(y)$$
and $\mathcal{GE}^p_n:=\mathcal{GE}^p[V_n,\mu_n;\varepsilon_n]$.
We refer to these functions as discrete variations (or just variations).

Next we define a collection of functions on $L^p(\Omega;\mu)$ which will be useful in our analysis, as we may exploit their close relationship to the discrete variations. For a function $J:[0,+\infty)\rightarrow [0,+\infty)$  we define the nonlocal variation $\mathcal{F}^p[J;(\Omega,\mu)]:L^p(\Omega;\mu)\rightarrow L^p(\Omega;\mu)$ by
$$ \mathcal{F}^p[J;(\Omega,\mu)]:= \int_{\Omega}\int_{\Omega}J(|x-y|)|u(x)-u(y)|^p\rho(x)\rho(y) \, dx \, dy.$$

Next we define the local variations denoted by $\mathcal{E}^p:L^p(\Omega;\mu)\rightarrow \mathbb{R}$. In this case there is a notable difference between $p=1$ and $p>1$. For $p=1$, 
\begin{align}
\mathcal{E}^1(u) &:=  \sup \biggl\{\left. \int_{\Omega} u(x) \nabla \cdot \Phi(x) \, dx \; \right| \; \Phi \in C_c^{\infty}(\Omega;\mathbb{R}^d) \text{ and},
 \text{ for all } x\in \Omega, \, |\Phi(x)| \leq 1 \biggr\}.\notag 
\end{align}
Note $\Phi\equiv0$ is admissible in the above definition, hence $\mathcal{E}^1\geq 0$. Moreover if $u\in W^{1,1}(\Omega)$, then $\mathcal{E}^1(u)=\int_{\Omega}|\nabla u|  \; dx$. For $p>1$, we define
$$
\mathcal{E}^p(u):=
\begin{cases}
\int_\Omega |\nabla u|^p \, dx, &\text{if} \; u\in W^{1,p}(\Omega),\\
+\infty, &\text{otherwise}.
\end{cases}
$$

In this section, as is the usual convention, we use the notation $\lesssim_n$ to denote less than or equal to up to a positive constant factor independent of the parameter $n\in \mathbb{N}$. Our objective is to produce uniform bounds independent of $n\in\mathbb{N}$ and so we shall make frequent use of this notation. It is useful to observe that, assuming part~(ii) from Assumptions~\ref{assum} holds, for a constant $a\in (0,1]$, $n\in \mathbb{N}$ and $u\in \mathbb{R}^{V_n}$, we have
\begin{equation}\label{eq:aepsilon}
    \mathcal{GE}^p[V_n,\mu_n;a\varepsilon_n](u)\lesssim_n\mathcal{GE}_n^p(u).
\end{equation}
This follows, since $\eta$ is non-increasing and thus $\eta_{a\varepsilon_n}(r) = \frac1{(a\varepsilon_n)^d} \eta\left(\frac{r}{a\varepsilon_n}\right) \leq \frac1{(a\varepsilon_n)^d} \eta\left(\frac{r}{\varepsilon_n}\right) = \frac1{a^d} \eta_{\varepsilon_n}(r) \lesssim_n \eta_{\varepsilon_n}(r)$.

We also define the notation $\lesssim_\varepsilon$ to denote less than or equal up to a positive constant factor independent of the parameter $\varepsilon>0$. We shall make regular use of this notation in the proof of Lemma~\ref{lem:ext1}.
 
Throughout this paper, for a given Borel set $A\subset \mathbb{R}^d$ and function $u:A\rightarrow \mathbb{R}$, we will use the notation $\Vert u \Vert_{L^p(A)}$ to denote the standard $L^p$ norm with respect to the Lebesgue measure on $A$. If instead we want to indicate integration with respect to a specific measure $\nu$ on $A$, then we will use the notation $\Vert u \Vert_{L^p(A;\nu)}$.

\subsection{Extension results}\label{sec:extensionresults}

In Sobolev space theory a useful tool is the theory of extensions. Typically, we want to extend a given function, defined on $\Omega$, to a larger subset $\Sigma\supset\Omega$, in such a way that regularity properties are preserved. In Lemma~\ref{lem:ext1} we do exactly this to uniformly control, over $\varepsilon>0$, the nonlocal variations given by $\frac{1}{\varepsilon^p}\mathcal{F}^p[\eta_{\varepsilon};(\Omega,\mu)]$. In Lemma~\ref{lem:ext2}, our goal is to develop a theory of extensions in the discrete setting. Specifically, we want to extend both our graphs and discrete functions on said graphs in such a way that the discrete variation given by $\mathcal{GE}^p_n$ is uniformly controlled, over $n\in\mathbb{N}$.       

We require these extension results to ensure that certain other properties, which only hold away from the boundary, can still be utilized on all of $\Omega$; see Section~\ref{sec:extension} for more details.

\begin{lemma}
\label{lem:ext1}
Let $p\geq 1$. Assume (ii) and (iii) from Assumptions~\ref{assum} hold and there exists $q>p$ such that $\aleph_\eta^q<+\infty$. Then there exists the following. 
\begin{enumerate}[(1)]
\item An open set $\Sigma\supset \Omega$ and $\delta>0$ such that $\cup_{x\in\Omega}B_{\delta}(x)\subset\Sigma$.
    \item A probability measure $\widetilde{\mu}$ with density $\widetilde{\rho}$ on $\Sigma$ such that there exists $\kappa>0$ with $\frac{1}{\kappa}\leq\widetilde{\rho}\leq\kappa$. Additionally $\widetilde{\rho}|_{\Omega}$ is proportional to $\rho$.
    \item An extension operator $\mathfrak{X}:L^p(\Omega;\mu)\rightarrow L^p(\Sigma;\widetilde{\mu})$ such that, for any $f\in L^p(\Omega;\mu)$, $\mathfrak{X}f|_{\Omega}=f$. Moreover there exists constants $A>0$, $B\geq 1$, $C>0$ such that, for any $\varepsilon>0$ and any $f\in L^p(\Omega;\mu)$, we have 
$$\frac{1}{\varepsilon^p}\mathcal{F}^p[\eta_\varepsilon;(\Sigma,\widetilde{\mu})](\mathfrak{X}f)\leq A\left(\Vert f\Vert^p_{L^p(\Omega;\mu)}+\frac{1}{B^p\varepsilon^p}\mathcal{F}^p[\eta_{B\varepsilon};(\Omega,\mu)](f)\right)$$
and
$$ \Vert \mathfrak{X}f \Vert_{L^p(\Sigma;\widetilde{\mu})}\leq C \Vert f \Vert_{L^p(\Omega;\mu)}.$$
\end{enumerate}

\end{lemma}

\begin{proof}
First we construct a partition of unity. Since $\overline{\Omega}$ is a Lipschitz manifold with boundary (see Remark~\ref{rem:Lipschitz}), for each $x\in\partial\Omega$, there exists an open neighbourhood $U_x\subset \mathbb{R}^d$ and a bi-Lipschitz homeomorphism $\phi_x:U_x\rightarrow \mathcal{B}$ where $\mathcal{B}$ denotes the unit ball around $0$ in $\mathbb{R}^d$, $\phi_x^{-1}(\mathcal{B}\cap\{x_d<0\})=\Omega\cap U_x$, $\phi_x^{-1}(\mathcal{B}\cap\{x_d=0\})=\partial\Omega\cap U_x$ and $\phi_x^{-1}(\mathcal{B}\cap\{x_d>0\})=(\mathbb{R}^d\setminus\Omega)\cap U_x$.

Since $\partial\Omega$ is compact with an open cover $\{U_x|x\in\partial\Omega\}$, we can then pass to a finite subcover $\{U_i|i\in I\}$ which we index by $i\in I$, with corresponding bi-Lipschitz homeomorphisms $\phi_i: U_i \to \mathcal{B}$. As $I$ is finite, we can choose a constant $L\geq 1$ such that for each $i\in I$, $\phi_i$ and $\phi_i^{-1}$ are Lipschitz with constant $L$. Let $J_i^{-1}$ be the Jacobian (determinant) of $\phi_i^{-1}$ as in \cite[Definition 3.6]{Evans1992}. It will be useful to observe $J_i^{-1}$ is bounded above by $L^d$ and bounded below by $L^{-d}$, given that $\phi_i^{-1}$ and $\phi_i$ have Lipschitz constant $L$. In future calculations we shall make use of the following change of variables $(\mathbb{X},\mathbb{Y})=(\phi_i(x),\phi_i(y))$. According to the change of variables formula in \cite[Theorem 3.9]{Evans1992}, for any function $g:U_i\times U_i \rightarrow \mathbb{R}$,
$$  \int_{U_i}\int_{U_i} g(x,y) \, dx \, dy = \int_{\mathcal{B}}\int_{\mathcal{B}} g\left(\phi_i^{-1}(\mathbb{X} ),\phi_i^{-1}(\mathbb{X} )\right)J_i^{-1}(\mathbb{X})J_i^{-1}(\mathbb{Y}) \; \, d\mathbb{X} \, d\mathbb{Y}.$$ 

Set $U_0:=\Omega$ and $I_0:=I\cup\{0\}$. We now define $\Sigma := \cup_{i\in I_0}U_i$. Since $\mathbb{R}^d\setminus\Sigma$ is closed and disjoint from $\overline{\Omega}$, there exists $\delta>0$ such that
$$ \inf_{x\in\overline{\Omega}} \; \;  \inf_{y\in \mathbb{R}^d\setminus\Sigma} |x-y| > \delta >0 .$$
In other words $\cup_{x\in \Omega} B_\delta(x)\subset \Sigma$; this handles part (1). As constructed $\Sigma$ is an open subset of $\mathbb{R}^d$ and thus is a smooth manifold (without boundary). Moreover, $\{U_i |i\in I_0\}$ is an open cover of $\Sigma$. Therefore, we can apply \cite[Theorem 2.23]{Lee2013} 
to construct a smooth partition of unity $\{\varphi_i\}_{i\in I_0}$ of $\Sigma$ that is subordinate to $\{U_i |i\in I_0\}$ with the following properties:
\begin{itemize}
    \item for any $i\in I_0$, $x\in \Sigma$ we have $0\leq\varphi_i(x)\leq 1$,
    \item for any $i\in I$, $\operatorname{supp}\varphi_i \subset U_i$ and $\operatorname{supp}\varphi_0 \subset \Omega$,
    \item for any $x\in\Sigma$, we have $\sum_{i\in I_0} \varphi_i(x)=1$.
\end{itemize}

For $i\in I_0$ and $r>0$ define $U_i[r]$ to be the set of points within $U_i$ a distance at most $r$ (inclusive) from the boundary $\partial U_i$. Any standard construction of a smooth partition of unity subordinate to a finite open cover (see the proof of \cite[Theorem 2.23]{Lee2013}) has the following property: there exist $a,b>0$ such that
$$ \sup_{x\in U_i[r]} |\varphi_i(x)| \leq ae^{-b/r}.$$
Again, since $I_0$ is finite, we can assume the constants $a,b$ are independent of $i\in I_0$.

Since each $\varphi_i$ is a smooth function with compact support and $I_0$ is finite we can choose $M>0$ such that each $\varphi_i$ is Lipschitz with constant $M$.

With this partition of unity we define the following reflection maps. Let $R:\mathbb{R}^d\rightarrow \mathbb{R}^d$ be defined by $R(x):=(x_1,x_2,\ldots,x_{d-1},-x_d)$, then for each $i\in I$ define the bijection $S_i:U_i\cap\Omega\rightarrow U_i\setminus\overline{\Omega}$ by $S_i:= \phi_i^{-1}\circ R\circ \phi_i$.

Next we construct our extended probability measure on $\Sigma$. For each $i\in I$, define $\widetilde{\rho_i}:U_i\rightarrow (0,+\infty)$ by
$\widetilde{\rho_i}(x):=\rho(x)$ if $x\in\Omega$ and $\widetilde{\rho_i}(x):=\rho(S_i^{-1}(x))$ if $x\in U_i\setminus\overline{\Omega}$ (note this is defined everywhere except on $\partial\Omega$ which has Lebesgue measure $0$). For convenience we denote $\rho_0:=\rho$. Finally we take
$$ \widetilde{\rho}(x):= \frac{\sum_{i\in I_0}\varphi_i(x)\widetilde{\rho_i}(x)}{\int_{\Sigma}\sum_{i\in I_0}\varphi_i(y)\widetilde{\rho_i}(y)dy}.$$

From this formula one can observe, for all $x\in \Sigma$,
$$ \frac{1}{D^2\mathfrak{L}^d(\Sigma)}\leq\widetilde{\rho}(x)\leq \frac{ D^2}{\mathfrak{L}^d(\Sigma)}.$$
Moreover, $\widetilde{\rho}|_{\Omega}$ is proportional to $\rho$ as required. We can now define $\widetilde{\mu}:=\widetilde{\rho}\mathfrak{L}^d|_{\Sigma}$. 

The final step is to construct an extension operator $\mathfrak{X}:L^p(\Omega;\mu)\rightarrow L^p(\Sigma;\widetilde{\mu})$ with our required regularity properties. To this end, for a given $f\in L^p(\Omega;\mu)$ and $i\in I$, define $f_i:\Sigma\rightarrow \mathbb{R}$, by $f_i(x):=f(x)$ if $x\in U_i\cap\Omega$,  $f_i(x):=f(S_i^{-1}(x))$ if $x\in U_i\setminus\overline{\Omega}$ and $f_i(x)=0$ if $x\in \Sigma\setminus U_i$ (again this is defined everywhere except $\partial\Omega$ which has Lebesgue measure $0$). Additionally we denote $f_0:=f$. We then define, for $x\in \Sigma$,
$$\mathfrak{X}f(x):= \sum_{i\in I_0}\varphi_i(x)f_i(x).$$
From this construction we have, for each $i\in I_0$, $\Vert f_i \Vert^p_{L^p(U_i)}\lesssim_{\varepsilon}\Vert f \Vert^p_{L^p(\Omega)}$ and additionally $\Vert \mathfrak{X}f\Vert^p_{L^p(\Sigma)}\lesssim_{\varepsilon}\Vert f\Vert^p_{L^p(\Omega)}$. This justifies the second bound in part (3).

It remains to justify our regularity control. Let $\varepsilon>0$, then 
\begin{align*}
&\frac{1}{\varepsilon^p}\mathcal{F}^p[\eta_\varepsilon;(\Sigma,\widetilde{\mu})](\mathfrak{X}f)= \int_{\Sigma}\int_{\Sigma} \frac{1}{\varepsilon^{d+p}}\eta\left(\frac{|x-y|}{\varepsilon}\right)|\mathfrak{X}f(x)-\mathfrak{X}f(y)|^p\widetilde{\rho}(y)\widetilde{\rho}(x)\, dx \, dy \\
&=\int_{\Sigma}\int_{\Sigma} \frac{1}{\varepsilon^{d+p}}\eta\left(\frac{|x-y|}{\varepsilon}\right)\left|\sum_{i\in I_0}(\varphi_i(x)f_i(x)-\varphi_i(y)f_i(y))\right|^p\widetilde{\rho}(y)\widetilde{\rho}(x)\, dx \, dy \\
&\lesssim_{\varepsilon} \sum_{i\in I_0}\int_{\Sigma}\int_{\Sigma} \frac{1}{\varepsilon^{d+p}}\eta\left(\frac{|x-y|}{\varepsilon}\right)\left|\varphi_i(x)f_i(x)-\varphi_i(y)f_i(y)\right|^p\widetilde{\rho}(y)\widetilde{\rho}(x)\, dx \, dy \\
&\lesssim_{\varepsilon} \sum_{i\in I_0}\int_{\Sigma}\int_{\Sigma} \frac{1}{\varepsilon^{d+p}}\eta\left(\frac{|x-y|}{\varepsilon}\right)\left|\varphi_i(x)f_i(x)-\varphi_i(y)f_i(x)\right|^p\widetilde{\rho}(y)\widetilde{\rho}(x)\, dx \, dy \\
&\hspace{0.4cm} + \sum_{i\in I_0}\int_{\Sigma}\int_{\Sigma} \frac{1}{\varepsilon^{d+p}}\eta\left(\frac{|x-y|}{\varepsilon}\right)\left|\varphi_i(y)f_i(x)-\varphi_i(y)f_i(y)\right|^p\widetilde{\rho}(y)\widetilde{\rho}(x)\, dx \, dy\\
&= \sum_{i\in I_0}\left(\int_{\Sigma}\int_{U_i} \frac{1}{\varepsilon^{d+p}}\eta\left(\frac{|x-y|}{\varepsilon}\right)\left|\varphi_i(x)-\varphi_i(y)\right|^p|f_i(x)|^p \, dx \, dy \right. \\
&\hspace{1.4cm}+ \left.\int_{U_i}\int_{\Sigma} \frac{1}{\varepsilon^{d+p}}\eta\left(\frac{|x-y|}{\varepsilon}\right)\left|f_i(x)-f_i(y)\right|^p |\varphi_i(y)|^p \, dx \, dy\right) \\
&= \sum_{i\in I_0}\left(\int_{\Sigma}\int_{U_i} \frac{1}{\varepsilon^{d+p}}\eta\left(\frac{|x-y|}{\varepsilon}\right)\left|\varphi_i(x)-\varphi_i(y)\right|^p|f_i(x)|^p \, dx \, dy \right. \\
&\hspace{1.4cm}+ \int_{U_i}\int_{U_i} \frac{1}{\varepsilon^{d+p}}\eta\left(\frac{|x-y|}{\varepsilon}\right)\left|f_i(x)-f_i(y)\right|^p |\varphi_i(y)|^p \, dx \, dy\\
&\hspace{2.4cm}+ \left. \int_{ U_i}\int_{\Sigma \setminus U_i} \frac{1}{\varepsilon^{d+p}}\eta\left(\frac{|x-y|}{\varepsilon}\right)\left|f_i(y)\right|^p |\varphi_i(y)|^p \, dx \, dy \right).
\end{align*}

For $i\in I_0$ define $\mathcal{J}_1(i)$, $\mathcal{J}_2(i)$, $\mathcal{J}_3(i)$ as the three integrals in the final line of the above computation, respectively. We will bound each separately. 

\begin{enumerate}
    \item[($\mathcal{J}_1$)] We have
    \begin{align*}
        \mathcal{J}_1(i)&= \int_{\Sigma}\int_{U_i} \frac{1}{\varepsilon^{d+p}}\eta\left(\frac{|x-y|}{\varepsilon}\right)\left|\varphi_i(x)-\varphi_i(y)\right|^p|f_i(x)|^p \, dx \, dy \\
        &\lesssim_{\varepsilon}  \int_{\Sigma}\int_{U_i} \frac{1}{\varepsilon^d}\eta\left(\frac{|x-y|}{\varepsilon}\right)\left| \frac{x-y}{\varepsilon}\right|^p|f_i(x)|^p \, dx \, dy %\\ &
        \leq  \int_{\mathbb{R}^d}\int_{U_i} \frac{1}{\varepsilon^d}\eta\left(\frac{|x-y|}{\varepsilon}\right)\left| \frac{x-y}{\varepsilon}\right|^p|f_i(x)|^p \, dx \, dy \\
        &= \int_{U_i}\int_{\mathbb{R}^d} \eta(|z|)|z|^p|f_i(x)|^p \, dz \, dx %\\ &
        \lesssim_{\varepsilon}\aleph_{\eta}^p \Vert f_i \Vert^p_{L^p(U_i)}.
    \end{align*}
    \item[($\mathcal{J}_2$)] Using the transformation of variables previously introduced we deduce that
    \begin{align*}
        \mathcal{J}_2(i)&= \int_{\mathcal{B}}\int_{\mathcal{B}} \frac{1}{\varepsilon^{d+p}}\eta\left(\frac{|\phi_i^{-1}(\mathbb{X})-\phi_i^{-1}(\mathbb{Y})|}{\varepsilon}\right)\left|f_i(\phi_i^{-1}(\mathbb{X}))-f_i(\phi_i^{-1}(\mathbb{Y}))\right|^p J_i^{-1}(\mathbb{X})J_i^{-1}(\mathbb{Y})\, d\mathbb{X} \, d\mathbb{Y} \\
        &\lesssim_{\varepsilon}  \int_{\mathcal{B}}\int_{\mathcal{B}} \frac{1}{\varepsilon^{d+p}}\eta\left(\frac{|\mathbb{X}-\mathbb{Y}|}{L\varepsilon}\right)\left|f_i(\phi_i^{-1}(\mathbb{X}))-f_i(\phi_i^{-1}(\mathbb{Y}))\right|^p \, d\mathbb{X} \, d\mathbb{Y} .
    \end{align*}
    The second line follows because $|\phi_i^{-1}(\mathbb{X})-\phi_i^{-1}(\mathbb{Y})|\geq |\mathbb{X}-\mathbb{Y}|/L$ and $\eta$ is non-increasing. If we denote $\mathcal{B}^-:=\{x\in\mathcal{B} \; | \; x_d<0\}$ and $\mathcal{B}^+:=\{x\in\mathcal{B} \; | \; x_d>0\}$, the integral splits into four parts:
    $$ \int_{\mathcal{B}}\int_{\mathcal{B}}=\int_{\mathcal{B}^-}\int_{\mathcal{B}^-}+\int_{\mathcal{B}^+}\int_{\mathcal{B}^+} +\int_{\mathcal{B}^-}\int_{\mathcal{B}^+}+\int_{\mathcal{B}^+}\int_{\mathcal{B}^-}.$$
    Additionally observe that 
    $$f_i(\phi_i^{-1}(\mathbb{X}))= \begin{cases}
        f(\phi_i^{-1}(\mathbb{X})), \; \; &\text{if} \; \; \mathbb{X}\in \mathcal{B}^- ,\\
        f(\phi_i^{-1}(R\mathbb{X})), \; \; &\text{if} \; \; \mathbb{X}\in \mathcal{B}^+.
    \end{cases}$$
    We then compute
    \begin{align*}
    &\int_{\mathcal{B}^+}\int_{\mathcal{B}^+}  \frac{1}{\varepsilon^{d+p}}\eta\left(\frac{|\mathbb{X}-\mathbb{Y}|}{L\varepsilon}\right)\left|f_i(\phi_i^{-1}(\mathbb{X}))-f_i(\phi_i^{-1}(\mathbb{Y}))\right|^p \, d\mathbb{X} \, d\mathbb{Y} \\
    = &\int_{\mathcal{B}^+}\int_{\mathcal{B}^+}  \frac{1}{\varepsilon^{d+p}}\eta\left(\frac{|\mathbb{X}-\mathbb{Y}|}{L\varepsilon}\right)\left|f(\phi_i^{-1}(R\mathbb{X}))-f(\phi_i^{-1}(R\mathbb{Y}))\right|^p \, d\mathbb{X} \, d\mathbb{Y} \\
    = &\int_{\mathcal{B}^+}\int_{\mathcal{B}^+}  \frac{1}{\varepsilon^{d+p}}\eta\left(\frac{|R\mathbb{X}-R\mathbb{Y}|}{L\varepsilon}\right)\left|f(\phi_i^{-1}(R\mathbb{X}))-f(\phi_i^{-1}(R\mathbb{Y}))\right|^p \, d\mathbb{X} \, d\mathbb{Y} \\
    = &\int_{\mathcal{B}^-}\int_{\mathcal{B}^-}  \frac{1}{\varepsilon^{d+p}}\eta\left(\frac{|\mathbb{X}-\mathbb{Y}|}{L\varepsilon}\right)\left|f(\phi_i^{-1}(\mathbb{X}))-f(\phi_i^{-1}(\mathbb{Y}))\right|^p \, d\mathbb{X} \, d\mathbb{Y} .
    \end{align*}
    Observe that for $\mathbb{X},\mathbb{Y}\in\mathcal{B}^-$ we have $|\mathbb{X}-\mathbb{Y}|\leq |\mathbb{X}-R\mathbb{Y}|$, hence 
    \begin{align*}
    &\int_{\mathcal{B}^+}\int_{\mathcal{B}^-}  \frac{1}{\varepsilon^{d+p}}\eta\left(\frac{|\mathbb{X}-\mathbb{Y}|}{L\varepsilon}\right)\left|f_i(\phi_i^{-1}(\mathbb{X}))-f_i(\phi_i^{-1}(\mathbb{Y}))\right|^p \, d\mathbb{X} \, d\mathbb{Y} \\
    = &\int_{\mathcal{B}^+}\int_{\mathcal{B}^-}  \frac{1}{\varepsilon^{d+p}}\eta\left(\frac{|\mathbb{X}-\mathbb{Y}|}{L\varepsilon}\right)\left|f(\phi_i^{-1}(\mathbb{X}))-f(\phi_i^{-1}(R\mathbb{Y}))\right|^p \, d\mathbb{X} \, d\mathbb{Y} \\
    = &\int_{\mathcal{B}^-}\int_{\mathcal{B}^-}  \frac{1}{\varepsilon^{d+p}}\eta\left(\frac{|\mathbb{X}-R\mathbb{Y}|}{L\varepsilon}\right)\left|f(\phi_i^{-1}(\mathbb{X}))-f(\phi_i^{-1}(\mathbb{Y}))\right|^p \, d\mathbb{X} \, d\mathbb{Y} \\
    \leq &\int_{\mathcal{B}^-}\int_{\mathcal{B}^-}  \frac{1}{\varepsilon^{d+p}}\eta\left(\frac{|\mathbb{X}-\mathbb{Y}|}{L\varepsilon}\right)\left|f(\phi_i^{-1}(\mathbb{X}))-f(\phi_i^{-1}(\mathbb{Y}))\right|^p \, d\mathbb{X} \, d\mathbb{Y} .
    \end{align*}
    All together we get
    \begin{align*}
        \mathcal{J}_2(i)&\lesssim_{\varepsilon} \int_{\mathcal{B}^-}\int_{\mathcal{B}^-}  \frac{1}{\varepsilon^{d+p}}\eta\left(\frac{|\mathbb{X}-\mathbb{Y}|}{L\varepsilon}\right)\left|f(\phi_i^{-1}(\mathbb{X}))-f(\phi_i^{-1}(\mathbb{Y}))\right|^p \, d\mathbb{X} \, d\mathbb{Y} \\
        &\leq  \int_{\mathcal{B}^-}\int_{\mathcal{B}^-}  \frac{1}{\varepsilon^{d+p}}\eta\left(\frac{|\phi_i^{-1}(\mathbb{X})-\phi_i^{-1}(\mathbb{Y})|}{L^2\varepsilon}\right)\left|f(\phi_i^{-1}(\mathbb{X}))-f(\phi_i^{-1}(\mathbb{Y}))\right|^p \, d\mathbb{X} \, d\mathbb{Y} \\
        &\lesssim_{\varepsilon} \int_{\mathcal{B}^-}\int_{\mathcal{B}^-}  \frac{1}{\varepsilon^{d+p}}\eta\left(\frac{|\phi_i^{-1}(\mathbb{X})-\phi_i^{-1}(\mathbb{Y})|}{L^2\varepsilon}\right)\left|f(\phi_i^{-1}(\mathbb{X}))-f(\phi_i^{-1}(\mathbb{Y}))\right|^p J_i^{-1}(\mathbb{X})J_i^{-1}(\mathbb{Y}) \, d\mathbb{X} \, d\mathbb{Y} \\
        &= \int_{U_i\cap \Omega}\int_{U_i\cap \Omega}  \frac{1}{\varepsilon^{d+p}}\eta\left(\frac{|x-y|}{L^2\varepsilon}\right)\left|f(x)-f(y)\right|^p \, dx \, dy \\
        &\lesssim_{\varepsilon} \frac{1}{\varepsilon^p}\mathcal{F}^p[\eta_{L^2\varepsilon};(\Omega,\mu)](f).
    \end{align*}
    For the equality we used that $\phi_i^{-1}(\mathcal{B}^-) = U_i \cap \Omega$.
    
    \item[($\mathcal{J}_3$)] The third integral, which we shall now bound, appears because the variation we are trying to control is nonlocal. This term, at first glance, looks particularly nasty. This is because we have no means of controlling the nonlocal variation of $f_i$ as it crosses the boundary of $U_i$. To overcome this difficulty, we can use the fact that the function $\varphi_i$ decays very quickly at the boundary of $U_i$. Since $\aleph_{\eta}^q<+\infty$ with $q$ as specified in the statement of the lemma (importantly, $q>p$), we have $\eta(r)\lesssim_{\varepsilon} 1/r^{d+q}$. This is precisely the part of the proof where we require the assumption $\aleph_{\eta}^q<+\infty$. 
    
    Fix $\xi:= 1/\varepsilon^{\frac{d+p}{d+q}}$ and $\gamma:=1-\frac{d+p}{d+q}>0$. Then $\eta(\xi)\lesssim_{\varepsilon}\varepsilon^{d+p}$. We compute
    \begin{align*}
        \mathcal{J}_3(i) &= \int_{U_i}\int_{\Sigma \setminus U_i} \frac{1}{\varepsilon^{d+p}}\eta\left(\frac{|x-y|}{\varepsilon}\right)\left|f_i(y)\right|^p |\varphi_i(y)|^p \, dx \, dy \\
        &= \int_{U_i}\int_{\Sigma \setminus U_i} \frac{1}{\varepsilon^{d+p}}\eta\left(\frac{|x-y|}{\varepsilon}\right)\left( \mathbf{1}_{|x-y|\leq \xi \varepsilon} +\mathbf{1}_{|x-y|> \xi \varepsilon}\right)\left|f_i(y)\right|^p |\varphi_i(y)|^p \, dx \, dy \\
        &= \int_{U_i}\int_{\Sigma \setminus U_i} \frac{1}{\varepsilon^{d+p}}\eta\left(\frac{|x-y|}{\varepsilon}\right) \mathbf{1}_{|x-y|\leq \xi \varepsilon}\left|f_i(y)\right|^p |\varphi_i(y)|^p \, dx \, dy \\
        &\hspace{0.4cm}+ \int_{U_i}\int_{\Sigma \setminus U_i} \frac{1}{\varepsilon^{d+p}}\eta\left(\frac{|x-y|}{\varepsilon}\right) \mathbf{1}_{|x-y|> \xi \varepsilon}\left|f_i(y)\right|^p |\varphi_i(y)|^p \, dx \, dy \\
        &\lesssim_{\varepsilon} \int_{U_i[\varepsilon\xi]}\int_{\Sigma \setminus U_i} \frac{1}{\varepsilon^{d+p}}\eta(0)\left|f_i(y)\right|^p |\varphi_i(y)|^p \, dx \, dy + \int_{U_i}\int_{\Sigma \setminus U_i} \frac{1}{\varepsilon^{d+p}}\eta(\xi)\left|f_i(y)\right|^p |\varphi_i(y)|^p \, dx \, dy \\  
        &\lesssim_{\varepsilon} \frac{1}{\varepsilon^{d+p}}\sup_{y\in U_i[\varepsilon\xi]} \left\{ |\varphi_i(y)|^p \right\} \Vert f_i \Vert^p_{L^p(U_i)} + \frac{\eta(\xi)}{\varepsilon^{d+p}} \Vert f_i \Vert^p_{L^p(U_i)} \\
        &\lesssim_{\varepsilon} \frac{e^{-bp/\varepsilon\xi}+\eta(\xi)}{\varepsilon^{d+p}}\Vert f \Vert^p_{L^p(\Omega;\mu)} \lesssim_{\varepsilon} \left(\frac{e^{-bp\varepsilon^{-\gamma}}}{\varepsilon^{d+p}} +1\right)\Vert f \Vert^p_{L^p(\Omega;\mu)} \lesssim_{\varepsilon} \Vert f \Vert^p_{L^p(\Omega;\mu)}.
    \end{align*}

\end{enumerate}

We now conclude
\begin{align*}
\frac{1}{\varepsilon^p}\mathcal{F}^p[\eta_\varepsilon;(\Sigma,\widetilde{\mu})](\mathfrak{X}f) &\lesssim_{\varepsilon}\sum_{i\in I_0} \mathcal{J}_1(i)+\mathcal{J}_2(i) + \mathcal{J}_3(i) %\\ &
\lesssim_{\varepsilon} \Vert f\Vert^p_{L^p(\Omega;\mu)} +  \frac{1}{\varepsilon^p}\mathcal{F}^p[\eta_{L^2\varepsilon};(\Omega,\mu)](f).
\end{align*}
Here we used that $\|f_i\|_{L^p(U_i)}^p \lesssim_\varepsilon \|f\|_{L^p(\Omega;\mu)}^p$.

Taking $B:=L^2$ and $A$ suitably we are done.
\end{proof} 

In Lemma~\ref{lem:discretization} we shall construct a discretization of the set $\Sigma\setminus\overline{\Omega}$ with nice properties. This will be useful for our discrete extension results. Before we proceed with this lemma we shall first prove a useful result that will be handy for the construction. 

\begin{lemma}
\label{lem:balltocube}
There exists a bi-Lipschitz homeomorphism between the unit cube and the unit hemisphere in $\mathbb{R}^d$.
\end{lemma}

\begin{proof}
As is the case throughout this paper, $|\cdot|$ denotes the standard euclidean norm on $\mathbb{R}^d$. Let $\mathcal{H}$ be the unit hemisphere, i.e. $\mathcal{H}:= \{ x\in \mathbb{R}^d \; | \; |x|\leq 1 \; \text{and} \; x_d\geq 0 \}$ and $\mathcal{C}=[0,1]^d$ be the unit cube. We also define the norm, for $x\in\mathbb{R}^d$,
$ |x|_\infty= \max_{i\in\{1,\ldots,d\}} |x_i|.$ Define $F:\mathbb{R}^d\rightarrow(-1,1)^d$ such that
$ F(x) = \frac{x}{1+|x|_\infty}$. Note that $F$ is a bijection with an inverse given by $F^{-1}(x)=\frac{x}{1-|x|_\infty}$. Observe that $F$ maps $\mathcal{H}$ to $[-\frac{1}{2},\frac{1}{2}]^{d-1}\times[0,\frac{1}{2}]$. Moreover $F|_\mathcal{H}$ is a product of bounded Lipschitz functions and therefore is Lipschitz. The same is true for $F^{-1}$ when restricted to $[-\frac{1}{2},\frac{1}{2}]^{d-1}\times[0,\frac{1}{2}]$. Therefore $F|_{\mathcal{H}}$ is a bi-Lipschitz homeomorphism from $\mathcal{H}$ to $[-\frac{1}{2},\frac{1}{2}]^{d-1}\times[0,\frac{1}{2}]$. There is also a bi-Lipschitz  homeomorphism from $[-\frac{1}{2},\frac{1}{2}]^{d-1}\times[0,\frac{1}{2}]$ to $\mathcal{C}$ via an affine transformation. Composing these together then concludes the proof.   
\end{proof}

\begin{lemma}
\label{lem:discretization}
Let $\Sigma$ be the set constructed in the proof of Lemma~\ref{lem:ext1} and $\tilde \mu$ the corresponding measure. Define $\nu$ to be the probability measure $\tilde{\mu}$ restricted to $\Sigma\setminus\overline{\Omega}$ and normalized. There exists constants $a,b,c,d>0$ such that, for any $n\in \mathbb{N}$, there exists a Borel map $M_n:\Sigma\setminus\overline{\Omega} \rightarrow \Sigma\setminus\overline{\Omega}$ where,  $|M_n(\Sigma\setminus\overline{\Omega})|\leq an$, $\sup_{x\in\Omega}|M_n(x)-x|\leq b/n^{1/d}$ and, for any $x\in M_n(\Sigma\setminus\overline{\Omega})$,
$$ \frac{c}{n}\leq M_n\#\nu(x) \leq \frac{d}{n}.$$
\end{lemma}

\begin{proof}

Take $\{U_i\}_{i\in I}$ and $\{\phi_i\}_{i\in I}$ as within the proof of Lemma~\ref{lem:ext1}. Define $U_i^+:= U_i\cap (\Sigma\setminus\overline{\Omega})$ and observe that $\{U^+_i\}_{i\in I}$ is an open cover of $\Sigma\setminus\overline{\Omega}$. Restricted to $U_i^+$, $\phi_i$ defines a bi-Lipschitz homeomorphism to $\mathcal{B}^+:= \{x\in\mathcal{B} \; | \; x_d>0\}$. We then apply Lemma~\ref{lem:balltocube} and choose a bi-Lipschitz homeomorphism $\vartheta: \mathcal{B}^+\rightarrow (0,1)^d$. Define $\Phi_i:=\vartheta\circ\phi_i|_{U_i^+}:U_i^+\rightarrow (0,1)^d$ and choose a constant $\widetilde{L}$ such that $\Phi_i$ and $\Phi_i^{-1}$ are both $\widetilde{L}$-Lipschitz. 

Consider a standard lattice discretization of $(0,1)^d$ onto $\left( \lceil n^{1/d}\rceil \right)^d\leq 2^dn$ points (where we used that $\lceil x \rceil \leq x+1 \leq 2x$ for $x\geq 1$). This we present in full detail within example 4. of Section~\ref{sec:applications} and Figure~\ref{fig:lattice}. We shall take $T_n:(0,1)^d\rightarrow (0,1)^d$ to be the projection onto this discretization, again this is the same map as specified within example 4. of Section~\ref{sec:applications}. Note that 
$$ \sup_{x\in (0,1)^d} |T_n(x)-x|  \leq \frac{1}{ \lceil n^{1/d}\rceil}.$$ 

Let $\nu$ be the probability measure as defined in the statement of the lemma. Construct a random variable $X$ with law $\nu$. Additionally enumerate the index set $I$ as $\{1,\ldots,l\}$ where $l$ is the cardinality of $I$; consider a random variable $\xi$ that is a uniform random permutation of $\{1,\ldots,l\}$ and define a random variable $\zeta$, taking values in $\{1,\ldots,l\}$, to be the unique $i\in\{1,\ldots,l\} $ such that $X\in U^+_i$ and $\xi(i)$ is maximized.

We then define a discrete random variable $Y_n$ taking values in $\Sigma\setminus\overline{\Omega}$ as follows:
$$ Y_n := \Phi_\zeta^{-1}\circ T_n \circ \Phi_{\zeta}(X).$$
For convenience we shall denote $\Pi$ to be the probability space on which the random variables $X,Y_n,\xi,\zeta$ have all been constructed. Note that $Y_n$ is well defined everywhere except on a set of $\Pi$-probability zero. Let $\nu_n$ be the law of $Y_n$ and observe that the support of $\nu_n$ has at most $2^dln$ points.

Next we observe that, with $\Pi$-probability one,
\begin{align*}
|Y_n-X|&\leq \widetilde{L}|\Phi_\zeta(Y_n)-\Phi_\zeta(X)| %\\ &
= \widetilde{L}|T_n\circ\Phi_\zeta(X)-\Phi_\zeta(X)| \leq \widetilde{L}/ \lceil n^{1/d}\rceil. 
\end{align*}
As a consequence, if we define $\pi$ to be the joint distribution of $(X,Y_n)$ on $(\Sigma\setminus\overline\Omega)\times(\Sigma\setminus\overline\Omega)$, we see that
$$d_\infty(\nu_n,\nu)\leq \underset{\left((\Sigma\setminus\overline{\Omega})\times(\Sigma\setminus\overline{\Omega}),\pi\right)}{\operatorname{esssup}} \left\{ (x,y) \xmapsto{} |x-y| \right\} \leq \widetilde{L}/ \lceil n^{1/d}\rceil.$$

By Theorem~\ref{thm:transportexistence} there exists a sequence $(M_n)_{n\in\mathbb{N}}$ of transport maps such that $M_n:\Sigma\setminus\overline{\Omega}\rightarrow \Sigma\setminus\overline{\Omega}$, $M_n\#\nu=\nu_n$ and $ \sup_{x\in \Sigma\setminus\overline{\Omega}}\{|M_n(x)-x|\}\leq \widetilde{L}/n^{1/d}$. Additionally, we  have $|M_n(\Sigma\setminus\overline{\Omega})|\leq 2^dln$. Next, let $x\in M_n(\Sigma\setminus\overline{\Omega})$, then there exists $z\in T_n\left((0,1)^d\right)$ and $i\in\{1,\ldots,l\}$ such that $z=\Phi_i(x)$. Define $\widetilde{J_i}^{-1}$ to be the Jacobian of $\Phi_i^{-1}$ which we observe is bounded above by $\widetilde{L}^d$ and from below by $1/\widetilde{L}^d$. Then we compute
\begin{align*}
\nu_n(x) &=\mathbb{P}(Y_n=x) \geq \mathbb{P}(\Phi_i(X)\in T_n^{-1}(z), \zeta=i) 
\geq \mathbb{P}(X\in \Phi_i^{-1}( T_n^{-1}(z)) , \xi(i)=l ) \\
&= \frac{1}{l}\int_{ \Phi_i^{-1}( T_n^{-1}(z))} d\nu
= \frac{1}{l}\int_{ \Phi_i^{-1}( T_n^{-1}(z))} \frac{\widetilde{\rho}(x)}{\widetilde{\mu}(\Sigma\setminus\overline{\Omega})} \; dx \\
&\geq \frac{1}{\kappa l \widetilde{\mu}(\Sigma\setminus\overline{\Omega})}\int_{ T_n^{-1}(z)}  \widetilde{J_i}^{-1}(\mathbb{X}) \; d\mathbb{X} 
\geq \frac{1}{\widetilde{L}^d\kappa l \widetilde{\mu}(\Sigma\setminus\overline{\Omega})}\int_{ T_n^{-1}(z)}   \; d\mathbb{X}\\
&\geq \frac{1}{2^d\widetilde{L}^d\kappa l \widetilde{\mu}(\Sigma\setminus\overline{\Omega})n} .  \end{align*}
Here $\widetilde{\rho}$ and $\kappa$ are as specified in part (2) of Lemma~\ref{lem:ext1}. Therefore it suffices to take $c:= \frac{1}{2^d\widetilde{L}^d\kappa l \widetilde{\mu}(\Sigma\setminus\Omega)}$.

The upper bound holds via a similar argument. Indeed, let $x\in M_n(\Sigma\setminus\overline{\Omega})$ and label all the indices $j\in I$ for which $x\in \Phi_j^{-1}(T_n^{-1}((0,1)^d))$ as $j_1,j_2,\ldots,j_k$. Note that $k\leq l$. Choose points $z_1,z_2,\ldots,z_k\in (0,1)^d$ such that $z_i=\Phi_{j_i}(x)$. Then
\begin{align*}
    \nu_n(x) &= \mathbb{P}(Y_n=x) = \sum_{i=1}^k   \mathbb{P}(\Phi_{j_i}(X)\in T_n^{-1}(z_i), \zeta={j_i}) 
    \leq \sum_{i=1}^k   \mathbb{P}(\Phi_{j_i}(X)\in T_n^{-1}(z_i),  X\in U_{j_i}^+) \\
    &\leq \sum_{i=1}^k \int_{ \Phi_{j_i}^{-1}( T_n^{-1} (z_i))} d\nu 
    \leq \frac{\kappa}{\widetilde{\mu}(\Sigma\setminus\overline{\Omega})} \sum_{i=1}^k\int_{T_n^{-1}(z_i)} \widetilde{J_{j_i}}^{-1}(\mathbb{X}) \; d\mathbb{X} \leq \frac{\kappa l \widetilde{L}^d}{\widetilde{\mu}(\Sigma\setminus\overline{\Omega}) n}.
\end{align*}
The first inequality follows by noting $\zeta=j_i$ implies that $X\in U_{j_i}^+$. We can now conclude the result by taking $d=\frac{\kappa l \widetilde{L}^d}{\widetilde{\mu}(\Sigma\setminus\Omega) }$.
\end{proof}

We can now deduce our discrete extension results.

\begin{lemma}
\label{lem:ext2}
Let $p\geq 1$. Assume that (i) for a $K>0$, (ii) and (iii) from Assumptions~\ref{assum} hold. Additionally assume, for some $q>p$, $\aleph_\eta^q<+\infty$. Then there exist the following.
\begin{enumerate}[(1)]
    \item An open subset $\Sigma\supset\Omega$ and $\delta>0$ such that $\cup_{x\in\Omega}B_{\delta}(x)\subset\Sigma$.
    \item For each $n\in\mathbb{N}$, a finite subset $\widetilde{V_n}\subset\Sigma$ such that $\widetilde{V_n}\cap\Omega = V_n$ and a discrete probability measure $\widetilde{\mu_n}$ on $\widetilde{V_n}$ such that the below properties hold.
    \begin{enumerate}[(i)]
        \item There exists a constant $\alpha\in (0,1)$ (independent of $n$) such that, for any $n\in\mathbb{N}$ and any $x\in \Omega$, $\widetilde{\mu_n}(x)=\alpha\mu_n(x)$. In other words $\widetilde{\mu_n}(V_n)$ is a constant independent of $n\in\mathbb{N}$ and $\widetilde{\mu_n}|_{V_n}/\widetilde{\mu_n}(V_n)=\mu_n$.
        \item For each $n\in \mathbb{N}$ define $\widetilde{\Lambda}_n^-=\min_{x\in\widetilde{V_n}}\widetilde{\mu_n}(x)$. There exists a constant $\beta>0$ (independent of $n$) such that, for any $n\in\mathbb{N}$,
        $$ \frac{1}{\beta} \Lambda_n^- \leq \widetilde{\Lambda}_n^-\leq \beta \Lambda_n^-.$$
        \item For each $n\in \mathbb{N}$ define $\widetilde{\Lambda}_n^+=\max_{x\in\widetilde{V_n}}\widetilde{\mu_n}(x)$.  There exists a constant $\gamma>0$ (independent of $n$) such that, for any $n\in\mathbb{N}$,
        $$ \frac{1}{\gamma} \Lambda_n^+ \leq \widetilde{\Lambda}_n^+\leq \gamma \Lambda_n^+.$$
    \end{enumerate}
    \end{enumerate}
   
    \begin{enumerate}[(1)]
    \setcounter{enumi}{2}
    \item A probability measure $\widetilde{\mu}$ with density $\widetilde{\rho}$ on $\Sigma$ such that there exists $\kappa>0$ with $\frac{1}{\kappa}\leq\widetilde{\rho}\leq\kappa$. Additionally $\widetilde{\rho}|_{\Omega}$ is proportional to $\rho$.
    \item For each $n\in\mathbb{N}$, a Borel map $\widetilde{T_n}:\Sigma\rightarrow\Sigma$ such that $\widetilde{T_n}\# \widetilde{\mu}= \widetilde{\mu_n}$, $\widetilde{T_n}|_{\Omega}=T_n$ and a constant $\theta\geq 1$ (independent of $K$) such that (i) from Assumptions~\ref{assum} still holds while replacing $T_n$ with $\widetilde{T_n}$, $\Omega$ with $\Sigma$ and $K$ with $\theta K$.  
    \item For each $n\in\mathbb{N}$, an extension operator $\mathfrak{X}_n:L^p(V_n;\mu_n)\rightarrow L^p(\widetilde{V_n},\widetilde{\mu_n})$ such that, for any $u:V_n\rightarrow \mathbb{R}$, $\mathfrak{X}_nu|_{V_n}=u$. Moreover, there exist constants $A',C'>0$, $B'\geq 1$ (independent of both $K$ and $n$), such that, for any $s\geq K$, for any $n\in\mathbb{N}$ and any $u:V_n\rightarrow \mathbb{R}$,
   \begin{align*}
    &\frac{1}{\varepsilon_n^{d+p}}\sum_{x\in \widetilde{V_n}}\sum_{y\in \widetilde{V_n}} \mathbf{1}_{|x-y|\leq s\varepsilon_n} |\mathfrak{X}_nu(x)-\mathfrak{X}_nu(y)|^p \widetilde{\mu_n}(x)\widetilde{\mu_n}(y)\\
    \leq & A' \left( \Vert u \Vert^p_{L^p(V_n;\mu_n)} + \frac{1}{\varepsilon_n^{d+p}}\sum_{x\in V_n}\sum_{y\in V_n} \mathbf{1}_{|x-y|\leq sB'\varepsilon_n} |u(x)-u(y)|^p \mu_n(x)\mu_n(y) \right)
   \end{align*}
    and
    $$ \Vert \mathfrak{X}_n u\Vert_{L^p(\widetilde{V_n},\widetilde{\mu_n})} \leq C' \Vert u \Vert_{L^p(V_n;\mu_n)}.$$
    \item If we additionally assume (iv) from Assumptions~\ref{assum} holds, then there exists a constant $\widetilde{\mathfrak{D}}>0$, such that (iv) still holds when we replace $V_n$ with $\widetilde{V_n}$, $\mu_n$ with $\widetilde{\mu_n}$ and $\mathfrak{D}$ with $\widetilde{\mathfrak{D}}$. 
    \item If we additionally assume (i$^*$) from Assumptions~\ref{assum} holds, then (i$^*$) still holds if we replace $T_n$ with $\widetilde{T_n}$ and $\Omega$ with $\Sigma$.
\end{enumerate}

\end{lemma}

In the discussion in Section~\ref{sec:discussion} we shall break down this lemma in further detail, but for now we shall head straight to the proof.

\begin{proof}

Firstly apply Lemma~\ref{lem:ext1} to introduce the extensions $\Sigma$, $\widetilde{\mu}$, $\widetilde{\rho}$ and the map $\mathfrak{X}:L^p(\Omega;\mu)\rightarrow L^p(\Sigma;\widetilde{\mu})$. Note that to do so we require $\aleph_\eta^q<+\infty$. This handles parts~(1) and~(3). We also apply Lemma~\ref{lem:discretization} to generate a sequence of Borel maps $(M_n:\Sigma\setminus\overline{\Omega}\rightarrow \Sigma\setminus\overline{\Omega})_{n\in\mathbb{N}}$ and constants $a,b,c,d>0$ with the following properties:
\begin{itemize}
    \item $|M_n(\Sigma\setminus\overline{\Omega})|\leq a|V_n|$,
    \item $\sup_{x\in\Sigma\setminus\overline{\Omega}}|M_n(x)-x|\leq\frac{b}{|V_n|^{1/d}},$
    \item define $\nu$ to be the probability measure $\widetilde{\mu}$ restricted to $\Sigma\setminus\overline{\Omega}$ and normalised; then, for any $x\in M_n(\Sigma\setminus\overline{\Omega})$,
    $$ \frac{c}{|V_n|}\leq M_n\#\nu(x)\leq \frac{d}{|V_n|}.$$ 
\end{itemize} 
Next we construct the Borel maps from part~(4). For each $n\in\mathbb{N}$, we define $\widetilde{T_n}:\Sigma \rightarrow \Sigma$ by
$$
\widetilde{T_n}(x) :=
\begin{cases}
    T_n(x), &\text{if } x\in \Omega, \\
    M_n(x), &\text{if } x\in \Sigma\setminus\overline{\Omega}, \\ 
     x, &\text{if } x\in \partial\Omega.
\end{cases}
$$
How we define $\widetilde{T_n}$ on $\partial\Omega$ is not important as $\partial\Omega$ has $\widetilde{\mu}$ measure zero. According to Lemma~\ref{lem:transportbound} there exists $\theta'>0$ such that 
$$ \sup_{x\in\Sigma\setminus \overline{\Omega}} |M_n(x)-x| \leq \frac{b}{|V_n|^{1/d}}\leq \theta' \sup_{x\in \Omega} |T_n(x)-x|\leq \theta' K\varepsilon_n.$$
Let $\theta:=\max\{\theta',1\}$. From the definition of $\widetilde{T_n}$ it follows that
$$ \sup_{x\in\Sigma} |\widetilde{T_n}(x)-x|\leq \theta K\varepsilon_n.$$
We now define, for each $n\in\mathbb{N}$, $\widetilde{\mu_n}:= \widetilde{T_n}\#\widetilde{\mu}$, $\widetilde{V_n}:=\widetilde{T_n}(\Sigma)$ and note that $\widetilde{V_n}\cap\Omega=V_n$ as required which then concludes part (4). Now we shall handle part~(2). Given that $\widetilde{T_n}(x)\in\Omega$ if and only if $x\in \Omega$ we note that, for $x\in V_n$, $\widetilde{T_n}^{-1}(x)=T_n^{-1}(x)$. Therefore, by applying part~(2) of Lemma~\ref{lem:ext1}, for $x\in V_n$, $$ \widetilde{\mu_n}(x)= \int_{\widetilde{T_n}^{-1}(x)}d\widetilde{\mu}=\widetilde{\mu}(\Omega)\int_{T_n^{-1}(x)}d\mu=\widetilde{\mu}(\Omega)\mu_n(x).$$ 
Taking $\alpha=\widetilde{\mu}(\Omega)$ we see that~(2,i) is satisfied. For~2,ii) and~(2,iii) we observe the following. Firstly, since $\sum_{x\in V_n}\mu_n(x)=1$,
$$ \Lambda_n^-\leq \frac{1}{|V_n|}\leq \Lambda_n^+.$$
If $x\in V_n$, then $\widetilde{\mu}_n(x)=\alpha\mu_n(x)$. On the contrary, if $x\in \widetilde{V_n}\setminus V_n$, then $\widetilde{T_n}^{-1}(x)=M_n^{-1}(x)$ and 
$$ \widetilde{\mu_n}(x)= \widetilde{\mu}\left(\widetilde{T_n}^{-1}(x)\right)= \widetilde{\mu}(\Sigma\setminus\overline\Omega) \,\,\nu\left(M_n^{-1}(x)\right),$$ and so, by the properties of the map $M_n$,
$$ c(1-\alpha)\Lambda_n^-\leq \frac{c(1-\alpha)}{|V_n|}\leq\widetilde{\mu_n}(x) \leq \frac{d(1-\alpha)}{|V_n|}\leq d(1-\alpha)\Lambda_n^+.$$
With these observations we can now conclude part~(2) is fulfilled.

Next we shall handle part~(5). For each $n\in \mathbb{N}$, we define the map $\mathfrak{X}_n:\mathbb{R}^{V_n}\rightarrow \mathbb{R}^{V_n}$, for $u\in \mathbb{R}^{V_n}$ and $x\in \widetilde{V_n}$, by
$$\mathfrak{X}_n u(x) := \frac{1}{\widetilde{\mu_n}(x)}\int_{\widetilde{T_n}^{-1}(x)}\mathfrak{X}(u\circ T_n)(y) \, d\widetilde{\mu}(y).$$
Indeed, if $x\in V_n$, then $\widetilde{T_n}^{-1}(x)= T_n^{-1}(x)$ and
$$ \mathfrak{X}_nu(x) = \frac{1}{\widetilde{\mu_n}(x)}\int_{T_n^{-1}(x)} u(T_n(y)) \, d\widetilde{\mu}(y)= \frac{1}{\widetilde{\mu_n}(x)}\int_{T_n^{-1}(x)} u(x) \, d\widetilde{\mu}(y) = u(x).$$
Henceforth we shall let $s>0$ and assume $K\leq s$. Then we compute 
\begin{align*}
&\frac{1}{\varepsilon_n^{d+p}} \sum_{x\in\widetilde{V_n}}\sum_{y\in \widetilde{V_n}} \mathbf{1}_{|x-y|\leq s\varepsilon_n}|\mathfrak{X}_nu(x)-\mathfrak{X}_nu(y)|^p\widetilde{\mu_n}(x)\widetilde{\mu_n}(y)   \\
&=\frac{1}{\varepsilon_n^{d+p}}\sum_{x\in\widetilde{V_n}}\sum_{y\in \widetilde{V_n}}  \mathbf{1}_{|x-y|\leq s\varepsilon_n}\left| \iint_{\widetilde{T_n}^{-1}(x)\times \widetilde{T_n}^{-1}(y)} \left(\mathfrak{X}(u\circ T_n)(\mathbb{X})-\mathfrak{X}(u\circ T_n)(\mathbb{Y})\right)\frac{d\widetilde{\mu}(\mathbb{X})}{\widetilde{\mu_n}(x)}\frac{d\widetilde{\mu}(\mathbb{Y})}{\widetilde{\mu_n}(y)}\right|^p\widetilde{\mu_n}(x)\widetilde{\mu_n}(y)   \\
&\leq \frac{1}{\varepsilon_n^{d+p}}\sum_{x\in\widetilde{V_n}}\sum_{y\in \widetilde{V_n}}  \mathbf{1}_{|x-y|\leq s\varepsilon_n} \left( \iint_{\widetilde{T_n}^{-1}(x)\times \widetilde{T_n}^{-1}(y)} \left|\mathfrak{X}(u\circ T_n)(\mathbb{X})-\mathfrak{X}(u\circ T_n)(\mathbb{Y})\right|^p\frac{d\widetilde{\mu}(\mathbb{X})}{\widetilde{\mu_n}(x)}\frac{d\widetilde{\mu}(\mathbb{Y})}{\widetilde{\mu_n}(y)}\right)\widetilde{\mu_n}(x)\widetilde{\mu_n}(y) \\
&= \frac{1}{\varepsilon_n^{d+p}}\sum_{x\in\widetilde{V_n}}\sum_{y\in \widetilde{V_n}}  \iint_{\widetilde{T_n}^{-1}(x)\times \widetilde{T_n}^{-1}(y)} \mathbf{1}_{|x-y|\leq s\varepsilon_n}\left|\mathfrak{X}(u\circ T_n)(\mathbb{X})-\mathfrak{X}(u\circ T_n)(\mathbb{Y})\right|^pd\widetilde{\mu}(\mathbb{X})d\widetilde{\mu}(\mathbb{Y}).
\end{align*}
The second line of the above computation holds via an application of Jensen's inequality. Next suppose $x,y\in \widetilde{V_n}$ such that $|x-y|\leq s\varepsilon_n$. Let $\mathbb{X}\in \widetilde{T_n}^{-1}(x)$, $\mathbb{Y}\in \widetilde{T_n}^{-1}(y)$, then 
\begin{align*}
|\mathbb{X}-\mathbb{Y}|&\leq |x-y|+|\mathbb{X}-x|+|\mathbb{Y}-y| %\\ &
= |x-y|+|\mathbb{X}-\widetilde{T_n}(\mathbb{X})|+|\mathbb{Y}-\widetilde{T_n}(\mathbb{Y})| \\
&\leq s\varepsilon_n +2\sup_{z\in \Sigma} |z - \widetilde{T_n}(z)| 
\leq (s+2\theta K)\varepsilon_n \leq (1+2\theta)s\varepsilon_n.
\end{align*}
Defining $B_1=1+2\theta$ we see that
$$ \mathbf{1}_{|x-y|\leq s\varepsilon_n} \leq \mathbf{1}_{|\mathbb{X}-\mathbb{Y}|\leq B_1s\varepsilon_n} .$$
Therefore
\begin{align*}
&\frac{1}{\varepsilon_n^{d+p}} \sum_{x\in\widetilde{V_n}}\sum_{y\in \widetilde{V_n}} \mathbf{1}_{|x-y|\leq s\varepsilon_n}|\mathfrak{X}_nu(x)-\mathfrak{X}_nu(y)|^p\widetilde{\mu_n}(x)\widetilde{\mu_n}(y)   \\ 
&\leq \frac{1}{\varepsilon_n^{d+p}}\sum_{x\in\widetilde{V_n}}\sum_{y\in \widetilde{V_n}}  \iint_{\widetilde{T_n}^{-1}(x)\times \widetilde{T_n}^{-1}(y)} \mathbf{1}_{|\mathbb{X}-\mathbb{Y}|\leq B_1s\varepsilon_n}\left|\mathfrak{X}(u\circ T_n)(\mathbb{X})-\mathfrak{X}(u\circ T_n)(\mathbb{Y})\right|^pd\widetilde{\mu}(\mathbb{X})d\widetilde{\mu}(\mathbb{Y}) \\ 
&\leq \frac{1}{\varepsilon_n^{d+p}}\iint_{\Sigma\times\Sigma} \mathbf{1}_{|\mathbb{X}-\mathbb{Y}|\leq B_1s\varepsilon_n}\left|\mathfrak{X}(u\circ T_n)(\mathbb{X})-\mathfrak{X}(u\circ T_n)(\mathbb{Y})\right|^pd\widetilde{\mu}(\mathbb{X})d\widetilde{\mu}(\mathbb{Y}) \\
&\leq A'\left( \Vert u \Vert^p_{L^p(V_n;\mu_n)}+\frac{1}{\varepsilon_n^{d+p}}\iint_{\Omega\times\Omega} \mathbf{1}_{|x-y|\leq B_2s\varepsilon_n}\left|(u\circ T_n)(x)-(u\circ T_n)(y)\right|^pd\mu(x)d\mu(y) \right) \\
&\leq  A'\left( \Vert u \Vert^p_{L^p(V_n;\mu_n)}+\frac{1}{\varepsilon_n^{d+p}}\iint_{\Omega\times\Omega} \mathbf{1}_{|T_n(x)-T_n(y)|\leq B's\varepsilon_n}\left|(u\circ T_n)(x)-(u\circ T_n)(y)\right|^pd\mu(x)d\mu(y) \right) \\
&=  A'\left( \Vert u \Vert^p_{L^p(V_n;\mu_n)}+\frac{1}{\varepsilon_n^{d+p}}\sum_{x\in V_n}\sum_{y\in V_n} \mathbf{1}_{|x-y|\leq B's\varepsilon_n}\left|u(x)-u(y)\right|^pd\mu_n(x)d\mu_n(y) \right). 
\end{align*}
The third inequality follows by applying Lemma~\ref{lem:ext1} and setting $B_2:=B\cdot B_1$, where $B$ is as in that lemma and $A'>0$ is chosen appropriately. The fourth inequality follows by taking $B':=B_2+2$ and observing that 
$$ \mathbf{1}_{|x-y|\leq B_2s\varepsilon_n} \leq \mathbf{1}_{|T_n(x)-T_n(y)|\leq (B_2s+2K)\varepsilon_n} \leq \mathbf{1}_{|T_n(x)-T_n(y)|\leq (B_2+2)s\varepsilon_n}  .$$
We also note that $B'\geq 1$ as required. Lastly we compute, again by applying Jensen's inequality,
\begin{align*}
\Vert \mathfrak{X}_nu \Vert^p_{L^p(\widetilde{V_n},\widetilde{\mu_n})}&= \sum_{x\in \widetilde{V_n}}|\mathfrak{X}_nu(x)|^p \widetilde{\mu_n}(x) %\\ &
= \sum_{x\in \widetilde{V_n}} \left|  \int_{\widetilde{T_n}^{-1}(x)}\mathfrak{X}(u\circ T_n)(y) \frac{d\widetilde{\mu}(y)}{\widetilde{\mu_n}(x)}\right|^p\widetilde{\mu_n}(x) \\
&\leq \sum_{x\in \widetilde{V_n}}   \int_{\widetilde{T_n}^{-1}(x)}|\mathfrak{X}(u\circ T_n)(y)|^p \, d\widetilde{\mu}(y) %\\ &
= \int_{\Sigma}|\mathfrak{X}(u\circ T_n)(y)|^p  \, d\widetilde{\mu}(y)\\
&\leq C^p\Vert u\circ T_n \Vert^p_{L^p(\Omega;\mu)}=C^p\Vert u \Vert^p_{L^p(V_n;\mu_n)} .
\end{align*}
For the second inequality we use part~(3) of Lemma~\ref{lem:ext1}. With this, part~(5) is fulfilled.

Part~(6) holds via a similar argument as the one we used to prove parts~(2,ii) and~(2,iii); part~(7) holds via a similar argument as the one we used to prove part~(4).
\end{proof}

\subsection{Uniform Poincar\'e Inequalities}\label{sec:unifPoincare}

For the main theorem of this subsection, Theorem~\ref{thm:poincare}, we will require a compactness result in the space $TL^p(\Omega)$. Under an additional assumption on the parameter sequence $(\varepsilon_n)_{n\in\mathbb{N}}$  and assuming $p<d$,  we shall later provide improvements in Corollaries~\ref{cor:compactness1} and \ref{cor:compactness2}. The following results apply for any $p\in [1,\infty)$. 

\begin{proposition}
\label{prop:compactness}
Let $p\in [1,\infty)$. There exists a constant $H(\eta)>0$, depending only on $\eta$, such that the following holds. Assume that (i) for a $K\leq H(\eta)$, (ii) and (iii) from Assumptions~\ref{assum} hold. Additionally assume there exists $q>p$ such that $\aleph_{\eta}^q<+\infty$. Let $(u_n)_{n\in\mathbb{N}}$ be a sequence of functions $u_n:V_n\rightarrow \mathbb{R}$ such that $\{\Vert u_n \Vert_{L^p(V_n;\mu_n)}\}_{n\in\mathbb{N}}$ and $\{\mathcal{GE}_n^p(u_n)\}_{n\in\mathbb{N}}$ are bounded. Then the sequence $\big((u_n,\mu_n)\big)_{n\in \mathbb{N}}$ is relatively compact in $TL^p(\Omega)$.
\end{proposition}

We point out that the above result only requires $(\varepsilon_n)_{n\in\mathbb{N}}$ to converge to zero. Unlike later compactness results in Corollaries~\ref{cor:compactness1} and~\ref{cor:compactness2}, there is no additional requirement on how fast it converges to zero. 
\begin{proof}[Proof of Proposition~\ref{prop:compactness}]

To begin, we shall define some constants and functions that will be of relevance later in the proof. Take a sequence $(u_n)_{n\in \mathbb{N}}$ bounded as in the prerequisites of the theorem. By (ii) in Assumptions~\ref{assum} there exists $l>0$ such that, for $x\leq l$, $\eta(x)\geq \eta(0)/2>0$. Choose a smooth, compactly supported function $J:\mathbb{R}^d\rightarrow [0,+\infty)$ such that $\operatorname{supp}(J)\subset B_{l}(0)$ and $\int_{\mathbb{R}^d}J(x)dx=1$.
For $\varepsilon>0$ define $J_\varepsilon:= \frac{1}{\varepsilon^d}J\left(\frac{\cdot}{\varepsilon}\right)$. 

Apply Lemma~\ref{lem:ext1} to introduce the extension operator $\mathfrak{X}:L^p(\Omega;\mu)\rightarrow L^p(\Sigma;\widetilde{\mu})$, with $\Sigma$ and $\tilde\mu$ as in that lemma. This we can do since $\aleph_{\eta}^q<+\infty$. We shall also take the constant $B>0$ as specified in Lemma~\ref{lem:ext1} and the constant $b>0$ as specified in Remark~\ref{rem:simpletokernel}. Define $l_1:=Bl$, $l_2:=bBl$, $\alpha:=\frac{1}{2l_2}$ and $H(\eta):=\frac{1}{4b}$. Per assumption $K\leq H(\eta)$. For each $n\in\mathbb{N}$, define $v_n:\Omega \rightarrow \mathbb{R}$ by
$$ v_n(x) := \int_{\mathbb{R}^d} J_{\alpha\varepsilon_n}(x-y)\mathfrak{X}(u_n\circ T_n)(y) dy.$$
Since $J$ has support in a ball of radius $l$ we can replace the domain of the above integral with $\cup_{z\in \Omega} B_{l\alpha\varepsilon_n}(z)$. For $n$ sufficiently large we can assume $l\alpha\varepsilon_n \leq \delta$ ($\delta$ is the constant specified in part~(1) of Lemma~\ref{lem:ext1}) and so $\cup_{z\in \Omega} B_{l\alpha\varepsilon_n}(z)\subset \Sigma$. Thus we may in fact replace the domain of the integral above with $\Sigma$. This shows that the integral is well defined, even though the domain of $\mathfrak{X}(u_n\circ T_n)$ is only $\Sigma$, not all of $\mathbb{R}^d$. Furthermore, for $x\in \Omega$, $J_{\alpha\varepsilon_n}(x-y)$ defines a probability density (with respect to the Lebesgue measure) on $\Sigma$. We use this fact to apply Jensen's inequality:
\begin{align*}
    \Vert v_n \Vert_{L^p(\Omega;\mu)} &= \left( \int_{\Omega} \left| \int_{\Sigma} J_{\alpha\varepsilon_n}(x-y)\mathfrak{X}(u_n\circ T_n)(y) \, dy\right|^p \, d\mu(x) \right)^{1/p} \\
    &\leq \left( \int_{\Omega}  \int_{\Sigma} J_{\alpha\varepsilon_n}(x-y)\left|\mathfrak{X}(u_n\circ T_n)\right|^p(y) \, dy \, d\mu(x) \right)^{1/p} \\
    &\lesssim_n \left( \int_{\mathbb{R}^d}  \int_{\Sigma} J_{\alpha\varepsilon_n}(x-y)\left|\mathfrak{X}(u_n\circ T_n)\right|^p(y) \, dy \, dx \right)^{1/p} \\
    &= \Vert \mathfrak{X}(u_n\circ T_n) \Vert_{L^p(\Sigma)} \lesssim_n \Vert u_n\circ T_n \Vert_{L^p(\Omega;\mu)} = \Vert u_n \Vert_{L^p(V_n;\mu_n)}. 
\end{align*}
For the second inequality above we used (iii) in Assumptions~\ref{assum}. Next we compute
\begin{align*}
    \Vert \nabla v_n \Vert_{L^p(\Omega)} &= \left( \int_{\Omega} \left| \int_{\Sigma} \frac{1}{(\alpha\varepsilon_n)^{d+1}}\nabla J\left(\frac{x-y}{\alpha\varepsilon_n}\right)\mathfrak{X}(u_n\circ T_n)(y) \, dy\right|^p \, d\mu(x) \right)^{1/p} \\
    &= \left( \int_{\Omega} \left| \int_{\Sigma} \frac{1}{(\alpha\varepsilon_n)^{d+1}}\nabla J\left(\frac{x-y}{\alpha\varepsilon_n}\right)\left(\mathfrak{X}(u_n\circ T_n)(y)-\mathfrak{X}(u_n\circ T_n)(x) \right) \, dy\right|^p \, d\mu(x)  \right)^{1/p} \\
    &\leq  \left( \int_{\Omega} \left| \int_{\Sigma} \frac{1}{(\alpha\varepsilon_n)^{d+1}}|\nabla J|\left(\frac{x-y}{\alpha\varepsilon_n}\right)\left|\mathfrak{X}(u_n\circ T_n)(y)-\mathfrak{X}(u_n\circ T_n)(x) \right| \, dy\right|^p \, d\mu(x)  \right)^{1/p}.
\end{align*}
The second equality follows since $J\left(\frac{x-\cdot}{\alpha \varepsilon_n}\right)$ has compact support in $\Sigma$ and thus $\mathfrak{X}(u_n\circ T_n)(x) \int_\Sigma \nabla J \left(\frac{x-y}{\alpha \varepsilon_n}\right) \, dy = 0$. 

The next step is to use Jensen's inequality again. Define $A:=\int_{\mathbb{R}^d}|\nabla J|(z)dz$ and note, for $x\in \Omega$, $\frac{1}{A(\alpha\varepsilon_n)^d}|\nabla J|\left(\frac{x-y}{\alpha\varepsilon_n}\right)$ defines a probability density on $\Sigma$. From this we now produce the bound
\begin{align}
    \Vert \nabla v_n\Vert_{L^p(\Omega;\mu)} &\leq \frac{A}{\alpha\varepsilon_n} \left( \int_{\Omega} \left| \int_{\Sigma} \frac{1}{A(\alpha\varepsilon_n)^d}|\nabla J|\left(\frac{x-y}{\alpha\varepsilon_n}\right)\left|\mathfrak{X}(u_n\circ T_n)(y)-\mathfrak{X}(u_n\circ T_n)(x) \right| \, dy\right|^p \, d\mu(x)  \right)^{1/p} \label{eq:gradvnineq}\\
    &\leq  \frac{A}{\alpha\varepsilon_n} \left( \int_{\Omega}  \int_{\Sigma} \frac{1}{A(\alpha\varepsilon_n)^d}|\nabla J|\left(\frac{x-y}{\alpha\varepsilon_n}\right)\left|\mathfrak{X}(u_n\circ T_n)(y)-\mathfrak{X}(u_n\circ T_n)(x) \right|^p \, dy \, d\mu(x)  \right)^{1/p} \notag \\
     &\lesssim_n  \frac{1}{\varepsilon_n} \left( \int_{\Sigma}  \int_{\Sigma} \frac{1}{\varepsilon_n^d} \mathbf{1}_{\left|x-y\right|\leq l\alpha\varepsilon_n}\left|\mathfrak{X}(u_n\circ T_n)(y)-\mathfrak{X}(u_n\circ T_n)(x) \right|^p \, dy \, dx  \right)^{1/p},\notag
    \end{align}
    where the final inequality follows since $|\nabla J|$ is bounded and has support in the ball $B_{l\alpha \varepsilon_n}(x)$ and by using (iii) in Assumptions~\ref{assum}. 
    
    Next, applying part~(3) in Lemma~\ref{lem:ext1} we compute   
\begin{align*}
     &\Vert \nabla v_n\Vert_{L^p(\Omega;\mu)} \\
     &\lesssim_n \Vert u_n \circ T_n \Vert_{L^p(\Omega)} +  \frac{1}{\varepsilon_n} \left( \int_{\Omega}  \int_{\Omega} \frac{1}{\varepsilon_n^d} \mathbf{1}_{\left|x-y\right|\leq l_1\alpha\varepsilon_n}\left|(u_n\circ T_n)(y)-(u_n\circ T_n)(x) \right|^p \, dy \, dx  \right)^{1/p} \\
     &\lesssim_n \Vert u_n \Vert_{L^p(V_n;\mu_n)} +\frac{1}{\varepsilon_n} \left( \int_{\Omega}  \int_{\Omega} \frac{1}{\varepsilon_n^d} \mathbf{1}_{\left|T_n(x)-T_n(y)\right|\leq (l_1\alpha+2K)\varepsilon_n}\left|(u_n\circ T_n)(y)-(u_n\circ T_n)(x) \right|^p \, dy \, dx  \right)^{1/p} \\
     &\lesssim_n \Vert u_n \Vert_{L^p(V_n;\mu_n)} + \frac{1}{\varepsilon_n} \left( \int_{\Omega}  \int_{\Omega} \frac{1}{\varepsilon_n^d} \eta_{ (l_2\alpha+2bK)\varepsilon_n}(|T_n(x)-T_n(y)|)\left|(u_n\circ T_n)(y)-(u_n\circ T_n)(x) \right|^p \, dy \, dx  \right)^{1/p} \\
     &\lesssim_n \Vert u_n \Vert_{L^p(V_n;\mu_n)} + \frac{1}{\varepsilon_n} \left( \int_{\Omega}  \int_{\Omega} \frac{1}{\varepsilon_n^d} \eta_{ (l_2\alpha+2bK)\varepsilon_n}(|T_n(x)-T_n(y)|)\left|(u_n\circ T_n)(y)-(u_n\circ T_n)(x) \right|^p \, d\mu(y) \, d\mu(x)  \right)^{1/p} \\
     &\lesssim_n \Vert u_n \Vert_{L^p(V_n;\mu_n)} + \left(\mathcal{GE}^p[V_n,\mu_n;(l_2\alpha+2bK)\varepsilon_n](u_n)\right)^{1/p} \\
     &\lesssim_n \Vert u_n \Vert_{L^p(V_n;\mu_n)} + \left(\mathcal{GE}_n^p(u_n)\right)^{1/p}.
\end{align*}  
The second inequality follows from noting that, if $|x-y|\leq l_1\alpha\varepsilon_n$, then, by (i) from Assumptions~\ref{assum},
\begin{equation}\label{eq:TnxTny}
|T_n(x)-T_n(y)| \leq |x-y| + 2\sup_{x\in\Omega} |T_n(x)-x|\leq l_1\alpha\varepsilon_n + 2K\varepsilon_n .
\end{equation}
The third inequality follows from the observation stated in Remark~\ref{rem:simpletokernel}. The fourth inequality is a consequence of (iii) from Assumptions~\ref{assum} and the fifth follows from $T_n\# \mu = \mu_n$ and the definition of $\mathcal{GE}^p$. The sixth inequality follows from our choice of $\alpha$ and $H(\eta)$ which satisfy $l_2\alpha+2bK\leq 1$, justifying the final bound by~\eqref{eq:aepsilon}.

We now see that the sequence $(v_n)_{n\in\mathbb{N}}$ is bounded in $W^{1,p}(\Omega)$ and thus relatively compact in $L^p(\Omega;\mu)$ (see~\cite[Theorem 11.10]{Leoni2009}). To conclude we will show that $\Vert u_n\circ T_n - v_n \Vert_{L^p(\Omega;\mu)} \rightarrow 0$ as $n\rightarrow \infty$, which suffices to conclude convergence in $TL^p(\Omega)$ by \eqref{eq:TLpconvcond} and part~(i) in Assumptions~\ref{assum}.  Indeed, applying (iii) from Assumptions~\ref{assum} and Jensen's inequality again,
\begin{align*}
\Vert u_n\circ T_n - v_n \Vert_{L^p(\Omega;\mu)} &= \left( \int_{\Omega}\left| \int_{\Sigma} J_{\alpha\varepsilon_n}(x-y)\left[\mathfrak{X}(u_n\circ T_n)(y)-\mathfrak{X}(u_n\circ T_n)(x)\right] \, dy \right|^p \, d\mu(x)\right)^{1/p} \\
&\lesssim_n \left( \int_{\Omega}\left| \int_{\Sigma} J_{\alpha\varepsilon_n}(x-y)\left[\mathfrak{X}(u_n\circ T_n)(y)-\mathfrak{X}(u_n\circ T_n)(x)\right] \, dy \right|^p \, dx\right)^{1/p} \\
&\leq \left( \int_{\Omega} \int_{\Sigma} J_{\alpha\varepsilon_n}(x-y)\left|\mathfrak{X}(u_n\circ T_n)(y)-\mathfrak{X}(u_n\circ T_n)(x)\right|^p \, dy \, dx\right)^{1/p} \\
&\lesssim_n \varepsilon_n \Vert u_n \Vert_{L^p(V_n;\mu_n)} +\varepsilon_n  \mathcal{GE}_n^p(u_n) \rightarrow 0 \; \; \text{as} \; \; n \rightarrow \infty.
\end{align*}
The third inequality is a consequence of a similar argument as the one we used to bound $\Vert \nabla v_n \Vert_{L^p(\Omega)}$. Indeed, comparing with \eqref{eq:gradvnineq}, we note that the main difference is that the kernel $J$ appears above, rather than $|\nabla J|$ in \eqref{eq:gradvnineq}. Since both are bounded and supported within the same ball, the same argument as that following \eqref{eq:gradvnineq} holds, mutatis mutandis. The final limit follows since $\varepsilon_n\rightarrow 0$ as $n\rightarrow\infty$, and $\{\Vert u_n \Vert_{L^p(V_n;\mu_n)}\}_{n\in\mathbb{N}}$ and $\{\mathcal{GE}_n^p(u_n)\}_{n\in\mathbb{N}}$ are bounded. The proof is now concluded. 
\end{proof}

\begin{remark}
For the rest of the paper $H(\eta)$ will refer to the same constant as constructed in the proof of Lemma~\ref{prop:compactness}.
\end{remark}

It is known that if we assume (i$^*$), (ii) and (iii) from Assumptions~\ref{assum}, then $\mathcal{GE}_n^p$ $\Gamma$-converges to $\mathcal{E}^p$ over $TL^p(\Omega)$ as $n\rightarrow\infty$ (see \cite[Theorem 1.1]{garcia2016continuum} and \cite[Theorem 4.7]{Dejan2019}).  The following lemma justifies that the $\liminf$ inequality condition from the definition of $\Gamma$-convergence (see \cite{DalMaso93}) still holds if we replace (i$^*$) with (i), up to a constant factor.

\begin{lemma}
\label{lem:liminfweak}
Let $p\in [1,\infty)$. Assume that (i) for a $K\leq H(\eta)$, (ii) and (iii) from Assumptions~\ref{assum} hold and $\aleph_\eta^p<+\infty$. Then, there exists a constant $L>0$ such that, for any sequence $\left((u_n,\mu_n)\right)_{n\in\mathbb{N}}$ converging to $(u,\mu)$ in $TL^p(\Omega)$, we have
$$ \mathcal{E}^p(u) \leq L\cdot \liminf_{n\rightarrow \infty} \mathcal{GE}_n^p(u_n) .$$
\end{lemma}

\begin{proof}
Take $\left((u_n,\mu_n)\right)_{n\in\mathbb{N}}$ and $(u,\mu)$ as within the statement of the theorem. By~\eqref{eq:TLpconvcond}, $(u_n\circ T_n)_{n\in\mathbb{N}}$ converges to $u$ in $L^p(\Omega;\mu)$ as $n\rightarrow \infty$. Take the constants $a>0$, $b\geq1$ as specified in Remark~\ref{rem:simpletokernel} and let $\alpha=\frac{1}{2b}$. For $x,y\in \Omega$ and $n\in \mathbb{N}$, we have that $|x-y|\leq \alpha\varepsilon_n$ implies $|T_n(x) -T_n(y)|\leq \alpha\varepsilon_n+2K\varepsilon_n$. This follows by a similar argument as in~\eqref{eq:TnxTny} in the previous proof. Hence, by the observation made in Remark~\ref{rem:simpletokernel}, 
$$ \mathbf{1}_{|x-y|\leq \alpha\varepsilon_n} \leq \mathbf{1}_{|T_n(x) -T_n(y)|\leq \alpha\varepsilon_n+2K\varepsilon_n} \leq \frac{1}{a} \eta \left( \frac{|T_n(x) -T_n(y)|}{\alpha b\varepsilon_n+2bK\varepsilon_n}\right).$$
As in the proof of Proposition~\ref{prop:compactness}, $K\leq H(\eta)= \frac{1}{4b}$. Putting this together with (ii) from Assumptions~\ref{assum} we get $\varepsilon_n^{-d}\mathbf{1}_{|x-y|\leq \alpha\varepsilon_n} \lesssim_n \eta_{\varepsilon_n}(|T_n(x)-T_n(y)|) $. We can now apply \cite[Lemma 4.6]{Dejan2019}\footnote{This lemma is a nonlocal-to-local $\Gamma$-convergence result for variations on $L^p(\Omega;\mu)$; we only need to apply the $\liminf$ part of $\Gamma$-convergence. We note that the assumptions (A6) and (A7) in \cite{Dejan2019} which \cite[Lemma 4.6]{Dejan2019} requires, are satisfied by part~(ii) in Assumptions~\ref{assum}. Assumption (A8) follows since $\aleph_\eta^p<+\infty$ (where we note `$p+d$' in (A8) is a typo that should read `$p+d-1$' instead). Assumption (A2) holds by part~(iii) in Assumptions~\ref{assum} and (A1) holds by our assumptions on $\Omega$.} to conclude 
\begin{align*}
    \mathcal{E}^p(u) &\lesssim_n \liminf_{n\rightarrow\infty} \frac{1}{\varepsilon_n^{d+p}} \iint_{\Omega\times\Omega} \mathbf{1}(|x-y|\leq \alpha\varepsilon_n) |(u\circ T_n)(x) - (u\circ T_n)(y)|^pd \, \mu(x) \, d\mu(y) \\
    &\lesssim_n \liminf_{n\rightarrow\infty}\frac{1}{\varepsilon_n^{p}} \iint_{\Omega\times\Omega}  \eta_{\varepsilon_n}(|T_n(x)-T_n(y)|)|(u\circ T_n)(x) - (u\circ T_n)(y)|^pd \, \mu(x) \, d\mu(y) \\
    &= \liminf_{n\rightarrow \infty} \mathcal{GE}_n^p(u_n) .
\end{align*}
\end{proof}

We also require a quick lemma regarding convergence of means in the $TL^p(\Omega)$ metric.

\begin{lemma}\label{lem:convmeans} Let $p\in [1,\infty)$. Let $(u_n)_{n\in \mathbb{N}}$ be a sequence of functions $u_n:V_n\rightarrow \mathbb{R}$ that converges to a function $u\in L^p(\Omega;\mu)$ in $TL^p(\Omega)$. Then $\mathbb{E}_{\mu_n}u_n\rightarrow \mathbb{E}_{\mu}u$ as $n\rightarrow \infty$.  
\end{lemma}

\begin{proof}
By the definition of the $TL^p$ metric we can choose a sequence of transport plans $\gamma_n\in \Gamma(\mu_n,\mu)$ such that
$$ \int_{\Omega\times\Omega} \left(|u_n(x)-u(y)|^p+ |x-y|^p\right) \;d\gamma_n(x,y) \rightarrow 0 \; \; \text{as} \; \; n\rightarrow \infty .$$
Then we have
\begin{align*}
    |\mathbb{E}_{\mu_n}u_n - \mathbb{E}_{\mu}u|&= \left| \int_{\Omega\times\Omega} ( u_n(x)-u(y))  \; d\gamma_n(x,y) \right| 
    \leq \int_{\Omega\times\Omega} |u_n(x)-u(y)| \, d\gamma_n(x,y) \\
     &\leq \left(\int_{\Omega\times\Omega} |u_n(x)-u(y)|^p \, d\gamma_n(x,y) \right)^{1/p} \rightarrow 0 \; \; \text{as} \; \; n\rightarrow\infty.
\end{align*}
\end{proof}

We can now deduce the following uniform collection of Poincar\'e inequalities.

\begin{theorem}
\label{thm:poincare} 

Let $p\in [1,\infty)$. Assume that (i) for a $K\leq H(\eta)$, (ii) and (iii) from Assumptions~\ref{assum} hold. Additionally assume there exists $q>p$ such that $\aleph_{\eta}^q<+\infty$. Then there exists a constant $C>0$ such that, for any $n\in \mathbb{N}$ and function $u:V_n \rightarrow \mathbb{R}$, we have
$$ \Vert u -\mathbb{E}_{\mu_n}u\Vert_{L^p(V_n;\mu_n)} \leq C \left( \mathcal{GE}^p_n(u) \right)^{1/p}.$$

\end{theorem}

\begin{proof}
We will argue by contradiction. Indeed, assume the statement of the result fails. Thus (after possibly passing to a subsequence) there exists a sequence of functions $u_n:V_n \rightarrow \mathbb{R}$ and a sequence of constants $C_n>0$ such that $C_n\rightarrow \infty$ as $n\rightarrow \infty$ and additionally
$$ \Vert u_n - \mathbb{E}_{\mu_n}u_n \Vert_{L^p(V_n;\mu_n)}^p > C_n\cdot \mathcal{GE}_n^p(u_n).$$
Next, for each $n\in \mathbb{N}$, define
$$ v_n:= \frac{u_n - \mathbb{E}_{\mu_n}u_n}{\Vert u_n  - \mathbb{E}_{\mu_n}u_n\Vert_{L^p(V_n;\mu_n)}} .$$
We note $\Vert v_n \Vert_{L^p(V_n;\mu_n)}=1$ and $\mathbb{E}_{\mu_n}v_n=0$. Moreover,
\begin{align*}
\mathcal{GE}_n^p(v_n)&= \mathcal{GE}_n^p \left( \frac{u_n - \mathbb{E}_{\mu_n}u_n}{\Vert u_n  - \mathbb{E}_{\mu_n}u_n\Vert_{L^p(V_n;\mu_n)}} \right) 
= \frac{1}{\Vert u_n  - \mathbb{E}_{\mu_n}u_n\Vert_{L^p(V_n;\mu_n)}^p}\mathcal{GE}_n^p(u_n) \leq \frac{1}{C_n} . 
\end{align*}
Because also $\aleph_\eta^q<+\infty$, we can apply Proposition~\ref{prop:compactness} to conclude (again after possibly passing to another subsequence) there exists a $v\in L^p(\Omega;\mu)$ such that $(V_n;\mu_n) \rightarrow (v,\mu)$ as $n\rightarrow\infty$ in $TL^p(\Omega)$. Next we apply Lemma~\ref{lem:liminfweak}, which is again possible since $\aleph_\eta^q<+\infty$, to note that
$$ \mathcal{E}^p(v) \leq L\cdot \liminf_{n\rightarrow \infty} \mathcal{GE}^p_{(n,\varepsilon_n)}(v_n) \leq \liminf_{n\rightarrow \infty} \frac{L}{C_n} = 0.$$
So, by the definition of $\mathcal{E}^p$ and (iii) from Assumptions~\ref{assum}, $v$ is constant on $\Omega$; in particular we get, by Lemma~\ref{lem:convmeans},
$$ v =\mathbb{E}_{\mu}v = \lim_{n\rightarrow\infty} \mathbb{E}_{\mu_n}v_n=0.$$
Convergence in $TL^p(\Omega)$ implies convergence of norms, as proved in part 4 of~\cite[Proposition 5.17]{mercer2025}. Therefore,
$$ 0 = \Vert v \Vert_{L^p(\Omega;\mu)}=\lim_{n\rightarrow \infty} \Vert v_n \Vert_{L^p(V_n;\mu_n)} =1,$$
giving us our contradiction.
\end{proof}

\subsection{Graph Sobolev inequalities for \texorpdfstring{$p<d$}{plessd}}\label{sec:Sobolevplessd}

Assume that $1\leq p<d$, then the standard theory of Sobolev spaces \cite[Theorem 1.4.4.1]{Grisvard1985} (assuming part (iii) of Assumptions~\ref{assum}) shows that, if $p\leq q\leq \frac{pd}{d-p}$, there exists a constant $C>0$ such that, for any $u\in L^p(\Omega;\mu)$,
\begin{equation}\label{eq:sobolev}
\Vert u \Vert_{L^q(\Omega;\mu)} \leq C\left(( \Vert u \Vert_{L^p(\Omega;\mu)} + (\mathcal{E}^p(u))^{1/p}\right). 
\end{equation}
In the theorem that follows we will extend this control to our graph-based setting. We investigate under what conditions there exists a constant $C>0$ such that, for any $n\in\mathbb{N}$ and $u\in L^p(\Omega,\mu_n)$,
\begin{equation}
\label{eq:graphsobolev}
\Vert u \Vert_{L^q(V_n;\mu_n)} \leq C \left( \Vert u \Vert_{L^p(V_n;\mu_n)} + (\mathcal{GE}_n^p(u))^{1/p}\right).    
\end{equation}

As mentioned in the introduction, proving the above Sobolev inequality for any fixed $n\in \mathbb{N}$ is trivial. This is because any two norms on a finite dimensional space must be Lipschitz equivalent. Therefore, in trying to prove~\eqref{eq:graphsobolev} it will be sufficient to assume $n$ is sufficiently large.

We will observe that the existence of such a Sobolev inequality depends on the concentrating parameter sequence $(\varepsilon_n)_{n\in\mathbb{N}}$ and how fast it converges to zero.

\begin{theorem}
\label{thm:graphsobolev}
Let $p,q$ be constants such that $1\leq p<d$ and $p\leq q \leq \frac{dp}{d-p}$. Define $\gamma := \frac{1}{p}-\frac{1}{q}$. There exists a constant $S(\Omega,\eta)>0$, depending only on the kernel $\eta$ and domain $\Omega$, such that the following holds. Assume that (i) for a $K\leq S(\Omega,\eta)$, (ii) and (iii) from Assumptions~\ref{assum} all hold. Additionally assume there exists $t>q$ such that $\aleph_\eta^t<+\infty$. Then the uniform graph Sobolev inequality in~\eqref{eq:graphsobolev} holds if
$$ \left(\frac{\varepsilon_n (\Lambda_n^+)^{1/q+1/p}}{(\Lambda_n^-)^{2/p}} \right)_{n\in\mathbb{N}} \; \; \text{is bounded}.$$
Moreover, if we also assume (iv) from Assumptions~\ref{assum} holds, then the uniform graph Sobolev inequality in~\eqref{eq:graphsobolev} holds if and only if
$$ \left(\varepsilon_n|V_n|^{\gamma}\right)_{n\in\mathbb{N}} \; \; \text{is bounded.} $$

\end{theorem}

\begin{remark}\label{rem:boundedsequences}
It is easy to see that, under (iv) from Assumptions~\ref{assum}, boundedness of $\left(\frac{\varepsilon_n (\Lambda_n^+)^{1/q+1/p}}{(\Lambda_n^-)^{2/p}}\right)_{n\in\mathbb{N}}$ is equivalent to boundedness of $\left(\varepsilon_n|V_n|^{\gamma}\right)_{n\in\mathbb{N}}$. Even without this assumption we have the following useful bound. 

For any $n\in\mathbb{N}$, $\sum_{x\in V_n}\mu_n(x)=1$. Therefore
$ \Lambda_n^-\leq \frac{1}{|V_n|}\leq \Lambda_n^+$ and we compute
$$ \frac{\varepsilon_n (\Lambda_n^+)^{1/q+1/p}}{(\Lambda_n^-)^{2/p}} \geq \frac{\varepsilon_n \left(\frac{1}{|V_n|}\right)^{1/q+1/p}}{\left(\frac{1}{|V_n|}\right)^{2/p}} = \varepsilon_n |V_n|^{\gamma}. $$
As a result, in the statement of Theorem~\ref{thm:graphsobolev}, the assumption that $\left(\frac{\varepsilon_n (\Lambda_n^+)^{1/q+1/p}}{(\Lambda_n^-)^{2/p}}\right)_{n\in\mathbb{N}}$ is bounded is stronger than the assumption that $\varepsilon_n|V_n|^{\gamma}$ is bounded.
\end{remark}

\begin{proof}[Proof of Theorem~\ref{thm:graphsobolev}] To prove the first part of the theorem we assume $\left(\frac{\varepsilon_n (\Lambda_n^+)^{1/q+1/p}}{(\Lambda_n^-)^{2/p}}\right)_{n\in\mathbb{N}}$ is bounded. We start by specifying the constants we use throughout this proof. Take $B,l,l_1,l_2$ and the function $J$ as within the proof of Proposition~\ref{prop:compactness}. Take $b$ as within Remark~\ref{rem:simpletokernel}. We shall also apply Lemma~\ref{lem:ext1} to introduce the extended domain $\Sigma$ with probability measure $\widetilde{\mu}$, the constant $\kappa$, the constant $\delta>0$ and the extension operator $\mathfrak{X}:L^p(\Omega;\mu)\rightarrow L^p(\Sigma,\widetilde{\mu})$. Moreover, we apply Lemma~\ref{lem:ext2} to introduce the constants $\theta \geq 1$ and $B'\geq 1$ and, for each $n\in\mathbb{N}$, the vertex set $\widetilde{V_n}$, the discrete probability measure $\widetilde{\mu_n}$ on $\widetilde{V_n}$, the constant $\widetilde{\Lambda_n}^+$ and the extension operator $\mathfrak{X}_n:L^p(V_n;\mu_n)\rightarrow L^p(\widetilde{V_n},\widetilde{\mu}_n)$. The proposition and lemmas just mentioned can be used since $\aleph_\eta^t<+\infty$. We also define
\begin{equation}\label{eq:constants2}
\alpha := \frac{1}{2l_2B'} \; \; \text{and} \; \; S(\Omega,\eta):= \frac{1}{8B'b\theta}.    
\end{equation}

Let $n\in \mathbb{N}$, $u:V_n\rightarrow \mathbb{R}$ and define $v\in L^p(\Omega)$ by
$$ v(x) := \int_{\mathbb{R}^d} J_{\alpha\varepsilon_n}(x-y)\mathfrak{X}(u\circ T_n)(y) \, dy.$$
For $n$ large enough, the support of $J_{\alpha\varepsilon_n}(x-\cdot)$ is a subset of $\Sigma$, hence, despite $\mathfrak{X}(u\circ T_n)$ not being defined outside of $\Sigma$, the integral above is well defined for such $n$. For the finitely many $n$ which are not large enough, proving the desired Sobolev inequality in equation~\eqref{eq:graphsobolev} is trivial. Since $\theta\geq 1$ and $B'\geq 1$, we have $S(\Omega,\eta)\leq \frac1{4b}$ and $\alpha\leq \frac{1}{2l_2}$; hence we can use the same argument as in Proposition~\ref{prop:compactness} to show that 
$$ \Vert v \Vert_{W^{1,p}(\Omega)}\lesssim_n \Vert u \Vert_{L^p(V_n;\mu_n)}+ (\mathcal{GE}_n^p(u))^{1/p}. $$
Then, by the classical Sobolev inequality in \eqref{eq:sobolev},
\begin{equation}\label{eq:vnormestimate} \Vert v \Vert_{L^q(\Omega)}\lesssim_n  \Vert u \Vert_{L^p(V_n;\mu_n)}+ (\mathcal{GE}_n^p(u))^{1/p}.\end{equation}
Our next job is to bound $\Vert u\circ T_n-v\Vert_{L^q(\Omega)} $. Indeed
\begin{align*}
   &\Vert u\circ T_n-v\Vert_{L^q(\Omega)} \\
   &\leq \left( \int_{\Omega} \int_{\Sigma} J_{\alpha\varepsilon_n}(x-y)\left|\mathfrak{X}(u\circ T_n)(y)-\mathfrak{X}(u\circ T_n)(x)\right|^q \, dy \, dx\right)^{1/q} \\
   &\lesssim_n \left( \int_{\Sigma} \int_{\Sigma} \frac{1}{\varepsilon_n^d} \mathbf{1}_{|x-y|\leq l\alpha\varepsilon_n}\left|\mathfrak{X}(u\circ T_n)(y)-\mathfrak{X}(u\circ T_n)(x)\right|^q \, dy \, dx\right)^{1/q} \\
   &=\varepsilon_n\left( \int_{\Sigma} \int_{\Sigma} \frac{1}{\varepsilon_n^{d+q}} \mathbf{1}_{|x-y|\leq l\alpha\varepsilon_n}\left|\mathfrak{X}(u\circ T_n)(y)-\mathfrak{X}(u\circ T_n)(x)\right|^q \, dy \, dx\right)^{1/q} \\
   &\lesssim_n \varepsilon_n\Vert u \Vert_{L^q(V_n;\mu_n)} + \varepsilon_n\left( \int_{\Omega} \int_{\Omega} \frac{1}{\varepsilon_n^{d+q}} \mathbf{1}_{|x-y|\leq l_1\alpha\varepsilon_n}\left|(u\circ T_n)(y)-(u\circ T_n)(x)\right|^q \, dy \, dx\right)^{1/q} \\
   \end{align*}
The second inequality in the above computation follows since $J_{\alpha \varepsilon_n}$ is bounded and supported in a ball with radius $l \alpha \varepsilon_n$. The third inequality follows as an application of Lemma~\ref{lem:ext1} (with $q$ instead of $p$ with $q$, the indicator function $\mathbf{1_{r\leq 1}}$ in place of $\eta(r)$ and $l\alpha\varepsilon_n$ instead of $\varepsilon$) and Minkowski's inequality. Observe that we are applying this lemma with $q$ in place of $p$, which is why we need to assume there exists a $t>q$ such that $\aleph^t_\eta<+\infty$. We then deduce 
   \begin{align*}
   &\Vert u\circ T_n-v\Vert_{L^q(\Omega)} \\
    &\lesssim_n \varepsilon_n\Vert u \Vert_{L^q(V_n;\mu_n)} + \left( \int_{\Omega} \int_{\Omega} \frac{1}{\varepsilon_n^d} \mathbf{1}_{|T_n(x)-T_n(y)|\leq (l_1\alpha+2K)\varepsilon_n}\left|(u\circ T_n)(y)-(u\circ T_n)(x)\right|^q \, d\mu(y) \,  d\mu(x)\right)^{1/q} \\
    &= \varepsilon_n\Vert u \Vert_{L^q(V_n;\mu_n)} + \varepsilon_n^{-d/q} \left( 2 \sum_{xy\in E_{n, (l_1\alpha+2K)\varepsilon_n}} \left|u(y)-u(x)\right|^q \mu_n(y) \mu_n(x)\right)^{1/q} \\
     &\leq \varepsilon_n\Vert u \Vert_{L^q(V_n;\mu_n)} + \varepsilon_n^{-d/q} (\Lambda_n^+)^{2/q} \left( 2 \sum_{xy\in E_{n, (l_1\alpha+2K)\varepsilon_n}} \left|u(y)-u(x)\right|^q \right)^{1/q}.
\end{align*}
The first inequality in the above computation holds by (iii) from Assumptions~\ref{assum} in combination with, for any $x,y\in\Omega$,
$$ \mathbf{1}_{|x-y|\leq l_1\alpha\varepsilon_n} \leq \mathbf{1}_{|T_n(x)-T_n(y)|\leq (l_1\alpha+2K)\varepsilon_n},  $$ 
which is a consequence of (i) from Assumptions~\ref{assum}.

 Define, for each $n\in\mathbb{N}$, $s_n:=(l_1\alpha+4\theta K)\varepsilon_n$ and the simple graph $\widetilde{\mathbb{G}}_{n,s_n}:=(\widetilde{V_n},\widetilde{E}_{n.s_n})$. If we replace $\Omega$ with $\Sigma$, $V_n$ with $\widetilde{V_n}$ and $\mu_n$ with $\widetilde{\mu_n}$ then (i) from Assumptions~\ref{assum} holds with $\theta K$ in place of $K$. Moreover, if we replace $\rho$ with $\widetilde{\rho}$ then (iii) from Assumptions~\ref{assum} holds with $\kappa$ in place of $D$. Note that $s_n\geq 4\theta K\varepsilon_n$, thus for $n$ sufficiently large we can assume $s_n\in [4\theta K\varepsilon_n,\delta)$. Hence, we can apply Lemma~\ref{lem:graphfriendly} (with the replacements just described) to find that $\widetilde{\mathbb{G}}_{n,s_n}$ is $k$-friendly with respect to $\mathbb{G}_{n,s_n}$ where 
$$  k= \frac{\pi^{d/2}s_n^d}{\widetilde{\Lambda}_n^+\kappa 4^d\Gamma(\frac{d}{2}+1)}-2.$$ With this in hand, we can apply Lemma~\ref{lem:friendlybound} combined with the property $\mathfrak{X}_n u|_{V_n}=u$ and $s_n\geq (l_1\alpha+2K)\varepsilon_n$ to deduce 
\begin{align*}
 \left( \sum_{xy\in E_{n, (l_1\alpha+2K)\varepsilon_n}} \left|u(y)-u(x)\right|^q \right)^{1/q} &\leq \left( \sum_{xy\in E_{n, s_n}} \left|u(y)-u(x)\right|^q \right)^{1/q} \\
&\leq 2^{1-\frac{p}{q}} (k+2)^{\frac{1}{q}-\frac{1}{p}} \left( \sum_{xy\in \widetilde{E}_{n, s_n}} \left|\mathfrak{X}_nu(y)-\mathfrak{X}_nu(x)\right|^p \right)^{1/p} \\
&\lesssim_n \left(\frac{\varepsilon_n^d}{\widetilde{\Lambda}_n^+}\right)^{\frac{1}{q}-\frac{1}{p}} \left( \sum_{xy\in \widetilde{E}_{n, s_n}} \left|\mathfrak{X}_nu(y)-\mathfrak{X}_nu(x)\right|^p \right)^{1/p}.
\end{align*}
Putting this together with our previous bound we get
\begin{align}\label{eq:boundone}
 &\Vert u\circ T_n-v\Vert_{L^q(\Omega)}\notag \\
 &\lesssim_n \varepsilon_n\Vert u \Vert_{L^q(V_n;\mu_n)} + \varepsilon_n^{-d/q} (\Lambda_n^+)^{2/q}\left(\frac{\varepsilon_n^d}{\widetilde{\Lambda}_n^+}\right)^{\frac{1}{q}-\frac{1}{p}} \left( \sum_{xy\in \widetilde{E}_{n, s_n}} \left|\mathfrak{X}_nu(y)-\mathfrak{X}_nu(x)\right|^p \right)^{1/p}\notag \\
&=\varepsilon_n\Vert u \Vert_{L^q(V_n;\mu_n)} +  \frac{\varepsilon_n (\Lambda_n^+)^{2/q}(\widetilde{\Lambda}_n^+)^{\frac{1}{p}-\frac{1}{q}}}{(\widetilde{\Lambda}_n^-)^{2/p}} \left( \frac{1}{\varepsilon_n^{d+p}} \sum_{xy\in \widetilde{E}_{n, s_n}} \left|\mathfrak{X}_nu(y)-\mathfrak{X}_nu(x)\right|^p \widetilde{\Lambda}_n^- \widetilde{\Lambda}_n^- \right)^{1/p} \notag\\
&\lesssim_n \varepsilon_n\Vert u \Vert_{L^q(V_n;\mu_n)} +    \frac{\varepsilon_n (\Lambda_n^+)^{1/q+1/p}}{(\Lambda_n^-)^{2/p}}\left( \frac{1}{\varepsilon_n^{d+p}} \sum_{x\in \widetilde{V_n}}\sum_{y\in\widetilde{V_n}} \mathbf{1}_{|x-y|\leq s_n} \left|\mathfrak{X}_nu(y)-\mathfrak{X}_nu(x)\right|^p \widetilde{\mu_n}(x) \widetilde{\mu_n}(y) \right)^{1/p} \\
&\lesssim_n \varepsilon_n\Vert u \Vert_{L^q(V_n;\mu_n)} +  \left( \frac{1}{\varepsilon_n^{d+p}} \sum_{x\in \widetilde{V_n}}\sum_{y\in\widetilde{V_n}} \mathbf{1}_{|x-y|\leq s_n} \left|\mathfrak{X}_nu(y)-\mathfrak{X}_nu(x)\right|^p \widetilde{\mu_n}(x) \widetilde{\mu_n}(y) \right)^{1/p}.\notag
\end{align}
The second inequality follows from parts (2,ii) and (2,iii) from Lemma~\ref{lem:ext2}. The final bound holds according to part (2) from the extension result in Lemma~\ref{lem:ext2} alongside our assumption that $\left(\frac{\varepsilon_n (\Lambda_n^+)^{1/q+1/p}}{(\Lambda_n^-)^{2/p}}\right)_{n\in\mathbb{N}}$ is bounded.

Recall that $s_n=(l_1\alpha+4\theta K)\varepsilon_n$ and $l_1\alpha+4\theta K\geq \theta K$. By part~(5) from Lemma~\ref{lem:ext2} it then follows that
\begin{align}\label{eq:boundtwo}
&\Vert u\circ T_n-v\Vert_{L^q(\Omega)} \notag\\
&\lesssim_n \varepsilon_n\Vert u \Vert_{L^q(V_n;\mu_n)} + \Vert u \Vert_{L^p(V_n;\mu_n)} +  \left( \frac{1}{\varepsilon_n^{d+p}} \sum_{x\in V_n}\sum_{y\in V_n} \mathbf{1}_{|x-y|\leq (l_1B'\alpha+4B'\theta K)\varepsilon_n} \left|u(y)-u(x)\right|^p \mu_n(x) \mu_n(y) \right)^{1/p} \notag\\
&\lesssim_n \varepsilon_n\Vert u \Vert_{L^q(V_n;\mu_n)} +  \Vert u \Vert_{L^p(V_n;\mu_n)} +  \left( \frac{1}{\varepsilon_n^{p}} \sum_{x\in V_n}\sum_{y\in V_n} \eta_{ (l_2B'\alpha+4B'b\theta K)\varepsilon_n} \left|u(y)-u(x)\right|^p \mu_n(x) \mu_n(y) \right)^{1/p} \notag\\
&\lesssim_n \varepsilon_n\Vert u \Vert_{L^q(V_n;\mu_n)} +   \Vert u \Vert_{L^p(V_n;\mu_n)} +  (\mathcal{GE}_n^p(u))^{1/p}.
\end{align}
The second inequality holds by Remark~\ref{rem:simpletokernel}, where we recall that $b$ is the constant from that remark. The final bound in the above computation holds by our choice of $\alpha$ and $S(\Omega,\eta)$.

Combining the bound above with the one in \eqref{eq:vnormestimate}, we find
$$ \Vert u \Vert_{L^q(V_n;\mu_n)} = \Vert u\circ T_n \Vert_{L^q(\Omega;\mu)} \lesssim_n \Vert v \Vert_{L^q(\Omega)} + \Vert v - u\circ T_n \Vert_{L^q(\Omega)} \lesssim_n \varepsilon_n\Vert u \Vert_{L^q(V_n;\mu_n)}+ \Vert u \Vert_{L^p(V_n;\mu_n)} + (\mathcal{GE}_n^p(u))^{1/p}.$$
Since $\varepsilon_n\rightarrow 0$ as $n\rightarrow\infty$ we see the required Sobolev inequality from \eqref{eq:graphsobolev} holds.

Next we shall handle the second part of the theorem. The `if' part of the statement follows directly by the first observation in Remark~\ref{rem:boundedsequences}. It is then left to prove the `only if' direction. 

For a proof by contradiction we assume that $\left(\varepsilon_n |V_n|^\gamma\right)_{n\in\mathbb{N}}$ is unbounded. By passing to a subsequence we can assume $\varepsilon_n |V_n|^{\gamma}\rightarrow \infty $ as $n\rightarrow \infty$. For each $n\in \mathbb{N}$ select a vertex $x_n\in V_n$ and define a function $u_n:V_n\rightarrow \mathbb{R}$ by $u_n(x_n):=1$ and $u_n(x):=0$ for $x\in V_n\setminus\{x_n\}$. 
Then, by part (iv) of Assumptions~\ref{assum}, we compute $\Vert u_n\Vert_{L^p(V_n;\mu_n)} \lesssim_n |V_n|^{-1/p}$, $\Vert u_n\Vert_{L^q(V_n;\mu_n)}\gtrsim_n|V_n|^{-1/q}$ and 
\begin{align}\label{eq:GEpnbound}
\mathcal{GE}_n^p(u_n)&\lesssim_n \frac{1}{|V_n|\varepsilon_n^p}\sum_{y\in V_n} \eta_{\varepsilon_n}(y-x_n)|u_n(y)-u_n(x_n)|^p\mu_n(y)\notag\\
&= \frac{1}{|V_n|\varepsilon_n^p}\sum_{y\in V_n\setminus \{x_n\}} \eta_{\varepsilon_n}(y-x_n)\mu_n(y) \lesssim_n \frac{1}{|V_n|\varepsilon_n^p}.
\end{align}
The final upper bound follows from Lemma~\ref{lem:weightbound}, which we can use since $\aleph_\eta^t<+\infty$ and thus $\aleph_\eta^0<+\infty$. Then we have
\begin{align*}
    \frac{\Vert u_n \Vert_{L^q(V_n;\mu_n)}}{\Vert u_n \Vert_{L^p(V_n;\mu_n)} + (\mathcal{GE}_n^p(u_n))^{1/p}} \gtrsim_n \frac{\varepsilon_n|V_n|^{\gamma}}{1+\varepsilon_n} \rightarrow \infty.
\end{align*}
Hence the Sobolev inequality in \eqref{eq:graphsobolev} cannot hold. The theorem is now proved.

\end{proof}

Henceforth, in this subsection, we shall fix $S(\Omega,\eta)$ as the constant specified in \eqref{eq:constants2} in the proof of the above theorem. As a consequence of Theorem~\ref{thm:graphsobolev} we can introduce an improvement to the compactness result from \cite[Theorem 4.1]{Dejan2019}. By assuming additional information on the concentrating parameter sequence $(\varepsilon_n)_{n\in\mathbb{N}}$ we are able to improve the integrability of the space in which we deduce compactness. We will provide two different formulations depending on our assumptions.

\begin{corollary}
\label{cor:compactness1}
Let $p,q$ be constants such that $1\leq p<d$ and $p\leq q \leq \frac{dp}{d-p}$. Assume that (i) for a $K\leq S(\Omega,\eta)$, (ii) and (iii) from Assumptions~\ref{assum} all hold. Additionally assume there exists $t>q$ such that $\aleph_\eta^t<+\infty$ and
$$ \left(\frac{\varepsilon_n (\Lambda_n^+)^{1/q+1/p}}{(\Lambda_n^-)^{2/p}}\right)_{n\in\mathbb{N}} \; \; \text{is bounded}.$$

Let $(u_n)_{n\in\mathbb{N}}$ be a sequence of functions $u_n:V_n\rightarrow \mathbb{R}$ such that $(\Vert u_n \Vert_{L^p(V_n;\mu_n)})_{n\in\mathbb{N}}$ and $(\mathcal{GE}_n^p(u_n))_{n\in\mathbb{N}}$ are bounded. Then, for any $r\in [1,q)$, the sequence $\big((u_n,\mu_n)\big)_{n\in \mathbb{N}}$ is relatively compact in $TL^r(\Omega)$. 
\end{corollary}

\begin{proof}
Firstly, observe that $S(\Omega,\eta)\leq H(\eta)$ (since $\theta\geq 1$ and $B'\geq 1$), where $H(\eta)$ is the constant from Proposition~\ref{prop:compactness}. We can then apply Proposition~\ref{prop:compactness} to deduce $\big((u_n,\mu_n)\big)_{n\in\mathbb{N}}$ is relatively compact in $TL^p(\Omega)$. By Theorem~\ref{thm:graphsobolev} we observe \eqref{eq:graphsobolev} holds and thus $(\Vert u_n\Vert_{L^q(V_n;\mu_n)})_{n\in\mathbb{N}}$ is bounded. This proposition and theorem can be used since $\aleph_\eta^t<+\infty$. By Lemma~\ref{lem:tlpholder} it then follows $\big((u_n,\mu_n)\big)_{n\in\mathbb{N}}$ is relatively compact in $TL^r(\Omega)$ for any $r\in (1,q]$.  
\end{proof}

Now the second formulation.

\begin{corollary}
\label{cor:compactness2}
Let $p,\gamma$ be constants such that $1\leq p<d$ and $0\leq \gamma \leq \frac{1}{d}$. Assume that (i) for a $K\leq S(\Omega,\eta)$, (ii), (iii) and (iv) from Assumptions~\ref{assum} all hold. Additionally assume there exists $t>q$ such that $\aleph_\eta^t<+\infty$ and
$$ (\varepsilon_n|V_n|^{\gamma})_{n\in\mathbb{N}} \; \; \text{is bounded} .$$

Let $(u_n)_{n\in \mathbb{N}}$ be a sequence of functions $u_n:V_n\rightarrow \mathbb{R}$ such that $(\Vert u_n \Vert_{L^p(V_n;\mu_n)})_{n\in\mathbb{N}}$ and $(\mathcal{GE}_n^p(u_n))_{n\in\mathbb{N}}$ are bounded. Then, for any $r\in \left[1,\frac{p}{1-p\gamma}\right)$, the sequence $\big( (u_n,\mu_n)\big)_{n\in \mathbb{N}}$ is relatively compact in $TL^r(\Omega)$. 
\end{corollary}

\begin{proof}
Define $q$ such that $\gamma=\frac{1}{p}-\frac{1}{q}$ and note that $q=\frac{p}{1-p\gamma}$ and $1\leq q\leq \frac{pd}{d-p}$.
The argument follows as in the proof of Corollary~\ref{cor:compactness1}. 
\end{proof}

We can provide a third compactness statement assuming a faster rate for the convergence of the parameter sequence $(\varepsilon_n)_{n\in \mathbb{N}}$ to zero. The proof follows a similar sequence of steps as the proof of Theorem~\ref{thm:graphsobolev}, so we will only sketch the proof.

\begin{proposition}
\label{prop:compactness3}
Let $p,q$ be constants such that $1\leq p<d$ and $p\leq q < \frac{dp}{d-p}$. Assume that (i) for a $K\leq S(\Omega,\eta)$, (ii) and (iii) from Assumptions~\ref{assum} all hold. Additionally assume there exists $t>q$ such that $\aleph_\eta^t<+\infty$ and 
$$ \frac{\varepsilon_n (\Lambda_n^+)^{1/q+1/p}}{(\Lambda_n^-)^{2/p}}\rightarrow 0 \; \; \text{as} \; \; n\rightarrow \infty .$$

Let $(u_n)_{n\in\mathbb{N}}$ be a sequence of functions $u_n:V_n\rightarrow \mathbb{R}$ such that $(\Vert u_n \Vert_{L^p(V_n;\mu_n)})_{n\in\mathbb{N}}$ and $(\mathcal{GE}_n^p(u_n))_{n\in\mathbb{N}}$ are bounded. Then the sequence $\big((u_n,\mu_n)\big)_{n\in \mathbb{N}}$ is relatively compact in $TL^q(\Omega)$. 
\end{proposition}

\begin{proof}
As mentioned above, since the proof is similar to that of Theorem~\ref{thm:graphsobolev}, we will just sketch the steps.
\begin{enumerate}
    \item We take the same constants $\alpha$ and $S(\Omega,\eta)$ from \eqref{eq:constants2} in the proof of Theorem~\ref{thm:graphsobolev} and define a sequence $(v_n)_{n\in\mathbb{N}}$ in $L^p(\Omega)$ in the same way as in that proof:
    $$  v_n(x) := \int_{\mathbb{R}^d} J_{\alpha\varepsilon_n}(x-y)\mathfrak{X}(u_n\circ T_n)(y) \, dy. $$
    \item We show that $(v_n)_{n\in\mathbb{N}}$ is bounded in $W^{1,p}(\Omega)$ and by the classical Rellich--Kondrachov compact-embedding result \cite[Theorem 11.10]{Leoni2009} is relatively compact in $L^q(\Omega)$. 
    \item Using the same argument as in the proof of Theorem~\ref{thm:graphsobolev} (in particular equation~\eqref{eq:boundtwo}, where we keep the explicit constant from~\eqref{eq:boundone} rather than absorbing it into its upper bound), we justify that
    \begin{align*}
        \Vert v_n - u_n\circ T_n\Vert_{L^q(\Omega;\mu)}&\lesssim_n \varepsilon_n \Vert u_n\circ T_n \Vert_{L^q(\Omega;\mu)} +\frac{\varepsilon_n (\Lambda_n^+)^{1/q+1/p}}{(\Lambda_n^-)^{2/p}} \left( \Vert u_n \Vert_{L^p(V_n;\mu_n)} +\left( \mathcal{GE}_n^p(u_n) \right)^{1/p} \right).
    \end{align*}
    So there is a constant $c>0$ such that
    \begin{align*}
     \Vert v_n - u_n\circ T_n\Vert_{L^q(\Omega;\mu)} &\leq c\varepsilon_n\Vert u_n\circ T_n \Vert_{L^q(\Omega;\mu)}+  \frac{c\varepsilon_n (\Lambda_n^+)^{1/q+1/p}}{(\Lambda_n^-)^{2/p}} \left( \Vert u_n \Vert_{L^p(V_n;\mu_n)} +\left( \mathcal{GE}_n^p(u_n) \right)^{1/p} \right) \\
     &\leq c\varepsilon_n\Vert v_n - u_n\circ T_n\Vert_{L^q(\Omega;\mu)} + c\varepsilon_n\Vert v_n \Vert_{L^q(\Omega;\mu)} \\
     &+ \frac{c\varepsilon_n (\Lambda_n^+)^{1/q+1/p}}{(\Lambda_n^-)^{2/p}} \left( \Vert u_n \Vert_{L^p(V_n;\mu_n)} +\left( \mathcal{GE}_n^p(u_n) \right)^{1/p} \right).
    \end{align*}
   Therefore,
   $$ (1- c\varepsilon_n)\Vert v_n - u_n\circ T_n\Vert_{L^q(\Omega;\mu)} \leq c\varepsilon_n\Vert v_n \Vert_{L^q(\Omega;\mu)} + \frac{c\varepsilon_n (\Lambda_n^+)^{1/q+1/p}}{(\Lambda_n^-)^{2/p}} \left( \Vert u_n \Vert_{L^p(V_n;\mu_n)} +\left( \mathcal{GE}_n^p(u_n) \right)^{1/p} \right).$$
   Using that $(\Vert v_n\Vert_{L^q(\Omega;\mu)})_{n\in\mathbb{N}}$ is bounded in $L^q(\Omega)$ and, per assumption, $\left(\Vert u_n \Vert_{L^p(V_n;\mu_n)}\right)_{n\in\mathbb{N}}$ and $\left(\mathcal{GE}_n^p(u_n)\right)_{n\in\mathbb{N}}$ are bounded and     
    $$ \frac{\varepsilon_n (\Lambda_n^+)^{1/q+1/p}}{(\Lambda_n^-)^{2/p}}\rightarrow 0 \; \; \text{as} \; \; n\rightarrow \infty,$$
    we conclude
    $$ \Vert v_n - u_n\circ T_n\Vert_{L^q(\Omega)} \rightarrow 0 \; \; \text{as} \; \; n\rightarrow \infty.$$
    \item Putting parts 2 and 3 together we get that $(u_n\circ T_n)_{n\in\mathbb{N}}$ is relatively compact in $L^q(\Omega)$. Therefore, by~\eqref{eq:TLpconvcond} and part~(i) of Assumptions~\ref{assum}, $\big((u_n,\mu_n)\big)_{n\in \mathbb{N}}$ is relatively compact in $TL^q(\Omega)$.
\end{enumerate}
\end{proof}

In the following subsection we shall establish Sobolev inequalities for the $p>d$ setting. We do not consider the case $p=d$ for the same technical reasons that make Sobolev bounds difficult in the continuum setting. None the less, in the classical Sobolev setting various interesting embedding results are known at the $p=d$ threshold. One such embedding is into a suitably chosen Orlicz space~\cite[Theorem 11.29]{Leoni2009}. In order to produce a similar bound in the graph setting one would have to find the correct analogue for the norm of an Orlicz space. This is beyond the scope of the current paper.

\subsection{Graph Sobolev inequalities for \texorpdfstring{$p>d$}{pgreatd}}\label{sec:Sobolevpgreaterd}

In the classical theory of Sobolev spaces, in the case $p\in(d,\infty)$, we have the following famous embedding result~\cite[Theorem 1.4.4.1]{Grisvard1985} (Morrey's inequality). For any $u\in W^{1,p}(\Omega)$, $u$ can be identified with a H\"older-continuous function in $C^{0,\gamma}(\overline{\Omega})$, where, differently than in the previous subsection, $\gamma:=1-\frac{d}{p}$. In other words, $u$ is equal to such a H\"older-continuous function Lebesgue-almost everywhere on $\overline{\Omega}$. Moreover, there exists a constant $C>0$ such that
\begin{equation}\label{eq:Sobolev2}
\Vert u \Vert_{L^\infty(\Omega)}\leq\Vert u \Vert_{C^{0,\gamma}(\overline{\Omega})} \leq C \left( \Vert u \Vert_{L^p(\Omega)} + \left(\mathcal{E}^p(u)\right)^{1/p} \right),
\end{equation}
where,
$$ \Vert u \Vert_{C^{0,\gamma}(\overline{\Omega})} := \sup_{x\in \overline{\Omega}} |u(x)|+\sup_{x,y\in \overline{\Omega}} \frac{|u(x)-u(y)|}{|x-y|^{\gamma}}.$$

It is not clear to the authors in the present moment what is a suitable analogue to the $C^{0,\gamma}(\overline{\Omega})$ norm in a graph setting. For this reason we will consider uniform Sobolev inequalities for the $L^\infty(\Omega)$ norm.
In particular our aim is to investigate the circumstances in which there exists a constant $C>0$ such that, for any $n\in\mathbb{N}$ and $u:V_n\rightarrow \mathbb{R}$,
\begin{equation}
\label{eq:graphsobolev2}
\sup_{x\in V_n} |u(x)| \leq C \left( \Vert u \Vert_{L^p(V_n;\mu_n)} +  (\mathcal{GE}_n^p(u))^{1/p}\right).    
\end{equation}

As we observed in the $p<d$ case, proving the above Sobolev inequality for any fixed $n\in \mathbb{N}$ is trivial. This is because any two norms on a finite dimensional space must be Lipschitz equivalent. Therefore, in trying to prove~\eqref{eq:graphsobolev} it will be sufficient to assume $n$ is sufficiently large.

\begin{theorem}
\label{thm:graphsobolev2}
Let $p$ be a constant such that $p\in (d,\infty)$. There exists a constant $S_\infty(\Omega,\eta)>0$, depending only on $\Omega$ and $\eta$, such that the following holds. Assume that (i) for a $K\leq S_\infty(\Omega,\eta)$, (ii) and (iii) from Assumptions~\ref{assum} all hold. Moreover, assume there exists a $t>p$ such that $\aleph_\eta^t<+\infty$. Then the uniform graph Sobolev inequality in~\eqref{eq:graphsobolev2} holds if
$$ \left(\frac{\varepsilon_n^p\Lambda_n^+}{(\Lambda_n^-)^2} \right)_{n\in\mathbb{N}} \; \; \text{is bounded.} $$
Moreover, if we also assume (iv) from Assumptions~\ref{assum} holds, then the uniform graph Sobolev inequality in~\ref{eq:graphsobolev2} holds if and only if
$$ \left(\varepsilon_n|V_n|^{1/p}\right)_{n\in\mathbb{N}} \; \; \text{is bounded.}$$
\end{theorem}

\begin{remark}\label{rem:boundedsequences2}
    Similar to our observation in Remark~\ref{rem:boundedsequences}, if we assume (iv) from Assumptions~\ref{assum} holds, boundedness of $\left(\frac{\varepsilon_n^p\Lambda_n^+}{(\Lambda_n^-)^2}\right)_{n\in\mathbb{N}}$ is equivalent to boundedness of $\left(\varepsilon_n|V_n|^{1/p}\right)_{n\in\mathbb{N}}$. Without this assumption we simply have
    $$ \frac{\varepsilon_n^p\Lambda_n^+}{(\Lambda_n^-)^2} \geq \left(\varepsilon_n|V_n|^{1/p}\right)^p,$$
    which follows from $\Lambda_n^-\leq \frac{1}{|V_n|}\leq \Lambda_n^+$, as in the second part of Remark~\ref{rem:boundedsequences}.
\end{remark}

\begin{proof}[Proof of Theorem~\ref{thm:graphsobolev2}]
We will prove the first part of the theorem and assume $\left(\frac{\varepsilon_n^p\Lambda_n^+}{(\Lambda_n^-)^2}\right)_{n\in\mathbb{N}}$ is bounded. We begin by specifying the constants we use throughout this proof. Take $B,l,l_1,l_2$ and the function $J$ as within the proof of Proposition~\ref{prop:compactness}. Take $b$ as within Remark~\ref{rem:simpletokernel}. We shall also apply Lemma~\ref{lem:ext1} to introduce the extended domain $\Sigma$ with probability measure $\widetilde{\mu}$, the constant $\kappa$, the constant $\delta>0$ and the extension operator $\mathfrak{X}:L^p(\Omega;\mu)\rightarrow L^p(\Sigma,\widetilde{\mu})$. Moreover, we apply Lemma~\ref{lem:ext2} to introduce the constants $\theta \geq 1$ and $B'\geq 1$ and, for each $n\in\mathbb{N}$, the vertex set $\widetilde{V_n}$, the discrete probability measure $\widetilde{\mu_n}$ on $\widetilde{V_n}$, the constant $\widetilde{\Lambda_n}^+$, the Borel map $\widetilde{T_n}:\Sigma\rightarrow\Sigma$ and the extension operator $\mathfrak{X}_n:L^p(V_n;\mu_n)\rightarrow L^p(\widetilde{V_n},\widetilde{\mu}_n)$. The proposition and lemmas just mentioned can be used since $\aleph_\eta^t<+\infty$. Moreover, we define 
$$  \alpha := \frac{1}{4lB'b} \; \; \text{and} \; \; S_\infty(\Omega,\eta):= \frac{1}{8B'b\theta}.$$ 

Let $n\in \mathbb{N}$ and $u:V_n\rightarrow \mathbb{R}$ and define $v\in L^p(\Omega)$ by
$$ v(x) := \int_{\mathbb{R}^d} J_{\alpha\varepsilon_n}(x-y) (\mathfrak{X}_nu\circ \widetilde{T_n})(y) \, dy .$$
The above integral is well defined assuming $l\alpha\varepsilon_n\leq \delta$, which is satisfied for $n$ sufficiently large. For the rest of the proof we will assume this is indeed the case as proving the inequality in \eqref{eq:graphsobolev2} for any fixed $n$ is trivial.

Using the same argument \footnote{There is a subtle difference between the context here and that of Proposition~\ref{prop:compactness} where we wish to apply to the same argument. We have defined $v$ slightly differently by replacing $\mathfrak{X}(u\circ T_n)$ with $\mathfrak{X}_nu\circ T_n$. None the less, one may apply the same reasoning to reach our desired conclusion, taking care to check that our choice of constants $\alpha$ and $S_\infty(\Omega,\eta)$ are valid for the same computation to hold.} as in the proof of Proposition~\ref{prop:compactness} we deduce that$$ \Vert v \Vert_{W^{1,p}(\Omega)}\lesssim_n \Vert u \Vert_{L^p(V_n;\mu_n)}+ (\mathcal{GE}_n^p(u))^{1/p}. $$
Then, by the classical Sobolev inequality in \eqref{eq:Sobolev2},
\begin{equation*}
\Vert v \Vert_{L^\infty(\Omega)}\lesssim_n  \Vert u \Vert_{L^p(V_n;\mu_n)}+ (\mathcal{GE}_n^p(u))^{1/p}.
\end{equation*}

The next step is to bound $\Vert v-u\circ T_n \Vert_{L^{\infty}(\Omega)}$. Let $x\in\Omega$ and compute
\begin{align*}
    | u\circ T_n(x)-v(x)| &= \left| \int_{\mathbb{R}^d} J_{\alpha\varepsilon_n}(x-y)[(\mathfrak{X}_nu\circ \widetilde{T_n})(x)-(\mathfrak{X}_nu\circ \widetilde{T_n})(y)]dy \right| \\
     &\leq \sup_{y\in B_{l\alpha\varepsilon_n}(x)} |(\mathfrak{X}_nu\circ \widetilde{T_n})(x)-(\mathfrak{X}_nu\circ \widetilde{T_n})(y)|.
\end{align*}
Next we observe, for $y\in B_{l\alpha\varepsilon_n}(x)$,
$$ |\widetilde{T_n}(x)-\widetilde{T_n}(y)|\leq |x-y|+|\widetilde{T_n}(x)-x|+|\widetilde{T_n}(y)-y| \leq (l\alpha+2\theta K)\varepsilon_n,$$
where $\theta$ is the constant from part (4) of Lemma~\ref{lem:ext2}. Set $s:=l\alpha+2\theta K$ and define $\widetilde{\mathbb{G}}_{n,s\varepsilon_n}:=(\widetilde{V_n},\widetilde{E}_{n,s\varepsilon_n})$ such that $yz\in \widetilde{E}_{n,s\varepsilon_n}$ if and only if $|y-z|\leq s\varepsilon_n$. We then observe
\begin{align*}
|u\circ T_n(x)-v(x)|&\leq  \operatorname{osc}_{\widetilde{\mathbb{G}}_{n,s\varepsilon_n}}(\mathfrak{X}_nu)(\widetilde{T_n}(x)) 
\leq  \left(\frac{2^p}{|N_{\widetilde{\mathbb{G}}_{n,s\varepsilon_n}}(x)|} \sum_{yz\in \widetilde{E}_{n,2s\varepsilon_n}} |\mathfrak{X}_nu(y)-\mathfrak{X}_nu(z)|^p\right)^{1/p} .
\end{align*}
The second bound holds by Lemma~\ref{lem:envelopebound} since $\widetilde{\mathbb{G}}_{n,2s\varepsilon_n}$ envelopes $\widetilde{\mathbb{G}}_{n,s\varepsilon_n}$ (see Definition~\ref{def:envelopes}). In the bound that follows we apply Lemma~\ref{lem:degreebound}, replacing $\Omega$ with $\Sigma$, $V_n$ with $\widetilde{V_n}$, $T_n$ with $\widetilde{T_n}$, $K$ with $\theta K$, $D$ with $\kappa$, $\Lambda_n^+$ with $\widetilde{\Lambda_n^+}$ and $\mathbb{G}_{n,s\varepsilon_n}$ with $\widetilde{\mathbb{G}}_{n,s\varepsilon_n}$. For this to be valid we require $s\geq 2\theta K$, which indeed is the case. Moreover, $\widetilde{T_n}(x)=T_n(x)\in \Omega$ and, for $n$ sufficiently large, $s\varepsilon_n\leq \delta$, so, by part (1) of Lemma~\ref{lem:ext2}, $B_{s\varepsilon_n}(\widetilde{T_n}(x))\subset \Sigma$. Therefore
$$ 2|N_{\widetilde{\mathbb{G}}_{n,s\varepsilon_n}}(x)|\geq 1+ |N_{\widetilde{\mathbb{G}}_{n,s\varepsilon_n}}(x)| \geq \frac{\pi^{d/2}s^d\varepsilon_n^d}{\widetilde{\Lambda}_n^+\kappa2^d\Gamma(\frac{d}{2}+1)}.$$

Putting this together with our previous bound we get 
\begin{align}
|u\circ T_n(x)-v(x)|^p &\lesssim_n \frac{\widetilde{\Lambda}_n^+}{\varepsilon_n^d} \sum_{yz\in \widetilde{E}_{n,2s\varepsilon_n}} |\mathfrak{X}_nu(y)-\mathfrak{X}_nu(z)|^p \notag\\
& \lesssim_n \frac{\widetilde{\Lambda}_n^+}{\varepsilon_n^d(\widetilde{\Lambda}_n^-)^2} \sum_{yz\in E_{n,2s\varepsilon_n}} |\mathfrak{X}_nu(y)-\mathfrak{X}_nu(z)|^p \widetilde{\mu_n}(y)\widetilde{\mu_n}(z) \notag\\
& \lesssim_n \frac{\varepsilon_n^p\Lambda_n^+}{(\Lambda_n^-)^2}\cdot \frac{1}{\varepsilon_n^{d+p}}  \sum_{yz\in E_{n,2s\varepsilon_n}} |\mathfrak{X}_nu(y)-\mathfrak{X}_nu(z)|^p \widetilde{\mu_n}(y)\widetilde{\mu_n}(z) \label{eq:boundthree}\\
& \lesssim_n \frac{1}{\varepsilon_n^{d+p}}  \sum_{yz\in E_{n,2s\varepsilon_n}} |\mathfrak{X}_nu(y)-\mathfrak{X}_nu(z)|^p \widetilde{\mu_n}(y)\widetilde{\mu_n}(z) \notag\\
& \lesssim_n \Vert u \Vert^p_{L^p(V_n;\mu_n)} +  \frac{1}{\varepsilon_n^{d+p}} \sum_{y\in V_n} \sum_{z\in V_n} \mathbf{1}_{|y-z|\leq (2B'l\alpha+4B'\theta K)\varepsilon_n} |u(y)-u(z)|^p \mu_n(y)\mu_n(z) \notag\\
& \lesssim_n \Vert u \Vert^p_{L^p(V_n;\mu_n)} + \frac{1}{\varepsilon_n^{p}} \sum_{y\in V_n} \sum_{z\in V_n} \eta_{ (2bB'l\alpha+4bB'\theta K)\varepsilon_n}(|y-z|) |u(y)-u(z)|^p \mu_n(y)\mu_n(z) \notag\\
& \lesssim_n \Vert u \Vert^p_{L^p(V_n;\mu_n)} +  \mathcal{GE}_n^p(u). \label{eq:boundfour}
\end{align}
In the above calculation the second inequality follows from $\frac{\widetilde{\mu_n}}{\tilde\Lambda_n^-}\geq 1$. The third holds according to parts~(2,ii) and~(2,iii) of Lemma~\ref{lem:ext2}. The fourth inequality follows since we have assumed $\left(\frac{\varepsilon_n^p\Lambda_n^+}{(\Lambda_n^-)^2}\right)_{n\in\mathbb{N}}$ is bounded. The fifth holds by part~(5) of Lemma~\ref{lem:ext2}, where we remember that $s=l\alpha+2\theta K$. The sixth inequality follows from the observation in Remark~\ref{rem:simpletokernel}, with $b$ as in that remark. The final line holds by our choice of $\alpha$ and $S_\infty(\Omega,\eta)$. 

We can now piece together our bounds to conclude
$$ \sup_{x\in V_n} |u(x)| = \sup_{x\in \Omega} |u\circ T_n(x)| \leq \sup_{x\in \Omega} |v(x)| + \sup_{x\in \Omega} |v(x) -u\circ T_n(x)| \lesssim_n \Vert u \Vert_{L^p(V_n;\mu_n)}+\left(\mathcal{GE}_n^p(u)\right)^{1/p}.$$

Next we shall handle the second part of the theorem. From Remark~\ref{rem:boundedsequences} we know, assuming that (iv) from Assumptions~\ref{assum} holds, that boundedness of $\left(\frac{\varepsilon_n^p\Lambda_n^+}{(\Lambda_n^-)^2}\right)_{n\in\mathbb{N}}$ is equivalent to boundedness of $\left(\varepsilon_n|V_n|^{1/p}\right)_{n\in\mathbb{N}}$. Hence the `if' part of the second claim of the theorem follows from the first claim that we have just proven. It is left to prove the `only if' direction. We will mimic the last part of our proof of Theorem~\ref{thm:graphsobolev}.

For a proof by contradiction we assume that $\left(\varepsilon_n |V_n|^{1/p}\right)_{n\in\mathbb{N}}$ is unbounded. By passing to a subsequence we can assume $\varepsilon_n |V_n|^{1/p}\rightarrow \infty $ as $n\rightarrow \infty$. For each $n\in \mathbb{N}$ select a vertex $x_n\in V_n$ and define a function $u_n:V_n\rightarrow \mathbb{R}$ by $u_n(x_n):=1$ and $u_n(x):=0$ for $x\in V_n\setminus\{x_n\}$. 
Then we compute $\Vert u_n\Vert_{L^p(V_n;\mu_n)} \lesssim_n |V_n|^{-1/p}$ (by part (iv) of Assumptions~\ref{assum}), $\Vert u_n\Vert_{L^\infty(V_n;\mu_n)}=1$ and, by \eqref{eq:GEpnbound}, $\mathcal{GE}_n^p(u_n) \lesssim_n \frac{1}{|V_n|\varepsilon_n^p}$. 

Thus we have
\begin{align*}
    \frac{\Vert u_n \Vert_{L^\infty(V_n;\mu_n)}}{\Vert u_n \Vert_{L^p(V_n;\mu_n)} + (\mathcal{GE}_n^p(u_n))^{1/p}} \gtrsim_n \frac{\varepsilon_n |V_n|^{1/p}}{1+\varepsilon_n}\rightarrow \infty \; \; \text{as} \; \; n\rightarrow\infty.
\end{align*}
Hence the Sobolev inequality in \eqref{eq:graphsobolev2} cannot hold. The theorem is now proved.
\end{proof}

Henceforth, in this subsection, we shall fix $S_\infty(\Omega,\eta)$ as specified in the proof of Theorem~\ref{thm:graphsobolev2} above. As we did in the previous subsection for the case $p<d$, we can produce a number of compactness results in the case $p>d$.

\begin{corollary}
\label{cor:compactness4}
    Let $ p\in (d,\infty)$. Assume that (i) for a $K\leq S_\infty(\Omega,\eta)$, (ii) and (iii) from Assumptions~\ref{assum} all hold. Additionally assume there exists $t>p$ such that $\aleph_\eta^t<+\infty$ and
$$ \left(\frac{\varepsilon_n^p\Lambda_n^+}{(\Lambda_n^-)^2} \right)_{n\in \mathbb{N}}\; \; \text{is bounded.}$$

Let $(u_n)_{n\in\mathbb{N}}$ be a sequence of functions $u_n:V_n\rightarrow \mathbb{R}$ such that $(\Vert u_n \Vert_{L^p(V_n;\mu_n)})_{n\in\mathbb{N}}$ and $(\mathcal{GE}_n^p(u_n))_{n\in\mathbb{N}}$ are bounded. Then, for any $r\in [1,+\infty)$, the sequence $\big((u_n,\mu_n)\big)_{n\in \mathbb{N}}$ is relatively compact in $TL^r(\Omega)$. 
\end{corollary}

\begin{proof}
 Firstly we can apply Proposition~\ref{prop:compactness} to note $\big((u_n,\mu_n)\big)_{n\in \mathbb{N}}$ is relatively compact in $TL^p(\Omega)$. Next we can apply Theorem~\ref{thm:graphsobolev2} to observe that \eqref{eq:graphsobolev2} holds and thus $(\Vert u_n \Vert_{L^{\infty}(V_n;\mu_n)})_{n\in \mathbb{N}}$ is bounded. We can use this proposition and theorem since $\aleph_\eta^t<+\infty$. Then, as a consequence of Lemma~\ref{lem:tlpholder}, for any $r\in [1,+\infty)$, the sequence $\big((u_n,\mu_n)\big)_{n\in \mathbb{N}}$ is relatively compact in $TL^r(\Omega)$.     
\end{proof}

We also have a formulation of the above Corollary when we assume (iv) from Assumptions~\ref{assum}.

\begin{corollary}
\label{cor:compactness6}
Let $p$ be a constant such that $ p\in (d,\infty)$. Assume that (i) for a $K\leq S_\infty(\Omega,\eta)$, (ii), (iii) and (iv) from Assumptions~\ref{assum} all hold. Additionally assume there exists $t>p$ such that $\aleph_\eta^t<+\infty$ and
$$ (\varepsilon_n|V_n|^{1/p})_{n\in\mathbb{N}} \; \; \text{is bounded} .$$

Let $(u_n)_{n\in \mathbb{N}}$ be a sequence of functions $u_n:V_n\rightarrow \mathbb{R}$ such that $(\Vert u_n \Vert_{L^p(V_n;\mu_n)})_{n\in\mathbb{N}}$ and $(\mathcal{GE}_n^p(u_n))_{n\in\mathbb{N}}$ are bounded. Then, for any $r\in \left[1,\frac{p}{1-p\gamma}\right)$, the sequence $\big( (u_n,\mu_n)\big)_{n\in \mathbb{N}}$ is relatively compact in $TL^r(\Omega)$. 
\end{corollary}

We can achieve a stronger compactness result under the assumption  a faster convergence rate to zero for the parameter sequence $(\varepsilon_n)_{n\in \mathbb{N}}$. The proof follows a similar sequence of steps as in the proof of Theorem~\ref{thm:graphsobolev2}, therefore we will only sketch the details.

\begin{proposition}
\label{prop:compactness5}
Let $p$ be a constant such that $ p\in (d,\infty)$. Assume that (i) for a $K\leq S_\infty(\Omega,\eta)$, (ii) and (iii) from Assumptions~\ref{assum} all hold. Additionally assume there exists $t>p$ such that $\aleph_\eta^t<+\infty$ and
$$ \frac{\varepsilon_n^p\Lambda_n^+}{(\Lambda_n^-)^2} \rightarrow 0 \; \; \text{as} \; \; n\rightarrow \infty .$$

Let $\left(u_n\right)_{n\in\mathbb{N}}$ be a sequence of functions $u_n:V_n\rightarrow \mathbb{R}$ such that $\{\Vert u_n \Vert_{L^p(V_n;\mu_n)}\}_{n\in\mathbb{N}}$ and $\{\mathcal{GE}_n^p(u_n)\}_{n\in\mathbb{N}}$ are bounded. Then, after possibly passing to a subsequence, there exists a $u\in C^{0,1-\frac{d}{p}}(\overline{\Omega})$ such that
$$ \sup_{x\in \Omega} |u_n\circ T_n(x) -u(x)|\rightarrow 0 \; \; \text{as} \; \; n\rightarrow \infty .$$
\end{proposition}

\begin{proof}
    As previously mentioned we will just sketch the steps.
\begin{enumerate}
    \item We take the same constants $\alpha$ and $S_\infty(\Omega,\eta)$ from the proof of Theorem~\ref{thm:graphsobolev2} and define a sequence $(v_n)_{n\in\mathbb{N}}$ in $L^p(\Omega)$ by
    $$ v_n(x) := \int_{\mathbb{R}^d} J_{\alpha\varepsilon_n}(x-y) (\mathfrak{X}_nu\circ \widetilde{T_n})(y) \, dy .$$ 
    \item Using the same argument as in the proof of Theorem~\ref{thm:graphsobolev2} (i.e. the same argument as in the proof of Proposition~\ref{prop:compactness}) we show that $(v_n)_{n\in\mathbb{N}}$ is bounded in $W^{1,p}(\Omega)$ and by the classical Rellich--Kondrachov Theorem~\cite[Part III of Theorem 6.3]{AdamsFournier03} is relatively compact in $L^\infty(\Omega)$. We can pass to a subsequence to assume, without loss of generality, that there exists a $u\in C^{0,1-\frac{d}{p}}(\overline{\Omega})$ such that $v_n\rightarrow u$ in $L^{\infty}(\Omega)$.
    \item Using the same argument in Theorem~\ref{thm:graphsobolev2} (in particular~\eqref{eq:boundfour}, where we keep the explicit constant from~\eqref{eq:boundthree} rather than absorbing it into its upper bound), we justify that
    \begin{align*}
        \Vert v_n - u_n\circ T_n\Vert_{L^\infty(\Omega)}&\lesssim_n  \frac{\varepsilon_n^p\Lambda_n^+}{(\Lambda_n^-)^2} \left(  \Vert u_n \Vert_{L^p(V_n;\mu_n)} +\left( \mathcal{GE}_n^p(u_n) \right)^{1/p} \right) \rightarrow 0 \; \; \text{as} \; \; n\rightarrow \infty .
    \end{align*}
    \item Putting parts 2 and 3 together we get
    $$ \sup_{x\in \Omega}| u_n\circ T_n(x) -u(x)| \rightarrow 0 \; \; \text{as} \; \; n\rightarrow \infty.$$
\end{enumerate}
\end{proof}

\section{Discussion}
\label{sec:discussion}

In this final section we will discuss some consequences of the theoretical results derived in this paper alongside some related topics for possible future research.

\subsection{Connectivity of geometric graphs}\label{sec:connectivity}

Throughout this paper the concentrating parameter sequence $(\varepsilon_n)_{n\in\mathbb{N}}$ plays a pivotal role. The discrete variation $\mathcal{GE}_n^p$ cannot be useful as a tool for regularisation if the underlying graph $\mathbb{G}_{n,\varepsilon_n}$ is disconnected.

There is a close relationship between the concentrating parameter $\varepsilon_n$ and the connectivity of $\mathbb{G}_{n,\varepsilon_n}$. For example, if $\Omega$ is bi-Lipschitz homeomorphic to a convex set and (i) from Assumptions~\ref{assum} holds for $K\leq\frac{1}{C(\Omega)}$, where $C(\Omega)$ is the constant from Corollary~\ref{cor:connect2}, then by that same corollary the graph $\mathbb{G}_{n,\varepsilon_n}$ is connected. As we mentioned in Remark~\ref{rem:conndinfty}, this result suggests a close relationship between the connectivity of geometric graphs and the $\infty$-Wasserstein distance. A further investigation of this relationship is an interesting direction for future research.

In foundational works such as \cite{garcia2016continuum, Dejan2019}, the authors assume (i$^*$) from Assumptions~\ref{assum} holds instead of (i). This guarantees that the concentrating parameter $\varepsilon_n$, for large $n\in\mathbb{N}$, is much larger than the connectivity threshold $s^*(V_n)$. 

For certain applications the assumption (i$^*$) has advantages. For example, in the recent work~\cite{trillos2025} the authors reconstruct the eigenvalues and eigenfunctions of the Laplace--Beltrami operator on an unknown manifold with an unknown probability measure from finitely many sampled points. They find that a weighted graph Laplacian constructed from the samples with a concentrating parameter $\varepsilon_n$ satisfying (i$^*$) gives the best error estimate possible (in the best-worst-case sense of a minimax error, see~\cite[Equation (1.8)]{trillos2025}). Moreover, as we shall illustrate in Appendix A by constructing a counterexample, (i) alone is not sufficient to conclude $\Gamma$-convergence of the discrete variations $\mathcal{GE}_n^p$ towards an isotropic limit. 

On the other hand, some numerical experiments suggest that the performance when using $\mathcal{GE}_n^p$ as a regularizer, is best (in some sense) when $\varepsilon_n$ is chosen as close to $s^*(V_n)$ as possible. For example in \cite[Section 6]{Dejan2019}, $\mathcal{GE}_n^p$ is used as a regularizer in a data-labelling problem and the performance is measured in terms of the frequency of mislabelling. Unfortunately this choice of $\varepsilon_n$ falls outside of the scope of (i$^*$). Surprisingly, in the Poisson-point-cloud setting (see \cite[Theorem 3.3]{Braides2023}) and for $p=2$, an assumption weaker than (i) is actually sufficient to observe $\Gamma$-convergence. The language and notation used in~\cite{Braides2023} is naturally different to that used in our paper, but the reader should find the following heuristic useful. What we call the concentrating parameter `$\varepsilon_n$' corresponds to the product `$\varepsilon\lambda$' of two parameters in that paper. The authors require $\varepsilon\lambda$ to be large enough for percolation to occur, as described in \cite[Theorem 2.6]{Braides2023}, which is strictly smaller than what is required for the point cloud to be connected.  

Thus, there is need to further understand the consequences of different choices of the concentrating parameter $\varepsilon_n$ on analytical properties. The Sobolev inequalities help to bridge this gap in the literature by providing quantitative estimates on the regularising effects of the discrete variations $\mathcal{GE}_n^p$ close to the connectivity threshold, even when the requirements for $\Gamma$-convergence fail.

\subsection{Extension results}\label{sec:extension}

In Lemmas~\ref{lem:ext1} and~\ref{lem:ext2} we have introduced two new extension results. Lemma~\ref{lem:ext1} introduces an extension operator that preserves the control of a collection of nonlocal variations in a uniform way. Results similar to Lemma~\ref{lem:ext1} exist in the literature; for example, in the proof of \cite[Lemma 4.4]{garcia2016continuum} an extension is constructed which applies when $p=1$ and the domain has $C^2$ boundary. Our result holds for $p\geq 1$ and for domains with (weak) Lipschitz boundary. Lemma~\ref{lem:ext2} is a discretized version of Lemma~\ref{lem:ext1} which states we can extend both the graphs and functions on said graphs in a way that preserves the control on the discrete variations, again uniformly. In order to create our new discretized graphs we utilize Lemma~\ref{lem:discretization} which states we can construct a particularly nice discretization for the domain of the extended function in Lemma~\ref{lem:ext1}. 

These results are applied in this paper a few times in two different ways.
\begin{enumerate}
    \item The first is to construct mollifications. In order to approximate a function we mollify it by computing an average in a small local neighbourhood of each point in its domain. Such a mollifier is not suitably defined for points too close to the boundary. A typical strategy to overcome this is to introduce an extension operator. This is utilized in the proofs of Proposition~\ref{prop:compactness} and Theorems~\ref{thm:graphsobolev} and~\ref{thm:graphsobolev2}.
    \item The second is for combinatorial bounds. For example, in Lemmas~\ref{lem:graphfriendly} and~\ref{lem:degreebound} we can only establish the friendly property and the bound on the degrees, respectively, when we are fixed distance away from the boundary. Without a theory of extensions, this is not sufficient to prove our desired Sobolev inequalites over the full domain. Other authors have run into this challenge. For example, in \cite[Lemma 4.5]{Dejan2019}, the result is only stated over a compact subset contained inside the domain $\Omega$. We overcome this difficulty by introducing the discrete extension operators in Lemma~\ref{lem:ext2}.
\end{enumerate}

\subsection{Applications of Sobolev inequalities}
\label{sec:applications}

In this subsection we will investigate the assumptions in and consequences of Theorem~\ref{thm:graphsobolev}, the graph Sobolev inequalities for $p<d$, in a few different scenarios. We shall leave out a discussion of the $p>d$ case, as that topic is extensively covered in \cite{Dejan2019}.

Before we proceed, let us remark on some common choices of the function $\eta$ within the literature:  
\begin{itemize}
    \item an indicator function, $\eta(r)= \mathbf{1}_{r\leq 1}$, where $\aleph_\eta^t<+\infty$ for any $t\in [0,+\infty)$;
    \item an exponential function, $\eta(r)= e^{-r}$, where $\aleph_\eta^t<+\infty$ for any $t\in [0,+\infty)$;
    \item a power function, $\eta(r)= (1+r)^{-a} $ and $a>d$, where $\aleph_\eta^t<+\infty$ for any $t\in [0,a-d)$. 
\end{itemize}
Recall that for the $(p,q)$-Sobolev inequality in Theorem~\ref{thm:graphsobolev} to hold we require, for some $t>q$, $\aleph_\eta^t<+\infty$. In all of the applications that follow we shall assume that $\eta$ is such that this is the case.

\begin{enumerate}
    \item In order for the strongest Sobolev inequality that Theorem~\ref{thm:graphsobolev} can give us to hold, that is \eqref{eq:graphsobolev} for $q=\frac{dp}{d-p}$, we require 
    $\left(\frac{\varepsilon_n (\Lambda_n^+)^{1/q+1/p}}{(\Lambda_n^-)^{2/p}} \right)_{n\in\mathbb{N}}$ to be bounded. Recall from Remark~\ref{rem:boundedsequences} that
    $$ \varepsilon_n|V_n|^{\frac{1}{d}}=\varepsilon_n |V_n|^{\frac{1}{p}-\frac{1}{q}}\leq \frac{\varepsilon_n (\Lambda_n^+)^{1/q+1/p}}{(\Lambda_n^-)^{2/p}}.$$
    Putting this together with part~(i) from Assumptions~\ref{assum}, we get that there exists a constant $A>0$ such that
    $$ d_\infty(\mu_n,\mu)\leq \sup_{x\in\Omega} \{|T_n(x)-x|\}\leq K\varepsilon_n \leq AK|V_n|^{-\frac{1}{d}} .$$
    We can put this together with the bound in Lemma~\ref{lem:transportbound} to deduce
    $$ \frac{\Gamma(\frac{d}{2}+1)^{1/d}\mathfrak{L}^d(\Omega)^{1/d}}{K\pi^{1/2}|V_n|^{1/d}}\leq \frac{1}{K} d_\infty(\mu_n,\mu)\leq \varepsilon_n \leq \frac{A}{|V_n|^{1/d}}.$$
   This tells us that both $d_\infty(\mu_n,\mu)$ and $\varepsilon_n$ must converge to zero at a rate $|V_n|^{-\frac{1}{d}}.$ This can be achieved, for example, by a square lattice discretization and with $\varepsilon_n$ the width of the tiles. 

    \item We shall investigate what is the strongest Sobolev inequality we can prove, with the techniques provided in this paper, in circumstances where part~(iv) from Assumptions~\ref{assum} does not hold. Let us assume there exist constants $\alpha \in (0,1]$ and $\beta\in [1,+\infty)$ such that $\Lambda_n^+$ and $\Lambda_n^-$ are of the orders $|V_n|^{-\alpha}$ and $|V_n|^{-\beta}$, respectively. Then, analogously to what we observed in Remark~\ref{rem:boundedsequences}, boundedness of $\left(\frac{\varepsilon_n (\Lambda_n^+)^{1/q+1/p}}{(\Lambda_n^-)^{2/p}} \right)_{n\in\mathbb{N}}$ is equivalent to boundedness of $\left(\varepsilon_n |V_n|^{\frac{2\beta-\alpha}{p}-\frac{\alpha}{q}} \right)_{n\in\mathbb{N}}$. As shown in point 1 above, as a consequence of (i) from Assumptions~\ref{assum} and Lemma~\ref{lem:transportbound}, the fastest $\left(\varepsilon_n\right)_{n\in\mathbb{N}}$  can converge to zero is of the order $|V_n|^{-1/d}$. Therefore we require
    $$ \frac{2\beta-\alpha}{p}-\frac{\alpha}{q} \leq \frac{1}{d} .$$
    Rearranging we get $q\leq \frac{\alpha pd}{(2\beta-\alpha)d-p}$. Therefore, in this setting, the highest exponent for which we can produce the Sobolev inequality in \eqref{eq:graphsobolev} is $q=\frac{\alpha pd}{(2\beta-\alpha)d-p}$, which is less than or equal to $\frac{pd}{d-p}$ with equality (in the admissible set $(\alpha,\beta)\in (0,1]\times[1,+\infty)$) if and only if $\alpha=\beta=1$.

    \item Consider the random-point-cloud setup in Remark~\ref{rem:pointcloud} where $|V_n|=n$ and part~(iv) from Assumptions~\ref{assum} holds with $\mathfrak{D}=1$. Recall from Remark~\ref{rem:pointcloud} the following bound on the $\infty$-Wasserstein distance between $\mu_n$ and $\mu$: there exists a sequence of Borel maps $T_n:\Omega\rightarrow\Omega$ such that $T_n\#\mu=\mu_n$ and
    \begin{equation*}
    \sup_{x\in\Omega} | T_n(x) -x | \leq C \begin{cases}
        \left( \frac{\log \log n}{n} \right)^{1/2}, &\text{if } d=1, \\
        \frac{\log(n)^{3/4}}{n^{1/2}}, &\text{if } d=2,\\
        \frac{\log(n)^{1/d}}{n^{1/d}}, &\text{if } d\geq 3.
    \end{cases}
\end{equation*}

Moreover the above bound is the best rate of convergence possible \cite[Theorem 1.1]{Nicolas2015}. Therefore, the smallest we can choose $\varepsilon_n$ while still guaranteeing the validity of part~(i) in Assumptions~\ref{assum} is
$$ \varepsilon_n = \frac{C}{K}\begin{cases}
        \left( \frac{\log \log n}{n} \right)^{1/2}, &\text{if } d=1, \\
        \frac{(\log n)^{3/4}}{n^{1/2}}, &\text{if } d=2,\\
        \frac{(\log n)^{1/d}}{n^{1/d}}, &\text{if } d\geq 3.
    \end{cases}$$
In this case $\left(\varepsilon_n|V_n|^{\frac{1}{p}-\frac{1}{q}}\right)_{n\in \mathbb{N}}$ is bounded if and only if $q<\frac{dp}{d-p}$. That means we can achieve Sobolev inequalities all the way up to, but not including, $\frac{dp}{d-p}$.

    \item In the following example we shall use a trick that may be useful for other geometric graphs. Consider a standard lattice discretization of $(0,1)^d$ by $|V_n|=n^d$ points. We interpret this lattice as a simple graph $\mathbb{G}_n:=(V_n,E_n)$ where the vertices are given by the centres of each tile and two vertices are adjacent if their tiles are adjacent. This is illustrated below, the black dots represent the vertices and the black lines the edges. Define $\mu_n$ to be the uniform discrete probability measure on $V_n$ and $\mu$ the uniform probability measure on $(0,1)^d$. 

    We wish to interpret this graph as a geometric graph, but this requires the choice $\varepsilon_n=\frac1n$, which does not satisfy the requirements of Theorem~\ref{thm:graphsobolev}. Below we will choose a larger $\varepsilon_n$ for which the theorem can be used and then relate back the resulting Sobolev inequality to an inequality on the graph.

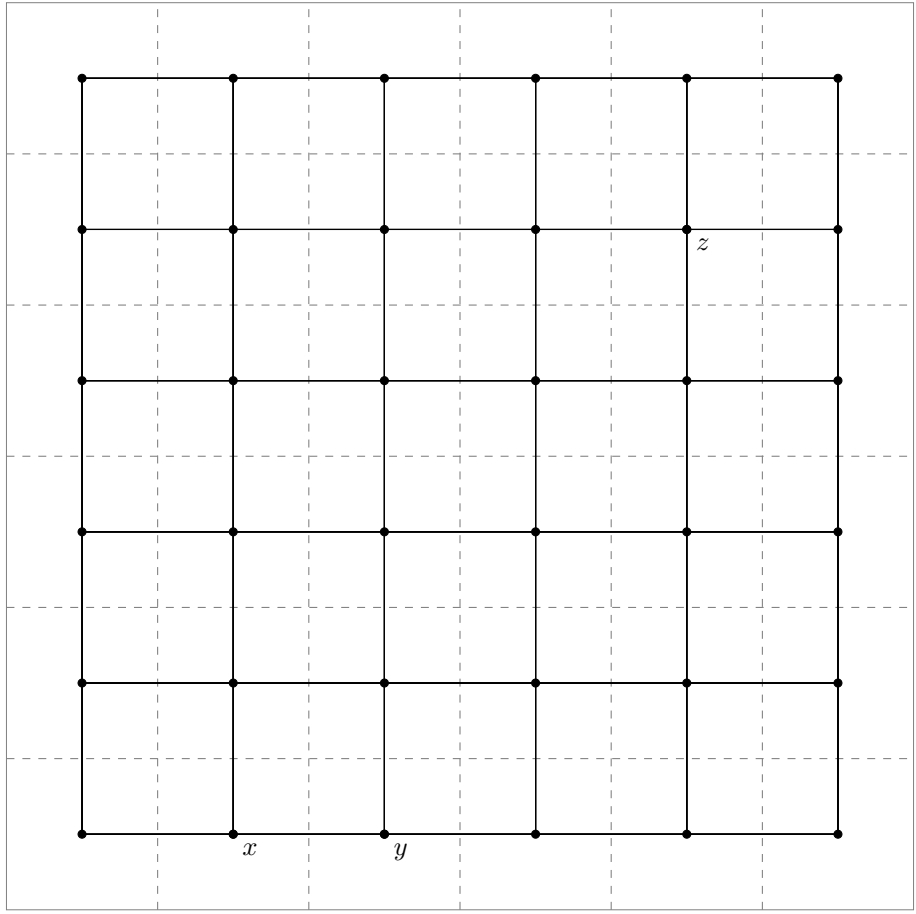
\begin{figure}[ht]
    \begin{center}
    \begin{tikzpicture}[scale=1]
    \draw[gray, very thin] (0,0) rectangle (12,12) ;
\foreach \x in {2,4,6,8,10}
\draw[gray,dashed] (\x,0) -- (\x,12) ;
\foreach \y in {2,4,6,8,10}
\draw[gray,dashed] (0,\y) -- (12,\y) ;
\foreach \x in {1,3,5,7,9,11}
\foreach \y in {1,3,5,7,9,11}
{
\filldraw[black] (\x,\y) circle (1.5pt) ;
\draw[thin] (\x,1) -- (\x,11) ;
\draw[thin] (1,\y) -- (11,\y) ;
}
\filldraw[black] (3,1) circle (1.5pt) node[anchor=north west]{$x$} ;
\filldraw[black] (5,1) circle (1.5pt) node[anchor=north west]{$y$} ;
\filldraw[black] (9,9) circle (1.5pt) node[anchor=north west]{$z$} ;
\end{tikzpicture}
\end{center}
    \caption{Lattice discretization for $d=2$, $n=6$}
    \label{fig:lattice}
\end{figure}

Define $T_n:(0,1)^d\rightarrow(0,1)^d$ by sending each point in the interior of a tile (depicted in Figure~\ref{fig:lattice} by dashed lines) to the vertex in its centre. We shall refrain from defining $T_n$ on the dashed lines themselves, this is not a problem as they have $\mu$-measure zero. Additionally define $\eta(r):=\mathbf{1}_{r\leq 1}$. Recall with this choice, $\aleph_{\eta}^t<+\infty$ for any $t\in[0,+\infty)$. With this setup (ii), (iii) and (iv) from Assumptions~\ref{assum} are all satisfied. Moreover $T_n\#\mu = \mu_n$ and also
$$\sup_{x\in (0,1)^d}|T_n(x)-x|=\frac{\sqrt{d}}{2n}=\frac{\sqrt{d}}{2|V_n|^{1/d}}.$$  We shall set $q:=\frac{pd}{d-p}$.    

Define $\mathbb{d}_n$ to be the graph metric on $\mathbb{G}_n$: the distance between vertices $x$ and $y$ is the least number of edges that have to be traversed to go from $x$ to $y$ in the graph. Thus in the example of Figure~\ref{fig:lattice}, since $x$ and $y$ are adjacent vertices, $\mathbb{d}_n(x,y)=1$, and $\mathbb{d}_n(x,z)=7$.  

Set $\widetilde S:=S([0,1]^d,\eta)$ and define $\varepsilon_n:=\frac{\sqrt{d}}{2\widetilde S n} $\footnote{From~\eqref{eq:constants2} we see that $\widetilde S \leq \frac18$, hence $\varepsilon_n \geq \frac4n$. As announced earlier in this example, this choice of $\varepsilon_n$ is larger than the choice that leads to a geometric-graph interpretation of the lattice. }.With these definitions, we obtain that $\sup_{x\in [0,1]^d}|T_n(x)-x| = \widetilde S\varepsilon_n$; thus part~(i) of Assumptions~\ref{assum} is satisfied alongside all of the requirements for the second statement of Theorem~\ref{thm:graphsobolev}. Since we also have $\varepsilon_n |V_n|^{1/d}= \frac{\sqrt{d}}{2\widetilde S}$, which is constant in $n$, Theorem~\ref{thm:graphsobolev} tells us that \eqref{eq:graphsobolev} holds: there exists a constant $C>0$ such that, for any $u:V_n\rightarrow \mathbb{R}$,
$$ \Vert u \Vert_{L^q(V_n;\mu_n)} \leq C \left( \Vert u \Vert_{L^p(V_n;\mu_n)} + \left(\frac{1}{n^{2d+p}}\sum_{x\in V_n}\sum_{\substack{y\in V_n \\ |x-y|\leq \varepsilon_n} } |u(x)-u(y)|^p\right)^{1/p} \right)    .$$
We shall bound the above double sum by a double sum written purely in terms of the graph structure of $\mathbb{G}_n$. We can achieve this since, for any $x,y\in V_n$, $ \mathbb{d}_n(x,y) \leq dn|x-y|$, and therefore $|x-y|\leq \varepsilon_n$ implies $\mathbb{d}_n(x,y)\leq  \frac{d\sqrt{d}}{2\widetilde S}$ and thus $\mathbb{d}_n(x,y)\leq \left\lfloor\frac{d\sqrt{d}}{2\widetilde S}\right\rfloor$, since $\mathbb{d}_n(x,y)$ is an integer. Define $N:= \left\lfloor \frac{d\sqrt{d}}{2\widetilde S} \right\rfloor$. We then have
$$ \Vert u \Vert_{L^q(V_n;\mu_n)} \leq C \left( \Vert u \Vert_{L^p(V_n;\mu_n)} + \left(\frac{1}{n^{2d+p}}\sum_{x\in V_n} \sum_{\substack{y\in V_n \\ \mathbb{d}_n(x,y)\leq N} } |u(x)-u(y)|^p\right)^{1/p} \right)    .$$
The above bound raises the question if we need to sum over all pairs of vertices $x$ and $y$ that lie within a distance $N$ of each other or if the inequality still holds if we only sum over adjacent vertices? Indeed the bound is then still valid, as we will now show. 

Let $s \in \{ 1,2,\ldots, N\}$ and consider an edge $e\in E_n$. By a path we mean a finite tuple of consecutively connected vertices without repetitions. We will find an upper bound for the number of paths of length $s$ that include the edge $e$, by building a path starting with the edge $e$ and iteratively adding edges at either end of the previously built path. At each step we have at most $2(1+2(d-1))=4d-2$ possible edges we can add: at each of the endpoints of the already existing path, there is at most one possible edge to add that continues along the dimension of the existing path and at most two possible edges in each of the other $d-1$ dimensions. Therefore we have at most $(4d-2)^{s-1}$ paths of length $s$ that include the edge $e$. Let $P$ be a path of length $s$ in $\mathbb{G}_n$ given by a sequence of vertices $x_0,x_1,\ldots,x_s$. Then, by Jensen's inequality,
$$  |u(x_0)-u(x_s)|^p = s^p\left|\frac{1}{s} \sum_{i=0}^{s-1} u(x_i)-u(x_{i+1}) \right|^p \leq s^{p-1}\sum_{i=0}^{s-1}| u(x_i)-u(x_{i+1})|^p .$$
For a general path $P$ we denote by $|P|$ its length, i.e. the number of vertices in the path, by $P_{\text{in}}$ its initial vertex and by $P_{\text{te}}$ its terminal vertex. We then compute
\begin{align*}
 \frac{1}{n^{2d+p}} \sum_{x\in V_n} \sum_{\substack{y\in V_n \\ \mathbb{d}_n(x,y)\leq N} } |u(x)-u(y)|^p\
&\leq \frac{1}{n^{2d+p}}\sum_{P: \; |P|\leq N }  |u(P_{\text{in}})-u(P_\text{te})|^p
\\
&\leq  \frac{N^{p-1}}{n^{2d+p}}\sum_{P: \; |P|\leq N } \sum_{xy\in P} |u(x)-u(y)|^p  
\\
&\leq   \frac{N^{p-1}}{n^{2d+p}}\sum_{s=1}^N (4d-2)^{s-1}\sum_{xy\in E_n} |u(x)-u(y)|^p 
\\
&=   \frac{N^{p-1}\left[(4d-2)^N-1\right]}{(4d-3)n^{2d+p}}\sum_{xy\in E_n} |u(x)-u(y)|^p.
\end{align*}
Therefore we have the bound
$$ \Vert u \Vert_{L^q(V_n;\mu_n)} \leq C \left( \Vert u \Vert_{L^p(V_n;\mu_n)} + \left(\frac{N^{p-1}\left[(4d-2)^N-1\right]}{(4d-3)} \frac{1}{n^{2d+p}}\sum_{xy\in E_n} |u(x)-u(y)|^p\right)^{1/p} \right)    .$$

Similar Sobolev inequalities on lattices exist in the literature, for example \cite[Theorem 4.2]{Fusco2025}. \cite{Chung1997}[Theorem 11.1 and Theorem 11.4]. The results \cite{Chung1997}[Theorem 11.1 and Theorem 11.4] are expressed in terms of the isoperimetric dimesion of a graph, hence the reason these results can be applied to lattices is, in this case, the isoperimetric dimension can be determined. For geometric graphs in general this is a much more difficult problem.  

We note that the path-counting argument above works, since the uniform degree of nodes in the lattice graph allows us to bound the number of paths of a fixed length containing a given edge. Similar arguments could work for other sequences of graphs in which the number of paths of fixed finite length containing a given edge is bounded.
    
\end{enumerate}

\subsection{Optimal discrete Sobolev constants}

In Theorem~\ref{thm:graphsobolev} we prove that a uniform Sobolev inequality~\eqref{eq:graphsobolev} holds under suitable conditions on the concentrating parameter $\varepsilon_n$. However, the constant $C>0$ in \eqref{eq:graphsobolev} derived in the proof of Theorem~\ref{thm:graphsobolev} is far from optimal. This is because we took a large number of overestimations in order to simplify the argument. For example, in restricting ourselves to proving \eqref{eq:graphsobolev} for sufficiently large $n$, we overestimate the Sobolev constant for the finitely many Sobolev inequalities at small $n$ by taking the maximum of the constants that the equivalence of the $p$ and $q$ norms gives us (as explained following \eqref{eq:graphsobolev}). Another overestimate appears when we use the estimate in Remark~\ref{rem:simpletokernel} to go from the case for general $\eta$ to the indicator-function setting.

For a fixed $n\in \mathbb{N}$, $p\in [1,d)$ and $q\in \left[ p, \frac{dp}{d-p} \right]$, computing the optimal Sobolev constant 
$$  C_n:= \sup_{\substack{
u:V_n\rightarrow\mathbb{R} \\
u\neq 0
}} \frac{\Vert u \Vert_{L^q(V_n;\mu_n)}}{\Vert u \Vert_{L^p(V_n;\mu_n)} + (\mathcal{GE}_n^p(u))^{1/p}}$$
is a very difficult problem. For any $n\in\mathbb{N}$, we should expect the constant $C$ that we constructed to be much larger than $C_n$. Nonetheless, under (i*), (ii) and (iii) in Assumptions~\ref{assum} and an additional requirement on the convergence rate of $\varepsilon_n\to0$, we can derive a convergence result for the optimal constants $C_n$ as $n\rightarrow\infty$. That is, we can show they converge to an optimal Sobolev constant in the continuum. This means, for large $n$, one can suitably approximate the problem of finding the optimal Sobolev constant $C_n$ with finding the optimal Sobolev constant in the continuum. A proof of this convergence is beyond the scope of this paper, however the details will appear in~\cite[Section 6.5.1]{methesis}.

\subsection{Future work}

\subsubsection{Open problem regarding Theorem~\ref{thm:graphsobolev}}
\label{sec:open}

For Theorem~\ref{thm:graphsobolev} to hold we have two different requirements on the concentrating parameter sequence $(\varepsilon_n)_{n\in\mathbb{N}}$ depending on whether we assume (iv) from Assumptions~\ref{assum} or not. Under (iv), the condition on $(\varepsilon_n)_{n\in\mathbb{N}}$ becomes sufficient and necessary. The condition we provide in the absence of (iv) is sufficient but we do not currently know if it is necessary.

Consider the following sequence of functions. For each $n\in\mathbb{N}$, select a vertex $x_n \in V_n$ such that $\mu_n(x_n)=\Lambda_n^+$. Define a function $u_n:V_n\rightarrow \mathbb{R}$ such that $u_n(x_n)=1$ and $u_n(x)=0$ for $x\in V_n\setminus\{x_n\}.$ Then we compute
$ \Vert u_n \Vert_{L^p(V_n;\mu_n)}=\left(\Lambda_n^+\right)^{1/p} $
and $\Vert u_n \Vert_{L^q(V_n;\mu_n)} =\left(\Lambda_n^+\right)^{1/q}$. Next we compute
\begin{align*}
\mathcal{GE}_n^p(u_n)&= \frac1{\varepsilon_n^p} \left(\sum_{x\in V_n\setminus\{x_n\}} \sum_{y\in V_n} \eta_{\varepsilon_n}(|x-y|) |u(x)-u(y)|^p \mu_n(x) \mu_n(y) + \Lambda_n^+ \sum_{y\in V_n} \eta_{\varepsilon_n}(|x-y|) |u(x_n)-u(y)| \mu_n(y)\right)\\
&= \frac1{\varepsilon_n^p} \left(\Lambda_n^+ \sum_{x\in V_n\setminus\{x_n\}} \eta_{\varepsilon_n}(|x-x_n|) |u(x_n)-u(x)|^p \mu_n(x) + \Lambda_n^+ \sum_{y\in V_n} \eta_{\varepsilon_n}(|x-y|) |u(x_n)-u(y)| \mu_n(y)\right)\\
&= \frac{2}{\varepsilon_n^p}\left( \Lambda_n^+\sum_{y\in V_n} \eta_{\varepsilon_n}(|y-x_n|)|u_n(y)-u_n(x_n)|^p\mu_n(y) \right) 
\leq \frac{2\alpha_K\aleph^0_\eta \Lambda_n^+}{\varepsilon_n^p},
\end{align*}
where the final bound follows from an application of Lemma~\ref{lem:weightbound}. Together we get
$$  \frac{\Vert u_n \Vert_{L^q(V_n;\mu_n)}}{\Vert u_n \Vert_{L^p(V_n;\mu_n)} + (\mathcal{GE}_n^p(u_n))^{1/p}} \gtrsim_n \varepsilon_n(\Lambda_n^+)^{\frac{1}{q}-\frac{1}{p}}. $$
Therefore, for the uniform Sobolev inequalities in Theorem~\ref{thm:graphsobolev} to hold, the sequence $\left(\varepsilon_n(\Lambda_n^+)^{\frac{1}{q}-\frac{1}{p}}\right)_{n\in \mathbb{N}}$ must be bounded. This is a weaker condition than asking for $\left(\frac{\varepsilon_n (\Lambda_n^+)^{1/q+1/p}}{(\Lambda_n^-)^{2/p}} \right)_{n\in\mathbb{N}}$ to be bounded. As it stands, under this weaker condition we do not know a necessary and sufficient conditon on the concentrating parameter sequence (in terms of $\Lambda_n^+$ and $\Lambda_n^-$) for our Sobolev inequalities to hold.

\subsubsection{H\"older regularity on graphs}
\label{sec:open2}

As we discussed prior to introducing Theorem~\ref{thm:graphsobolev2}, for $p>d$ the classical Morrey inequality~\cite[Theorem 1.4.4.1]{Grisvard1985} in the continuum presents a bound on the H\"older regularity of a $W^{1,p}$-function. Extending this control to the discrete setting presents a number of difficulties. First we need to decide how the H\"older semi-norm on a geometric graph should be defined. 

Let $\gamma\in (0,1]$. One option is to define, for $n\in\mathbb{N}$ and $u:V_n\rightarrow \mathbb{R}$,
$$ | u |_{C^{0,\gamma}(V_n;\mu_n)} := \sup_{x,y\in V_n} \frac{|u(x)-u(y)|}{|x-y|^{\gamma}}. $$
A second option is to set $\mathbb{G}_n:= \mathbb{G}_{n,\varepsilon_n}$, let $\mathbb{d}_n$ be the graph metric associated to $\mathbb{G}_n$ (as we also did in example 4 in Section~\ref{sec:applications}) and define, for $n\in\mathbb{N}$ and $u:V_n\rightarrow \mathbb{R}$,
$$ | u |_{C^{0,\gamma}(\mathbb{G}_n)} := \sup_{x,y\in V_n} \frac{|u(x)-u(y)|}{\varepsilon_n^{\gamma} \left(\mathbb{d}_n(x,y)\right)^{\gamma}}. $$

The first option is a convenient formulation as we simply let $V_n$ inherit the metric structure as a subset of $\Omega$. However the second option may be more applicable since the denominator in the semi-norm depends now entirely on $\varepsilon_n$ and the graph structure of $\mathbb{G}_n$, rather than an extrinsic metric. This mirrors how we defined the discrete variation $\mathcal{GE}_n^p$ in the case that $\eta$ is the indicator function. There is still a question of how to suitably scale $\mathbb{d}_n$, in such a way that it converges to a consistent limit. Multiplying by $\varepsilon_n$ certainly works for our lattice example (example~4) in Section~\ref{sec:applications} where $\varepsilon_n \mathbb{d}_n$ equals a multiple of the $\ell_1$ distance (with $\varepsilon_n$ as defined in that example); for general cases this will require justification and possibly additional assumptions. The lattice example also shows that $\varepsilon_n \mathbb{d}_n(x,y)$ does not necessarily converge to $|x-y|$ as $n\rightarrow \infty$; in this case though, the resulting metric (the $\ell_1$ metric) on $\Omega$ is Lipschitz equivalent to the $\ell_2$ metric on $\Omega$. 

We can now pose a second question. Let $\tilde\gamma\in \left(0,1-\frac{d}{p}\right]$. Under what conditions on the concentrating parameter sequence does there exist an $A>0$ such that, for $n\in\mathbb{N}$ and $u:V_n\rightarrow \mathbb{R}$,
$$ | u |_{C^{0,\tilde\gamma}(\mathbb{G}_n)} \leq A \left( \Vert u \Vert_{L^p(V_n;\mu_n)} +   \left(\mathcal{GE}_n^p(u) \right)^{1/p}\right)?$$

Of course we can also phrase the same question for the semi-norm $| u |_{C^{0,\tilde\gamma}(V_n;\mu_n)}$. Under the prerequisite that (iv) from Assumptions~\ref{assum} holds, we conjecture that it is necessary and sufficient for 
$ \left( \varepsilon_n^{1-\tilde\gamma} |V_n|^{1/p} \right)_{n\in \mathbb{N}}$ to be bounded. Taking the same sequence from Section~\ref{sec:open} we compute $|u_n|_{C^{0,\tilde\gamma}(\mathbb{G}_n)} = \frac{1}{\varepsilon_n^{\tilde\gamma}}$ and 
$$ \frac{|u_n|_{C^{0,\tilde\gamma}(\mathbb{G}_n)}}{\Vert u_n \Vert_{L^p(V_n;\mu_n)} + (\mathcal{GE}_n^p(u_n))^{1/p}} \gtrsim_n \varepsilon_n^{1-\tilde\gamma}|V_n|^{1/p}, $$
so indeed it is necessary.

\paragraph{Acknowledgements}

The authors are grateful for enlightening discussions and valuable advice from Matthew Thorpe, Nicol\'as Garc\'ia Trillos and Leon Bungert. The authors acknowledge receiving funding from the Dutch Research Council (NWO) via its Open Competition Domain Science (ENW) - M.

\appendix

\section{\texorpdfstring{$\Gamma$}{gamma}-convergence counterexample at the connectivity threshold.}\label{app:counterexamples}

For an overview of $\Gamma$-convergence and its properties we refer the reader to \cite{DalMaso93}. In order to simplify the discussion in this appendix we shall assume the limiting probability measure $\mu$ on $\Omega$ is the uniform probability measure.

Within the literature, $\Gamma$-convergence results for discrete-to-continuum limits on geometric graphs typically assume (i$^*$) in place of (i) from Assumptions~\ref{assum}. As a consequence of Corollary~\ref{cor:connect}, this means the concentrating parameter $\varepsilon_n$ is much larger than the connectivity threshold $s^*(V_n)$. Naturally one might ask if a strong assumption such as (i$^*$) is indeed necessary. For example, if we assume (i), (ii), (iii) and (iv) from Assumptions~\ref{assum} all hold and $\varepsilon_n\geq s^*(V_n)$, can we prove, for $p\in [1,\infty)$ and some suitable constant $C>0$,
\begin{equation}
\label{eq:gamma}
\mathcal{GE}_n^p \overset{\Gamma}{\longrightarrow} C\cdot\mathcal{E}^p \; \; \text{as} \; \; n\rightarrow \infty
\end{equation}
over the space $TL^p(\Omega)$?   

The immediate expectation is that~\eqref{eq:gamma} is false for $p\neq 2$. This is because $\mathcal{E}^p$ is isotropic and, if we are to take a regular lattice discretization, then it is typical for a discrete-to-continuum limit over said lattice to be anisotropic. Indeed, if we restrict a smooth function to a lattice, compute $\mathcal{GE}^p_n$ and use a Taylor expansion in the difference terms, we recognise that, at least formally, an anisotropic functional appears in the limit. Moreover, concrete $\Gamma$-convergence results on a lattice have been proven where the limit is seen to be anisotropic, for example \cite[Section 6.5.3]{mefuture} and \cite[Theorem 4.3]{Gennip2012}. Therefore, if we are looking for counterexamples to~\eqref{eq:gamma}, the non-trivial case is $p=2$. In the discussion that follows we shall construct such counterexamples for $p=2$.

\begin{figure}[ht]
    \begin{center}
    \begin{tikzpicture}[scale=1]
    \draw[gray, very thin] (0,0) rectangle (12,12) ;
\foreach \x in {2,4,6,8,10}
\draw[gray,dashed] (\x,0) -- (\x,12) ;
\foreach \y in {3,6,9}
\draw[gray,dashed] (0,\y) -- (12,\y) ;
\foreach \x in {1,3,5,7,9,11}
\foreach \y in {1.5,4.5,7.5,10.5}
{
\filldraw[black] (\x,\y) circle (1.5pt) ;
\draw[thin] (\x,1.5) -- (\x,10.5) ;
\draw[thin] (1,\y) -- (11,\y) ;
}
\end{tikzpicture}
\end{center}
    \caption{Rectangular discretization for $n=2$.}
    \label{fig:rectangle}
\end{figure}
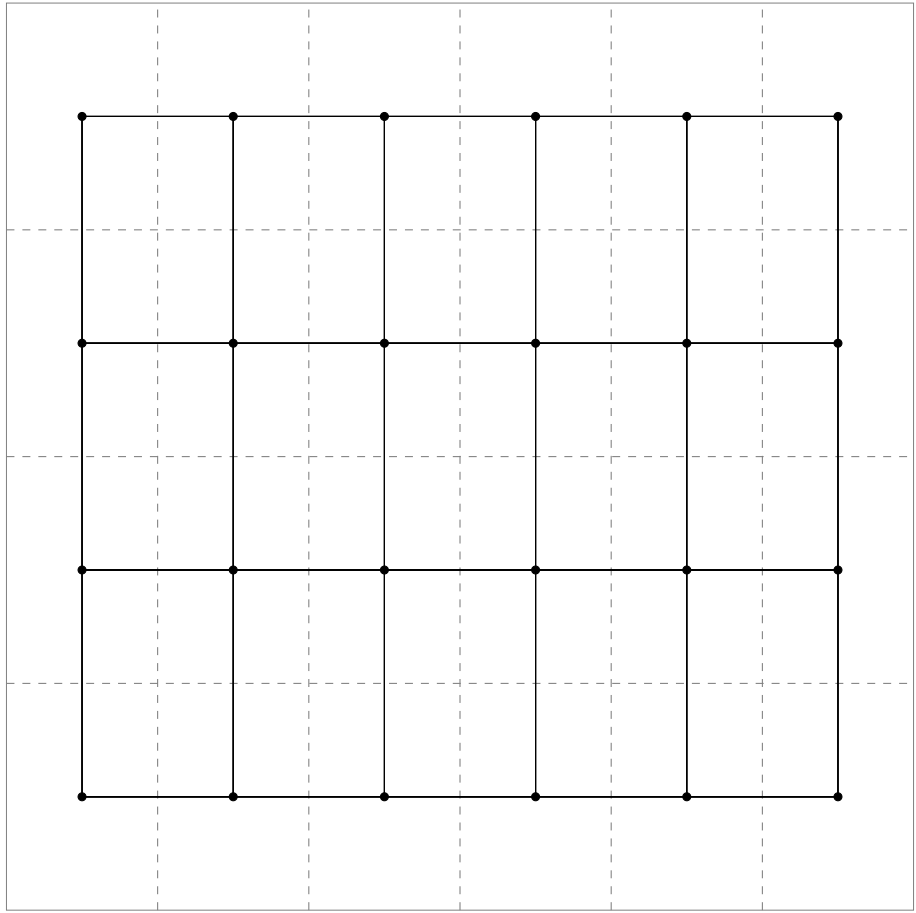

We start by taking $\Omega:=(0,1)^2$ and, for each $n\in\mathbb{N}$, consider a rectangular discretization as illustrated in Figure~\ref{fig:rectangle}. To be more explicit, for each $n\in\mathbb{N}$, we place $6n^2$ vertices at the following positions: for $k\in \{0,\ldots,3n-1\}$ and $l\in\{0,\ldots,2n-1\}$ we place a vertex at $\left( \frac{1+2k}{6n},\frac{\frac32+3l}{6n}\right)$. We define $V_n$ to be this set of vertices and $\mu_n$ to be the uniform discrete probability measure on $V_n$. We shall also define $\eta(r):=\mathbf{1}_{r\leq 1}$. Additionally define $\mu$ to be the uniform probability measure on $(0,1)^2$. So far we can verify that (ii), (iii) and (iv) from Assumptions~\ref{assum} are satisfied. We also define a Borel map $T_n:(0,1)^2\rightarrow (0,1)^2$ such that $T_n\#\mu=\mu_n$ as follows: using the dashed lines as illustrated in Figure~\ref{fig:rectangle}, forming the boxes $\left(\frac{k}{3n},\frac{k+1}{3n}\right)\times\left(\frac{l}{2n},\frac{l+1}{2n}\right)$, we send all of the points inside of a box to the vertex at its centre. Since the edges of the boxes have zero measure, we do not need to define the map there. In this case $\sup_{x\in\Omega}|T_n(x)-x|=\frac{\sqrt{13}}{12n}$. We then set $\varepsilon_n:=s^*(V_n)=\frac{1}{2n}$, which we can see to be the connectivity threshold of $V_n$. With this, we observe that (i) from Assumptions~\ref{assum} is satisfied with $K:=\frac{\sqrt{13}}6$.  

Defining a graph $\mathbb{G}_n:=(V_n,E_n)$, where two vertices are connected if they are strictly within a distance $\frac{1}{2n}$, gives us the graph illustrated by the black lines in Figure~\ref{fig:rectangle}. 

Let $u:\mathbb{R}^2\rightarrow\mathbb{R}$ be a $C^\infty$ function with compact support in $(0,1)^2$. In the calculation that follows we shall use the symbol $u$ interchangeably with $u|_{V_n}$ in a slight abuse of notation. Then, for $n$ large enough (such that there is at least one layer of vertices between the support of $u$ and the boundary of $(0,1)^2$), we have by Taylor's theorem
\begin{align*}
    \mathcal{GE}_n^2(u)&:=  \frac{2}{|V_n|^2\varepsilon_n^{4}}\sum_{k=1}^{3n}\sum_{l=1}^{2n} \left\{ \left| u\left(  \frac{1+2(k-1)}{6n},\frac{\frac{3}{2}+3(l-1)}{6n}\right)-u\left( \frac{1+2(k-1)}{6n},\frac{\frac{3}{2}+3l}{6n} \right) \right|^2 \right. \\
    &\hspace{3.5cm} +  \left| u\left(  \frac{1+2(k-1)}{6n},\frac{\frac{3}{2}+3(l-1)}{6n}\right)-u\left( \frac{1+2(k-1)}{6n},\frac{\frac{3}{2}+3(l-2)}{6n} \right) \right|^2 \\
    &\hspace{3.5cm} +  \left| u\left(  \frac{1+2(k-1)}{6n},\frac{\frac{3}{2}+3(l-1)}{6n}\right)-u\left( \frac{1+2k}{6n},\frac{\frac{3}{2}+3(l-1)}{6n} \right) \right|^2 \\
    &\hspace{3.5cm} + \left.  \left| u\left(  \frac{1+2(k-1)}{6n},\frac{\frac{3}{2}+3(l-1)}{6n}\right)-u\left( \frac{1+2(k-2)}{6n},\frac{\frac{3}{2}+3(l-1)}{6n} \right) \right|^2 \right\} \\
    &= \frac{2^{4}n^2}{3|V_n|} \sum_{k=1}^{3n}\sum_{l=1}^{2n}  \left\{2 \left( \left|\frac{1}{2n} \frac{\partial u}{\partial y} \right|^2+ \left|\frac{1}{3n} \frac{\partial u}{\partial x} \right|^2 \right) \left(  \frac{1+2(k-1)}{6n},\frac{\frac{3}{2}+3(l-1)}{6n}\right) + \mathcal{O}\left(\frac{1}{n^{4}}\right) \right\}  \\
    &\rightarrow \frac{2^{5}}{3}\int_{x=0}^1\int_{y=0}^1  \left\{ \frac{1}{3^2}\left| \frac{\partial u}{\partial x} \right|^2 + \frac{1}{2^2}\left| \frac{\partial u}{\partial y} \right|^2\right\} \; dxdy \; \; \text{as} \; \; n\rightarrow \infty .
\end{align*}

This limit is sufficient to see that the convergence in \eqref{eq:gamma} is false, which we shall now justify. For simplicity in the argument that follows we shall fix $p=2$. Assume for a contradiction there exists $C>0$ such that \eqref{eq:gamma} holds. We will show that the subderivative of $\mathcal{GE}_n^2$ is bounded at the smooth function $u$ used above and consequently the sequence $\left((u|_{V_n},\mu_n)\right))_{n\in\mathbb{N}}$ is a recovery sequence. Hence the anisotropic limit in the computation above will contradict the limit in~\eqref{eq:gamma}.

For each $n\in\mathbb{N}$, $\mathcal{GE}_n^2$ is convex and lower semicontinuous in the topology of $L^2(V_n;\mu_n)$. Its subdifferential at $u:V_n\rightarrow \mathbb{R}$ is a singleton containing the function defined by
$$ \partial\mathcal{GE}_n^2(u)(z):= \frac{64n^2}{3}\left\{ 4u(z)-u\left(z+\left(0,\frac{1}{2n} \right) \right)-u\left(z-\left(0,\frac{1}{2n} \right) \right) -u\left(z+\left(\frac{1}{3n},0 \right) \right)-u\left(z-\left(\frac{1}{3n},0 \right) \right) \right\},$$
for $z\in V_n$. If $z'\in \mathbb{R}^2$ is not in $V_n$ we interpret $u(z')=0$. This is the subdifferential of $\mathcal{GE}_n^2$ over the space $L^2(V_n;\mu_n)$ when equipped with the inner product which is defined for $u,v\in L^2(V_n;\mu_n)$ by
$$ \langle u,v \rangle_n := \frac{1}{6n^2}\sum_{x\in V_n} u(x)v(x).$$

Now again assume $u:\mathbb{R}^2\rightarrow \mathbb{R}$ is smooth with compact support in $(0,1)^2$. We compute via a Taylor expansion, for $z\in V_n$,
\begin{align*}
\partial\mathcal{GE}_n^2(u)(z) = -\frac{64}{27}\frac{\partial^2 u}{\partial x^2}(z) -\frac{16}{3}\frac{\partial^2u}{\partial^2 y}(z)+\mathcal{O}\left(\frac{1}{n} \right).  
\end{align*}
Therefore
\begin{align*}
    \Vert \partial\mathcal{GE}_n^2(u) \Vert_{L^2(V_n;\mu_n)}^2&= \frac{1}{|V_n|}\sum_{z\in V_n} \left| -\frac{64}{27}\frac{\partial^2 u}{\partial x^2}(z) -\frac{16}{3}\frac{\partial^2u}{\partial^2 y}(z)+\mathcal{O}\left(\frac{1}{n}\right) \right|^2 \\
    &\rightarrow \int_{x=0}^1\int_{y=0}^1 \left| \frac{64}{27}\frac{\partial^2 u}{\partial x^2}+\frac{16}{3}\frac{\partial^2u}{\partial^2 y} \right|^2 \; dxdy \; \; \text{as} \; \; n\rightarrow \infty.
\end{align*}
What is important for us to observe is that $\left( \Vert \partial\mathcal{GE}_n^2(u) \Vert_{L^2(V_n;\mu_n)}\right)_{n\in\mathbb{N}}$ is bounded. Next, using our assumption of $\Gamma$-convergence, we shall choose a recovery sequence of functions $u_n:V_n\rightarrow\mathbb{R}$, with $n\in\mathbb{N}$, such that $\left((u_n,\mu_n)\right)_{n\in\mathbb{N}}$ converges to $(u,\mu)$ in $TL^2(\Omega)$ and $\mathcal{GE}_n^2(u_n)\rightarrow C\cdot \mathcal{E}^2(u)$ as $n\rightarrow \infty$. We claim that $\| u_n -u\|_{L^2(V_n;\mu_n)}\rightarrow 0$ as $n\rightarrow \infty$. To see why this convergence is true we first observe that $\left((u|_{V_n},\mu_n)\right)_{n\in\mathbb{N}}$ converges to $(u,\mu)$ in $TL^2(\Omega)$ as $n\rightarrow\infty$. Indeed,
$$ \int_{\Omega} |u\circ T_n(z) -u(z)|^2 \; dz \leq \| \nabla u \|_{L^\infty(\Omega)}^2 \cdot \sup_{x\in\Omega} |T_n(x)-x|^2\rightarrow 0 \; \; \text{as} \; \; n\rightarrow\infty . $$
By~\eqref{eq:TLpconvcond} and part~(i) of Assumptions~\ref{assum} (which we have shown to be satisfied above), we deduce $\left((u_n-u|_{V_n},\mu_n)\right)_{n\in\mathbb{N}}$ converges to $(0,\mu)$ in $TL^2(\Omega)$. Using the properties of $TL^2(\Omega)$ proved in~\cite[Proposition 5.17]{mercer2025} it follows that $\| u_n -u\|_{L^2(V_n;\mu_n)}\rightarrow 0$ as $n\rightarrow \infty$.

Using the definition of the subdifferential and convexity,
\begin{align*}
    \mathcal{GE}_n^2(u_n)&\geq  \mathcal{GE}_n^2(u) + \langle \partial\mathcal{GE}_n^2(u) , u_n-u\rangle_n \\
    &\geq \mathcal{GE}_n^2(u) + \| \partial\mathcal{GE}_n^2(u) \|_{L^2(V_n;\mu_n)} \cdot \|u_n -u \|_{L^2(V_n;\mu_n)}.  
\end{align*}
The second line follows by the Cauchy--Schwarz inequality. Now, passing to the limit as $n\rightarrow \infty$ in the above inequality yields
$$ C\cdot \int_{x=0}^1\int_{y=0}^1 \left\{\left| \frac{\partial u}{\partial x} \right|^2 +\left| \frac{\partial u}{\partial y} \right|^2 \right\} \; dxdy \geq   \int_{x=0}^1\int_{y=0}^1  \left\{ \frac{32}{27}\left| \frac{\partial u}{\partial x} \right|^2 + \frac{8}{3}\left| \frac{\partial u}{\partial y} \right|^2\right\} \; dxdy.$$
However, by the $\liminf$ inequality in the definition of $\Gamma$-convergence, we also have $\liminf_{n\rightarrow\infty} \mathcal{GE}_n^2(u)\geq C\cdot \mathcal{E}^2(u)$. Putting this together with the previous bound gives, for any $u:\mathbb{R}^2\rightarrow \mathbb{R}$ smooth with compact support in $(0,1)^2$,
$$ C\cdot \int_{x=0}^1\int_{y=0}^1 \left\{\left| \frac{\partial u}{\partial x} \right|^2 +\left| \frac{\partial u}{\partial y} \right|^2 \right\} \; dxdy =   \int_{x=0}^1\int_{y=0}^1  \left\{ \frac{32}{27}\left| \frac{\partial u}{\partial x} \right|^2 + \frac{8}{3}\left| \frac{\partial u}{\partial y} \right|^2\right\} \; dxdy.$$
Choosing $u$ suitably then yields a contradiction. 

We identify the problem that leads to the continuum limit in our example being anisotropic to be the choice of our kernel $\eta(r):=\mathbf{1}_{r\leq 1}$. This results in the edges, as illustrated in Figure~\ref{fig:rectangle}, being assigned the same weight---in the sense that each difference $|u(x)-u(y)|^p$ for the incident nodes $x$ and $y$ has the same coefficient $\eta_{\varepsilon_n}(|x-y|)=1$ in $\mathcal{GE}_n^2$---even though the vertical edges are longer than the horizontal edges. If we chose weights suitably, i.e. ones that are inversely proportional to the lengths of the edges (via a non-radially-symmetric kernel $\eta$, which goes beyond our assumptions on $\eta$ in Section~\ref{sec:geometry}), then the continuum limit will be isotropic. Below we provide two counterexamples for which it is impossible to observe a isotropic continuum limit, regardless of how we choose to weight the edges. 

\begin{figure}[ht]
\centering
    \begin{subcaptionblock}{.5\textwidth}
    \centering
    \begin{tikzpicture}[scale=0.8]
%vertices
\filldraw[black] (1,0) circle (1.5pt) ;
\filldraw[black] (3,0) circle (1.5pt) ;
\filldraw[black] (5,0) circle (1.5pt) ;
\filldraw[black] (7,0) circle (1.5pt) ;
\filldraw[black] (0,1) circle (1.5pt) ;
\filldraw[black] (2,1) circle (1.5pt) ;
\filldraw[black] (4,1) circle (1.5pt) ;
\filldraw[black] (6,1) circle (1.5pt) ;
\filldraw[black] (8,1) circle (1.5pt) ;
\filldraw[black] (0,2) circle (1.5pt) ;
\filldraw[black] (2,2) circle (1.5pt) ;
\filldraw[black] (4,2) circle (1.5pt) ;
\filldraw[black] (6,2) circle (1.5pt) ;
\filldraw[black] (8,2) circle (1.5pt) ;
\filldraw[black] (1,3) circle (1.5pt) ;
\filldraw[black] (3,3) circle (1.5pt) ;
\filldraw[black] (5,3) circle (1.5pt) ;
\filldraw[black] (7,3) circle (1.5pt) ;
\filldraw[black] (0,4) circle (1.5pt) ;
\filldraw[black] (2,4) circle (1.5pt) ;
\filldraw[black] (4,4) circle (1.5pt) ;
\filldraw[black] (6,4) circle (1.5pt) ;
\filldraw[black] (8,4) circle (1.5pt) ;
\filldraw[black] (0,5) circle (1.5pt) ;
\filldraw[black] (2,5) circle (1.5pt) ;
\filldraw[black] (4,5) circle (1.5pt) ;
\filldraw[black] (6,5) circle (1.5pt) ;
\filldraw[black] (8,5) circle (1.5pt) ;
\filldraw[black] (1,6) circle (1.5pt) ;
\filldraw[black] (3,6) circle (1.5pt) ;
\filldraw[black] (5,6) circle (1.5pt) ;
\filldraw[black] (7,6) circle (1.5pt) ;
\filldraw[black] (0,7) circle (1.5pt) ;
\filldraw[black] (2,7) circle (1.5pt) ;
\filldraw[black] (4,7) circle (1.5pt) ;
\filldraw[black] (6,7) circle (1.5pt) ;
\filldraw[black] (8,7) circle (1.5pt) ;
\filldraw[black] (0,8) circle (1.5pt) ;
\filldraw[black] (2,8) circle (1.5pt) ;
\filldraw[black] (4,8) circle (1.5pt) ;
\filldraw[black] (6,8) circle (1.5pt) ;
\filldraw[black] (8,8) circle (1.5pt) ;
%edges
\draw[thin] (1,0) -- (0,1) ;
\draw[thin] (1,0) -- (2,1) ;
\draw[thin] (3,0) -- (2,1) ;
\draw[thin] (3,0) -- (4,1) ;
\draw[thin] (5,0) -- (4,1) ;
\draw[thin] (5,0) -- (6,1) ;
\draw[thin] (7,0) -- (6,1) ;
\draw[thin] (7,0) -- (8,1) ;

\draw[thin] (0,1) -- (0,2) ;
\draw[thin] (2,1) -- (2,2) ;
\draw[thin] (4,1) -- (4,2) ;
\draw[thin] (6,1) -- (6,2) ;
\draw[thin] (8,1) -- (8,2) ;

\draw[thin] (1,3) -- (0,2) ;
\draw[thin] (1,3) -- (2,2) ;
\draw[thin] (3,3) -- (2,2) ;
\draw[thin] (3,3) -- (4,2) ;
\draw[thin] (5,3) -- (4,2) ;
\draw[thin] (5,3) -- (6,2) ;
\draw[thin] (7,3) -- (6,2) ;
\draw[thin] (7,3) -- (8,2) ;

\draw[thin] (1,3) -- (0,4) ;
\draw[thin] (1,3) -- (2,4) ;
\draw[thin] (3,3) -- (2,4) ;
\draw[thin] (3,3) -- (4,4) ;
\draw[thin] (5,3) -- (4,4) ;
\draw[thin] (5,3) -- (6,4) ;
\draw[thin] (7,3) -- (6,4) ;
\draw[thin] (7,3) -- (8,4) ;

\draw[thin] (0,4) -- (0,5) ;
\draw[thin] (2,4) -- (2,5) ;
\draw[thin] (4,4) -- (4,5) ;
\draw[thin] (6,4) -- (6,5) ;
\draw[thin] (8,4) -- (8,5) ;

\draw[thin] (1,6) -- (0,5) ;
\draw[thin] (1,6) -- (2,5) ;
\draw[thin] (3,6) -- (2,5) ;
\draw[thin] (3,6) -- (4,5) ;
\draw[thin] (5,6) -- (4,5) ;
\draw[thin] (5,6) -- (6,5) ;
\draw[thin] (7,6) -- (6,5) ;
\draw[thin] (7,6) -- (8,5) ;

\draw[thin] (1,6) -- (0,7) ;
\draw[thin] (1,6) -- (2,7) ;
\draw[thin] (3,6) -- (2,7) ;
\draw[thin] (3,6) -- (4,7) ;
\draw[thin] (5,6) -- (4,7) ;
\draw[thin] (5,6) -- (6,7) ;
\draw[thin] (7,6) -- (6,7) ;
\draw[thin] (7,6) -- (8,7) ;

\draw[thin] (0,7) -- (0,8) ;
\draw[thin] (2,7) -- (2,8) ;
\draw[thin] (4,7) -- (4,8) ;
\draw[thin] (6,7) -- (6,8) ;
\draw[thin] (8,7) -- (8,8) ;
\draw[gray, very thin] (0,0) rectangle (8,8) ;
\end{tikzpicture}
\caption{Hexagon discretization for $n=8$}
\end{subcaptionblock}%
\begin{subcaptionblock}{.5\textwidth}
\centering
\begin{tikzpicture}[scale=0.8]

%vertices
\filldraw[black] (1,0) circle (1.5pt) ;
\filldraw[black] (3,0) circle (1.5pt) ;
\filldraw[black] (5,0) circle (1.5pt) ;
\filldraw[black] (7,0) circle (1.5pt) ;
\filldraw[black] (0,1) circle (1.5pt) ;
\filldraw[black] (2,1) circle (1.5pt) ;
\filldraw[black] (4,1) circle (1.5pt) ;
\filldraw[black] (6,1) circle (1.5pt) ;
\filldraw[black] (8,1) circle (1.5pt) ;
\filldraw[black] (0,2) circle (1.5pt) ;
\filldraw[black] (2,2) circle (1.5pt) ;
\filldraw[black] (4,2) circle (1.5pt) ;
\filldraw[black] (6,2) circle (1.5pt) ;
\filldraw[black] (8,2) circle (1.5pt) ;
\filldraw[black] (1,3) circle (1.5pt) ;
\filldraw[black] (3,3) circle (1.5pt) ;
\filldraw[black] (5,3) circle (1.5pt) ;
\filldraw[black] (7,3) circle (1.5pt) ;
\filldraw[black] (1,4) circle (1.5pt) ;
\filldraw[black] (3,4) circle (1.5pt) ;
\filldraw[black] (5,4) circle (1.5pt) ;
\filldraw[black] (7,4) circle (1.5pt) ;
\filldraw[black] (0,5) circle (1.5pt) ;
\filldraw[black] (2,5) circle (1.5pt) ;
\filldraw[black] (4,5) circle (1.5pt) ;
\filldraw[black] (6,5) circle (1.5pt) ;
\filldraw[black] (8,5) circle (1.5pt) ;
\filldraw[black] (0,6) circle (1.5pt) ;
\filldraw[black] (2,6) circle (1.5pt) ;
\filldraw[black] (4,6) circle (1.5pt) ;
\filldraw[black] (6,6) circle (1.5pt) ;
\filldraw[black] (8,6) circle (1.5pt) ;
\filldraw[black] (1,7) circle (1.5pt) ;
\filldraw[black] (3,7) circle (1.5pt) ;
\filldraw[black] (5,7) circle (1.5pt) ;
\filldraw[black] (7,7) circle (1.5pt) ;
\filldraw[black] (1,8) circle (1.5pt) ;
\filldraw[black] (3,8) circle (1.5pt) ;
\filldraw[black] (5,8) circle (1.5pt) ;
\filldraw[black] (7,8) circle (1.5pt) ;
%edges
\draw[thin] (1,0) -- (0,1) ;
\draw[thin] (1,0) -- (2,1) ;
\draw[thin] (3,0) -- (2,1) ;
\draw[thin] (3,0) -- (4,1) ;
\draw[thin] (5,0) -- (4,1) ;
\draw[thin] (5,0) -- (6,1) ;
\draw[thin] (7,0) -- (6,1) ;
\draw[thin] (7,0) -- (8,1) ;

\draw[thin] (0,1) -- (0,2) ;
\draw[thin] (2,1) -- (2,2) ;
\draw[thin] (4,1) -- (4,2) ;
\draw[thin] (6,1) -- (6,2) ;
\draw[thin] (8,1) -- (8,2) ;

\draw[thin] (1,3) -- (0,2) ;
\draw[thin] (1,3) -- (2,2) ;
\draw[thin] (3,3) -- (2,2) ;
\draw[thin] (3,3) -- (4,2) ;
\draw[thin] (5,3) -- (4,2) ;
\draw[thin] (5,3) -- (6,2) ;
\draw[thin] (7,3) -- (6,2) ;
\draw[thin] (7,3) -- (8,2) ;

\draw[thin] (1,3) -- (1,4) ;
\draw[thin] (3,3) -- (3,4) ;
\draw[thin] (5,3) -- (5,4) ;
\draw[thin] (7,3) -- (7,4) ;

\draw[thin] (1,4) -- (0,5) ;
\draw[thin] (1,4) -- (2,5) ;
\draw[thin] (3,4) -- (2,5) ;
\draw[thin] (3,4) -- (4,5) ;
\draw[thin] (5,4) -- (4,5) ;
\draw[thin] (5,4) -- (6,5) ;
\draw[thin] (7,4) -- (6,5) ;
\draw[thin] (7,4) -- (8,5) ;

\draw[thin] (0,5) -- (0,6) ;
\draw[thin] (2,5) -- (2,6) ;
\draw[thin] (4,5) -- (4,6) ;
\draw[thin] (6,5) -- (6,6) ;
\draw[thin] (8,5) -- (8,6) ;

\draw[thin] (1,7) -- (0,6) ;
\draw[thin] (1,7) -- (2,6) ;
\draw[thin] (3,7) -- (2,6) ;
\draw[thin] (3,7) -- (4,6) ;
\draw[thin] (5,7) -- (4,6) ;
\draw[thin] (5,7) -- (6,6) ;
\draw[thin] (7,7) -- (6,6) ;
\draw[thin] (7,7) -- (8,6) ;

\draw[thin] (1,7) -- (1,8) ;
\draw[thin] (3,7) -- (3,8) ;
\draw[thin] (5,7) -- (5,8) ;
\draw[thin] (7,7) -- (7,8) ;

\draw[gray, very thin] (0,0) rectangle (8,8) ;
\end{tikzpicture}
\caption{Snakeskin discretization for $n=8$}
\end{subcaptionblock}%   
\caption{Discretizations that have an anisotropic continuum limit.}
    \label{fig:bestagons}
\end{figure}

In the examples provided in Figure~\ref{fig:bestagons}, the vertices are placed in a spatially uniform way. If we assign a uniform discrete probability measure to these vertices, then the sequence of probability measures will converge to the Lebesgue measure on $[0,1]^d$. In both of these examples the edges with a horizontal component have a vertical component of equal magnitude. Alongside this, there is a collection of vertical edges with no horizontal component. In order for the continuum limit to be isotropic, the vertical edges must be weighted zero. Weighting the vertical edges zero disconnects the graph, making it impossible to observe a continuum limit. 

We shall not provide calculations to justify these claims due to their length. We hope that the explanation provided is sufficient to give the reader some intuition as to how the combinatorics of a discretization may create (potentially) adverse affects, such as an anisotropic limit (which one may interpret as bias). As previously mentioned in Section~\ref{sec:connectivity}, this also highlights how deterministic discretizations can significantly differ from stochastically sampled discretizations in their limiting behaviour. There a result is discussed, namely \cite[Theorem 3.3]{Braides2023}, that shows that, within the Poisson-point-cloud setting, $\Gamma$-convergence to an isotropic limit is observed up to the percolation threshold, which, as our counterexamples justify, is not the case within the deterministic setting.

\bibliographystyle{plain} 
\bibliography{mybib}

@article{garcia2016continuum,
  title={Continuum limit of total variation on point clouds},
  author={Garc{\'\i}a Trillos, Nicol{\'a}s and Slep{\v{c}}ev, Dejan},
  journal={Archive for Rational Mechanics and Analysis},
  volume={220},
  pages={193--241},
  year={2016},
  publisher={Springer}
}

@book{DalMaso93,
    AUTHOR = {Dal Maso, Gianni},
     TITLE = {An introduction to {$\Gamma$}-convergence},
    SERIES = {Progress in Nonlinear Differential Equations and their
              Applications},
    VOLUME = {8},
 PUBLISHER = {Birkh\"{a}user Boston, Inc., Boston, MA},
      YEAR = {1993},
     PAGES = {xiv+340},
      ISBN = {0-8176-3679-X},
   MRCLASS = {49-02 (46N10 47H99 47N10 49J45 73B27)},
  MRNUMBER = {1201152},
MRREVIEWER = {T.\ Zolezzi},
       DOI = {10.1007/978-1-4612-0327-8},
       URL = {https://doi.org/10.1007/978-1-4612-0327-8},
}

@article {Nicolas2015,
    AUTHOR = {Garc\'ia Trillos, Nicol\'as and Slep{\v{c}}ev, Dejan},
     TITLE = {On the rate of convergence of empirical measures in
              {$\infty$}-transportation distance},
   JOURNAL = {Canad. J. Math.},
  FJOURNAL = {Canadian Journal of Mathematics. Journal Canadien de
              Math\'ematiques},
    VOLUME = {67},
      YEAR = {2015},
    NUMBER = {6},
     PAGES = {1358--1383},
      ISSN = {0008-414X,1496-4279},
   MRCLASS = {60B10 (05C70 05D40 60D05)},
  MRNUMBER = {3415656},
MRREVIEWER = {Hac\`ene\ Djellout},
       DOI = {10.4153/CJM-2014-044-6},
       URL = {https://doi.org/10.4153/CJM-2014-044-6},
}

@article {Dejan2019,
    AUTHOR = {Slep{\v{c}}ev, Dejan and Thorpe, Matthew},
     TITLE = {Analysis of {$p$}-{L}aplacian regularization in semisupervised
              learning},
   JOURNAL = {SIAM J. Math. Anal.},
  FJOURNAL = {SIAM Journal on Mathematical Analysis},
    VOLUME = {51},
      YEAR = {2019},
    NUMBER = {3},
     PAGES = {2085--2120},
      ISSN = {0036-1410,1095-7154},
   MRCLASS = {49J55 (35J20 35J92 49J45 62G20)},
  MRNUMBER = {3953458},
MRREVIEWER = {Mathias\ Hudoba de Badyn},
       DOI = {10.1137/17M115222X},
       URL = {https://doi.org/10.1137/17M115222X},
}

@article {Garcia2018,
    AUTHOR = {Garc\'ia Trillos, Nicol\'as and Slep{\v{c}}ev, Dejan},
     TITLE = {A variational approach to the consistency of spectral
              clustering},
   JOURNAL = {Appl. Comput. Harmon. Anal.},
  FJOURNAL = {Applied and Computational Harmonic Analysis. Time-Frequency
              and Time-Scale Analysis, Wavelets, Numerical Algorithms, and
              Applications},
    VOLUME = {45},
      YEAR = {2018},
    NUMBER = {2},
     PAGES = {239--281},
      ISSN = {1063-5203,1096-603X},
   MRCLASS = {49J55 (49J45 60D05 62G20 68R10)},
  MRNUMBER = {3824112},
       DOI = {10.1016/j.acha.2016.09.003},
       URL = {https://doi.org/10.1016/j.acha.2016.09.003},
}

@article {Thorpe2017,
    AUTHOR = {Thorpe, Matthew and Park, Serim and Kolouri, Soheil and Rohde,
              Gustavo K. and Slep{\v{c}}ev, Dejan},
     TITLE = {A transportation {$L^p$} distance for signal analysis},
   JOURNAL = {J. Math. Imaging Vision},
  FJOURNAL = {Journal of Mathematical Imaging and Vision},
    VOLUME = {59},
      YEAR = {2017},
    NUMBER = {2},
     PAGES = {187--210},
      ISSN = {0924-9907,1573-7683},
   MRCLASS = {94A08 (68T10 68T45 68U10)},
  MRNUMBER = {3694804},
MRREVIEWER = {A.\ F.\ Gualtierotti},
       DOI = {10.1007/s10851-017-0726-4},
       URL = {https://doi.org/10.1007/s10851-017-0726-4},
}

@article {slepcev2020,
    AUTHOR = {Caroccia, Marco and Chambolle, Antonin and Slep{\v{c}}ev, Dejan},
     TITLE = {Mumford--{S}hah functionals on graphs and their asymptotics},
   JOURNAL = {Nonlinearity},
  FJOURNAL = {Nonlinearity},
    VOLUME = {33},
      YEAR = {2020},
    NUMBER = {8},
     PAGES = {3846--3888},
      ISSN = {0951-7715,1361-6544},
   MRCLASS = {49J55 (05C90 49J45 49Q22 62G20 65N12)},
  MRNUMBER = {4115077},
MRREVIEWER = {Koji\ Kikuchi},
       DOI = {10.1088/1361-6544/ab81ee},
       URL = {https://doi.org/10.1088/1361-6544/ab81ee},
}

@article {Thorpe2020,
    AUTHOR = {Cristoferi, Riccardo and Thorpe, Matthew},
     TITLE = {Large data limit for a phase transition model with the
              {$p$}-{L}aplacian on point clouds},
   JOURNAL = {European J. Appl. Math.},
  FJOURNAL = {European Journal of Applied Mathematics},
    VOLUME = {31},
      YEAR = {2020},
    NUMBER = {2},
     PAGES = {185--231},
      ISSN = {0956-7925,1469-4425},
   MRCLASS = {62R07 (49J45 49J55 62G20 62H30)},
  MRNUMBER = {4069821},
       DOI = {10.1017/s0956792518000645},
       URL = {https://doi.org/10.1017/s0956792518000645},
}

@article {Davis2018,
    AUTHOR = {Davis, Erik and Sethuraman, Sunder},
     TITLE = {Consistency of modularity clustering on random geometric
              graphs},
   JOURNAL = {Ann. Appl. Probab.},
  FJOURNAL = {The Annals of Applied Probability},
    VOLUME = {28},
      YEAR = {2018},
    NUMBER = {4},
     PAGES = {2003--2062},
      ISSN = {1050-5164,2168-8737},
   MRCLASS = {60D05 (05C80 49J45 49J55 62G20 68R10)},
  MRNUMBER = {3843822},
       DOI = {10.1214/17-AAP1313},
       URL = {https://doi.org/10.1214/17-AAP1313},
}

@article {Dunlop2020,
    AUTHOR = {Dunlop, Matthew M. and Slep{\v{c}}ev, Dejan and Stuart, Andrew M.
              and Thorpe, Matthew},
     TITLE = {Large data and zero noise limits of graph-based
              semi-supervised learning algorithms},
   JOURNAL = {Appl. Comput. Harmon. Anal.},
  FJOURNAL = {Applied and Computational Harmonic Analysis. Time-Frequency
              and Time-Scale Analysis, Wavelets, Numerical Algorithms, and
              Applications},
    VOLUME = {49},
      YEAR = {2020},
    NUMBER = {2},
     PAGES = {655--697},
      ISSN = {1063-5203,1096-603X},
   MRCLASS = {62G20 (49J45 62C10 62F15 90C20 90C48)},
  MRNUMBER = {4117856},
       DOI = {10.1016/j.acha.2019.03.005},
       URL = {https://doi.org/10.1016/j.acha.2019.03.005},
}

@article {Garcia2020,
    AUTHOR = {Garc\'ia Trillos, Nicol\'as and Gerlach, Moritz and Hein,
              Matthias and Slep{\v{c}}ev, Dejan},
     TITLE = {Error estimates for spectral convergence of the graph
              {L}aplacian on random geometric graphs toward the
              {L}aplace-{B}eltrami operator},
   JOURNAL = {Found. Comput. Math.},
  FJOURNAL = {Foundations of Computational Mathematics. The Journal of the
              Society for the Foundations of Computational Mathematics},
    VOLUME = {20},
      YEAR = {2020},
    NUMBER = {4},
     PAGES = {827--887},
      ISSN = {1615-3375,1615-3383},
   MRCLASS = {62G20 (05C50 58J50 60D05 62R30 65N25 68R10 81Q35)},
  MRNUMBER = {4130541},
       DOI = {10.1007/s10208-019-09436-w},
       URL = {https://doi.org/10.1007/s10208-019-09436-w},
}

@article {Kaplan2020,
    AUTHOR = {Garc\'ia Trillos, Nicol\'as and Kaplan, Zachary and
              Samakhoana, Thabo and Sanz-Alonso, Daniel},
     TITLE = {On the consistency of graph-based {B}ayesian semi-supervised
              learning and the scalability of sampling algorithms},
   JOURNAL = {J. Mach. Learn. Res.},
  FJOURNAL = {Journal of Machine Learning Research (JMLR)},
    VOLUME = {21},
      YEAR = {2020},
     PAGES = {Paper No. 28, 47},
      ISSN = {1532-4435,1533-7928},
   MRCLASS = {62H30 (62F15)},
  MRNUMBER = {4073761},
}

@article {Alonso2018,
    AUTHOR = {Garc\'ia Trillos, Nicol\'as and Sanz-Alonso, Daniel},
     TITLE = {Continuum limits of posteriors in graph {B}ayesian inverse
              problems},
   JOURNAL = {SIAM J. Math. Anal.},
  FJOURNAL = {SIAM Journal on Mathematical Analysis},
    VOLUME = {50},
      YEAR = {2018},
    NUMBER = {4},
     PAGES = {4020--4040},
      ISSN = {0036-1410,1095-7154},
   MRCLASS = {28A33 (46N30 62F15)},
  MRNUMBER = {3829512},
MRREVIEWER = {Junxiong\ Jia},
       DOI = {10.1137/17M1138005},
       URL = {https://doi.org/10.1137/17M1138005},
}

@article {Garcia2017,
    AUTHOR = {Garc\'ia Trillos, Nicol\'as and Slep{\v{c}}ev, Dejan and von
              Brecht, James},
     TITLE = {Estimating perimeter using graph cuts},
   JOURNAL = {Adv. in Appl. Probab.},
  FJOURNAL = {Advances in Applied Probability},
    VOLUME = {49},
      YEAR = {2017},
    NUMBER = {4},
     PAGES = {1067--1090},
      ISSN = {0001-8678,1475-6064},
   MRCLASS = {60D05 (62G20 68R10)},
  MRNUMBER = {3732187},
       DOI = {10.1017/apr.2017.34},
       URL = {https://doi.org/10.1017/apr.2017.34},
}

@article {von2016,
    AUTHOR = {Garc\'ia Trillos, Nicol\'as and Slep{\v{c}}ev, Dejan and von
              Brecht, James and Laurent, Thomas and Bresson, Xavier},
     TITLE = {Consistency of {C}heeger and ratio graph cuts},
   JOURNAL = {J. Mach. Learn. Res.},
  FJOURNAL = {Journal of Machine Learning Research (JMLR)},
    VOLUME = {17},
      YEAR = {2016},
     PAGES = {Paper No. 181, 46},
      ISSN = {1532-4435,1533-7928},
   MRCLASS = {62H30 (05C85 68R10)},
  MRNUMBER = {3567449},
}

@article {Osting2017,
    AUTHOR = {Osting, Braxton and Reeb, Todd Harry},
     TITLE = {Consistency of {D}irichlet partitions},
   JOURNAL = {SIAM J. Math. Anal.},
  FJOURNAL = {SIAM Journal on Mathematical Analysis},
    VOLUME = {49},
      YEAR = {2017},
    NUMBER = {5},
     PAGES = {4251--4274},
      ISSN = {0036-1410,1095-7154},
   MRCLASS = {62H30 (49J55 60D05 62G20 68R10)},
  MRNUMBER = {3717816},
       DOI = {10.1137/16M1098309},
       URL = {https://doi.org/10.1137/16M1098309},
}

@article {Theil2019,
    AUTHOR = {Thorpe, Matthew and Theil, Florian},
     TITLE = {Asymptotic analysis of the {G}inzburg--{L}andau functional on
              point clouds},
   JOURNAL = {Proc. Roy. Soc. Edinburgh Sect. A},
  FJOURNAL = {Proceedings of the Royal Society of Edinburgh. Section A.
              Mathematics},
    VOLUME = {149},
      YEAR = {2019},
    NUMBER = {2},
     PAGES = {387--427},
      ISSN = {0308-2105,1473-7124},
   MRCLASS = {49J45 (60B10 62G20 62H30)},
  MRNUMBER = {3937712},
MRREVIEWER = {Mikhail\ I.\ Sumin},
       DOI = {10.1017/prm.2018.32},
       URL = {https://doi.org/10.1017/prm.2018.32},
}

@article {ThorpeVanGennip23,
    AUTHOR = {Thorpe, Matthew and van Gennip, Yves},
     TITLE = {Deep limits of residual neural networks},
   JOURNAL = {Res. Math. Sci.},
  FJOURNAL = {Research in the Mathematical Sciences},
    VOLUME = {10},
      YEAR = {2023},
    NUMBER = {1},
     PAGES = {Paper No. 6, 44},
      ISSN = {2522-0144,2197-9847},
   MRCLASS = {49J45 (34E05 39A30 49J15 49N60)},
  MRNUMBER = {4522828},
       DOI = {10.1007/s40687-022-00370-y},
       URL = {https://doi.org/10.1007/s40687-022-00370-y},
}

@book{AdamsFournier03,
  title={Sobolev Spaces},
  author={Adams, Robert A. and Fournier, John J. F.},
  volume={140},
  year={2003},
  Series = {Pure and applied mathematics},
  address = {Oxford, Amsterdam},
  publisher={Elsevier}
}

@book{Villani09,
    author = {Villani, C\'edric},
    title = {Optimal Transport},
    subtitle = {Old and New},
    doi = {https://doi.org/10.1007/978-3-540-71050-9},
    year = {2009},
    edition = {first},
    publisher = {Springer},
    address = {Berlin, Heidelberg},
    series = {Grundlehren der mathematischen Wissenschaften},
    volume = {338},
    pages = {xxii+976}
}

@article{mefuture,
    author = {Mercer, Samuel and van Gennip, Yves},
    title = {Uniform graph Sobolev inequalities and graph gradient flows: convergence from discrete to continuum},
    year = {in prep.},
}

@article {Champion2008,
    AUTHOR = {Champion, Thierry and De Pascale, Luigi and Juutinen, Petri},
     TITLE = {The {$\infty$}-{W}asserstein distance: local solutions and
              existence of optimal transport maps},
   JOURNAL = {SIAM J. Math. Anal.},
  FJOURNAL = {SIAM Journal on Mathematical Analysis},
    VOLUME = {40},
      YEAR = {2008},
    NUMBER = {1},
     PAGES = {1--20},
      ISSN = {0036-1410,1095-7154},
   MRCLASS = {49Q20 (49K30)},
  MRNUMBER = {2403310},
       DOI = {10.1137/07069938X},
       URL = {https://doi-org.tudelft.idm.oclc.org/10.1137/07069938X},
}

@misc{mercer2025,
      title={An extension to {B}anach stackings of the {B}rezis--{P}azy semigroup-convergence theorem, with applications to $\lambda$-convex gradient flows}, 
      author={Samuel Mercer and Yves van Gennip},
      year={2025},
      eprint={2511.23233},
      archivePrefix={arXiv},
      primaryClass={math.AP},
      url={https://arxiv.org/abs/2511.23233}, 
      publisher = {arXiv},
    copyright = {arXiv.org perpetual, non-exclusive license},
    howpublished = {arxiv.org/abs/2511.23233v1 [math.AP]}
}

@book {Chung1997,
    AUTHOR = {Chung, Fan R. K.},
     TITLE = {Spectral graph theory},
    SERIES = {CBMS Regional Conference Series in Mathematics},
    VOLUME = {92},
 PUBLISHER = {Conference Board of the Mathematical Sciences, Washington, DC;
              by the American Mathematical Society, Providence, RI},
      YEAR = {1997},
     PAGES = {xii+207},
      ISBN = {0-8218-0315-8},
   MRCLASS = {58G99 (05C50 35P05 46N20 47N20)},
  MRNUMBER = {1421568},
MRREVIEWER = {Robert\ Brooks},
}

@book {Leoni2009,
    AUTHOR = {Leoni, Giovanni},
     TITLE = {A first course in {S}obolev spaces},
    SERIES = {Graduate Studies in Mathematics},
    VOLUME = {105},
 PUBLISHER = {American Mathematical Society, Providence, RI},
      YEAR = {2009},
     PAGES = {xvi+607},
      ISBN = {978-0-8218-4768-8},
   MRCLASS = {46E35 (26Axx 26B30 28A78 46-01)},
  MRNUMBER = {2527916},
MRREVIEWER = {Dorothee\ D.\ Haroske},
       DOI = {10.1090/gsm/105},
       URL = {https://doi-org.tudelft.idm.oclc.org/10.1090/gsm/105},
}

@article {Fusco2025,
    AUTHOR = {Fusco, Giuliana},
     TITLE = {Variational analysis of discrete {D}irichlet problems in
              periodically perforated domains},
   JOURNAL = {ESAIM Control Optim. Calc. Var.},
  FJOURNAL = {ESAIM. Control, Optimisation and Calculus of Variations},
    VOLUME = {31},
      YEAR = {2025},
     PAGES = {Paper No. 99, 28},
      ISSN = {1292-8119,1262-3377},
   MRCLASS = {49J45 (35B27 74Q05 82B20)},
  MRNUMBER = {5006155},
       DOI = {10.1051/cocv/2025085},
       URL = {https://doi-org.tudelft.idm.oclc.org/10.1051/cocv/2025085},
}

@book {Evans1992,
    AUTHOR = {Evans, Lawrence C. and Gariepy, Ronald F.},
     TITLE = {Measure theory and fine properties of functions},
    SERIES = {Studies in Advanced Mathematics},
 PUBLISHER = {CRC Press, Boca Raton, FL},
      YEAR = {1992},
     PAGES = {viii+268},
      ISBN = {0-8493-7157-0},
   MRCLASS = {28-02 (26-02 26Bxx 46E35)},
  MRNUMBER = {1158660},
MRREVIEWER = {R.\ G.\ Bartle},
}

@book {Lee2013,
    AUTHOR = {Lee, John M.},
     TITLE = {Introduction to smooth manifolds},
    SERIES = {Graduate Texts in Mathematics},
    VOLUME = {218},
   EDITION = {Second},
 PUBLISHER = {Springer, New York},
      YEAR = {2013},
     PAGES = {xvi+708},
      ISBN = {978-1-4419-9981-8},
   MRCLASS = {58-01 (53-01 57-01)},
  MRNUMBER = {2954043},
}

@article {Alberti1998,
    AUTHOR = {Alberti, Giovanni and Bellettini, Giovanni},
     TITLE = {A non-local anisotropic model for phase transitions:
              asymptotic behaviour of rescaled energies},
   JOURNAL = {European J. Appl. Math.},
  FJOURNAL = {European Journal of Applied Mathematics},
    VOLUME = {9},
      YEAR = {1998},
    NUMBER = {3},
     PAGES = {261--284},
      ISSN = {0956-7925,1469-4425},
   MRCLASS = {80A22 (49S05 73B40 73V25 82B26)},
  MRNUMBER = {1634336},
MRREVIEWER = {Matthias\ Wilhelm\ Winter},
       DOI = {10.1017/S0956792598003453},
       URL = {https://doi-org.tudelft.idm.oclc.org/10.1017/S0956792598003453},
}

@article {Clarke1976,
    AUTHOR = {Clarke, F. H.},
     TITLE = {On the inverse function theorem},
   JOURNAL = {Pacific J. Math.},
  FJOURNAL = {Pacific Journal of Mathematics},
    VOLUME = {64},
      YEAR = {1976},
    NUMBER = {1},
     PAGES = {97--102},
      ISSN = {0030-8730,1945-5844},
   MRCLASS = {26A57},
  MRNUMBER = {425047},
MRREVIEWER = {B.\ Rodr\'iguez-Salinas},
       URL = {http://projecteuclid.org.tudelft.idm.oclc.org/euclid.pjm/1102867214},
}

@article {Licht2019,
    AUTHOR = {Licht, Martin W.},
     TITLE = {Smoothed projections over weakly {L}ipschitz domains},
   JOURNAL = {Math. Comp.},
  FJOURNAL = {Mathematics of Computation},
    VOLUME = {88},
      YEAR = {2019},
    NUMBER = {315},
     PAGES = {179--210},
      ISSN = {0025-5718,1088-6842},
   MRCLASS = {65N30 (58A12)},
  MRNUMBER = {3854055},
MRREVIEWER = {V\'it\ Dolej\v s\'i},
       DOI = {10.1090/mcom/3329},
       URL = {https://doi-org.tudelft.idm.oclc.org/10.1090/mcom/3329},
}

@book {Grisvard1985,
    AUTHOR = {Grisvard, P.},
     TITLE = {Elliptic problems in nonsmooth domains},
    SERIES = {Monographs and Studies in Mathematics},
    VOLUME = {24},
 PUBLISHER = {Pitman (Advanced Publishing Program), Boston, MA},
      YEAR = {1985},
     PAGES = {xiv+410},
      ISBN = {0-273-08647-2},
   MRCLASS = {35J25 (35-02)},
  MRNUMBER = {775683},
MRREVIEWER = {P.\ Szeptycki},
}

@incollection {Gupta1999,
    AUTHOR = {Gupta, Piyush and Kumar, P. R.},
     TITLE = {Critical power for asymptotic connectivity in wireless
              networks},
 BOOKTITLE = {Stochastic analysis, control, optimization and applications},
    SERIES = {Systems Control Found. Appl.},
     PAGES = {547--566},
 PUBLISHER = {Birkh\"auser Boston, Boston, MA},
      YEAR = {1999},
      ISBN = {0-8176-4078-9},
   MRCLASS = {93C05},
  MRNUMBER = {1702981},
}

@book {Penrose2003,
    AUTHOR = {Penrose, Mathew},
     TITLE = {Random geometric graphs},
    SERIES = {Oxford Studies in Probability},
    VOLUME = {5},
 PUBLISHER = {Oxford University Press, Oxford},
      YEAR = {2003},
     PAGES = {xiv+330},
      ISBN = {0-19-850626-0},
   MRCLASS = {60-02 (05C80 60D05)},
  MRNUMBER = {1986198},
MRREVIEWER = {Ilya\ S.\ Molchanov},
       DOI = {10.1093/acprof:oso/9780198506263.001.0001},
       URL = {https://doi-org.tudelft.idm.oclc.org/10.1093/acprof:oso/9780198506263.001.0001},
}

@misc{trillos2025,
      title={Minimax Rates for the Estimation of Eigenpairs of Weighted {L}aplace-{B}eltrami Operators on Manifolds}, 
      author={Nicolás García Trillos and Chenghui Li and Raghavendra Venkatraman},
      year={2025},
      eprint={2506.00171},
      archivePrefix={arXiv},
      primaryClass={stat.ML},
      url={https://arxiv.org/abs/2506.00171}, 
      publisher = {arXiv},
    copyright = {arXiv.org perpetual, non-exclusive license},
    howpublished = {arxiv.org/abs/2506.00171 [stat.ML]}
}

@article {Braides2023,
    AUTHOR = {Braides, Andrea and Caroccia, Marco},
     TITLE = {Asymptotic behavior of the {D}irichlet energy on {P}oisson
              point clouds},
   JOURNAL = {J. Nonlinear Sci.},
  FJOURNAL = {Journal of Nonlinear Science},
    VOLUME = {33},
      YEAR = {2023},
    NUMBER = {5},
     PAGES = {Paper No. 80, 57},
      ISSN = {0938-8974,1432-1467},
   MRCLASS = {60G55 (35B27 49J45 82B43)},
  MRNUMBER = {4617151},
MRREVIEWER = {Christian\ Hirsch},
       DOI = {10.1007/s00332-023-09937-7},
       URL = {https://doi-org.tudelft.idm.oclc.org/10.1007/s00332-023-09937-7},
}

@phdthesis{methesis,
  author    = {Mercer, Samuel},
  title     = {Novel approaches to the study of gradient flows: discrete-to-continuum limits and beyond},
  school    = {Technische Universiteit Delft},
  year      = {in prep.},
  type      = {Ph{D} Thesis},
  address   = {Delft, Netherlands}
}

@article {Gennip2012,
    AUTHOR = {van Gennip, Yves and Bertozzi, Andrea L.},
     TITLE = {{$\Gamma$}-convergence of graph {G}inzburg-{L}andau
              functionals},
   JOURNAL = {Adv. Differential Equations},
  FJOURNAL = {Advances in Differential Equations},
    VOLUME = {17},
      YEAR = {2012},
    NUMBER = {11-12},
     PAGES = {1115--1180},
      ISSN = {1079-9389},
   MRCLASS = {35R02 (35B27 35Q56 49J45)},
  MRNUMBER = {3013414},
MRREVIEWER = {Anneliese\ Defranceschi},
}

\end{document}